\documentclass[11pt,a4paper,twoside]{article}
\usepackage[utf8]{inputenc}
\usepackage[english]{babel}
\usepackage[T1]{fontenc}
\usepackage{hyperref}
\usepackage{fancyhdr}
\usepackage{amsmath,amsfonts,amssymb,amsthm,mathrsfs,amscd}
\usepackage{mathtools}
\usepackage{graphicx, import}
\usepackage{color}
\usepackage[normal]{caption}
\usepackage{pdfpages}
\usepackage[left=2cm,right=2cm,top=2cm,bottom=2cm]{geometry}
\usepackage[all,cmtip]{xy}
\usepackage{float}
\usepackage{array}
\usepackage{textcomp}
\usepackage{gensymb}
\usepackage{mathtools}
\usepackage{caption}
\usepackage{bigints}
\usepackage{indentfirst}
\usepackage{enumerate}
\usepackage{tikz-cd}

\usepackage{comment}
\newcommand{\R}{\mathbb{R}}
\newcommand{\C}{\mathbb{C}}
\newcommand{\Z}{\mathbb{Z}}
\newcommand{\N}{\mathbb{N}}
\newcommand{\K}{\mathbb{K}}
\newcommand{\droit}{\mathrm{d}}

\newcommand{\liet}{\mathfrak{t}}
\newcommand{\liek}{\mathfrak{k}}

\newcommand{\F}{\mathcal{F}}
\newcommand{\G}{\mathcal{G}}
 
\newtheorem{remarksec}{Remark}[section]
\newtheorem{remarksubsec}{Remark}[subsection]

\newtheorem{defsubsec}{Definition}[subsection]
\newtheorem*{conjecture}{Conjecture}
\newtheorem{propsubsec}{Proposition}[subsection]
\newtheorem{propsec}{Proposition}[section]
\newtheorem{theoremsubsec}{Theorem}[subsection]

\newtheorem{lemmasubsec}{Lemma}[subsection]

\newtheorem{corsubsec}{Corollary}[subsection]

\newtheorem{exsec}{Example}[section]

\DeclareMathOperator{\Spec}{Spec}
\DeclareMathOperator{\pt}{pt}

\DeclareMathOperator{\Hcal}{\mathcal{H}}

\begin{document}
\floatplacement{figure}{H}

% Path for figures
% \graphicspath{{Figures/}}

%\begin{titlepage}
\begin{centering}
{\huge Translated points for contactomorphisms of prequantization spaces over monotone symplectic toric manifolds}\\
\vspace{0.8cm}
{\Large Brian \textsc{Tervil}}\\
\vspace{0.2cm}
{\large \textsc{University of Haifa}}\\
\end{centering}
%\end{titlepage}

\addcontentsline{toc}{section}{References}

\vspace{0.7cm}
\begin{centering}
\section*{Abstract}
\end{centering}

We prove a version of Sandon's conjecture on the number of translated points of contactomorphisms for the case of prequantization bundles over certain closed monotone symplectic toric manifolds. Namely we show that any contactomorphism of such a prequantization bundle lying in the identity component of the contactomorphism group possesses at least $N$ translated points, where $N$ is the minimal Chern number of the symplectic toric manifold. The proof relies on the theory of generating functions coupled with equivariant cohomology, whereby we adapt Givental's approach to the Arnold conjecture for integral symplectic toric manifolds to the context of prequantization bundles.

\vspace{1cm}

\section{Introduction and result}{\label{sec1}}

\subsection{The main result}{\label{sec1.1}}

A major driving force in symplectic topology is the celebrated Arnold conjecture \cite{Arn65}:
\begin{center}
\textit{The number of fixed points of a Hamiltonian symplectomorphism of a closed symplectic manifold is at least the minimal number of critical points of a smooth function.}
\end{center}
While in general diffeomorphisms and even symplectomorphisms have far fewer fixed points, it has been proved in full generality in a homological version using Floer homology (see for instance \cite{Flo89a}, \cite{HS95}, \cite{Ono95}, \cite{LT98}, \cite{FO99}): \textit{on any closed symplectic manifold $(M,\omega)$, \emph{non-degenerate} Hamiltonian symplectomorphisms have at least $\dim H^*(M; \mathbb{Q})$ fixed points}.\footnote{A fixed point $x \in M$ of a symplectomorphism $\phi$ is called \textit{non-degenerate} if $\text{det}(\droit_x \phi - Id_M) \neq 0$, or equivalently if the graph of $\phi$ is transversal to the diagonal in $M \times M$ at the point $x$. A symplectomorphism is called non-degenerate if it is non-degenerate at all its fixed points.} For general, not necessarily non-degenerate Hamiltonian symplectomorphisms, several estimates have been obtained, by Oh \cite{Ono95}, Schwarz \cite{Sch98}, and Givental \cite{Giv95}. The present paper is closer in spirit to the latter results.

The analogue of the Arnold conjecture in contact topology was introduced by S.\ Sandon \cite{San13}, through the notion of \textit{translated points}. Recall that a cooriented \textbf{contact manifold} \footnote{In this paper we will simply say \textit{contact manifold}.} is a pair $(V,\xi)$, where $V$ is an odd-dimensional manifold, and $\xi$ is a maximally non-integrable cooriented hyperplane field, called a \textbf{contact structure}. A \textbf{contactomorphism} of $(V,\xi)$ is a diffeomorphism preserving $\xi$ and its coorientation. In order to define the notion of translated points, fix a \textbf{contact form} $\alpha$ for $\xi$, that is a $1$-form such that $\xi = \ker \alpha$. Note that the maximal non-integrability of $\xi$ is equivalent to the non-degeneracy of the restriction $\droit \alpha_{|\xi}$ of $\droit \alpha$ to $\xi$. The \textbf{Reeb vector field} $R_{\alpha}$ of $\alpha$ is defined by 
$$
\alpha(R_{\alpha}) = 1 \quad \text{ and } \quad \iota_{R_{\alpha}} \droit \alpha = 0,
$$
and its flow is denoted by $\{\phi_{\alpha}^t\}_{t \in \R}$. Given a contactomorphism $\phi$, a point $x \in V$ is called an \textbf{$\alpha$-translated point} of $\phi$ if $x$ and $\phi(x)$ belong to the same Reeb orbit and if moreover $\phi$ preserves the contact form $\alpha$ at $x$:
$$
\exists \ s \in \R \quad \text{such that} \quad \phi(x) = \phi_{\alpha}^s(x) \quad \text{and} \quad (\phi^* \alpha)_x = \alpha_x.
$$
The equations defining the Reeb vector field can be generalized: given a contact manifold $(V, \xi)$ and a contact form $\alpha$, any \textbf{contact Hamiltonian} $h : V \times [0,1] \to \R$, that is a time-dependent function on $V$, gives rise to a unique time-dependent vector field $X_h^t$ satisfying
$$
\alpha(X_h^t) = h_t \quad \text{and} \quad \droit \alpha(X_h^t,.) = -\droit h_t + \droit h_t(R_{\alpha}) \alpha, \quad h_t := h(.,t).
$$
The vector field $\{X_h^t\}_{t \in [0,1]}$ preserves $\xi$ and, at least when $V$ is compact, its flow is defined for all $t \in [0,1]$, and gives rise to a contact isotopy $\{\phi_h^t\}_{t \in [0,1]}$. This procedure defines a bijection, depending on the contact form $\alpha$, between contact Hamiltonians and contact isotopies of $V$.\footnote{This is sharp in contrast to the symplectic setting: here the group of "Hamiltonian" contactomorphisms is the whole identity component of the group of contactomorphisms.} One motivation for the introduction of translated points as a contact analogue of fixed points of Hamiltonian symplectomorphisms is that contactomorphisms, even those obtained as time-$1$ maps of contact isotopies, may not have any fixed points. For instance, the Reeb flow $\{\phi_{\alpha}^t\}_{t \in \R}$ is an example of contact isotopy (with constant contact Hamiltonian equal to $1$), and since $R_{\alpha}$ never vanishes, the latter does not have any fixed points for small times.\footnote{Fixed and translated points are in fact two particular instances of a broader notion, called \textit{leafwise intersection} (see for instance \cite{San11}).} 

We let $\text{Cont}(V,\xi)$ be the group of contactomorphisms of $(V,\xi)$ and $\text{Cont}_0(V,\xi)$ its identity component. Sandon's conjecture is as follows:

\begin{conjecture}[\cite{San13}, Conjecture 1.2]
Let $(V,\xi)$ be a closed contact manifold, and $\phi \in \normalfont{\text{Cont}_0(V,\xi)}$. For any choice of contact form $\alpha$ for $\xi$, the number of $\alpha$-translated points of $\phi$ is at least the minimum number of critical points of a function on $V$.
\end{conjecture}

For any manifold $M$, we denote by $H^2_b(M; \Z) \subset H^2(M; \R)$ the image of the natural homomorphism $\rho : H^2(M; \Z) \to H^2(M; \R)$. A symplectic manifold $(M,\omega)$ is called \textbf{monotone} if the cohomology class of its symplectic form is \textit{positively proportional to $\rho(c_1)$, where $c_1$ is the first Chern class of $(M, \omega)$}.\footnote{Recall that given a symplectic manifold $(M,\omega)$, the set of almost complex structures $J$ on the tangent bundle $TM$ that are compatible with $\omega$, that is such that the $2$-form $\omega(.,J.)$ is a Riemannian metric, is non-empty and contractible, so that the first Chern class $c_1 := c_1(TM,J)$ is independent of the choice of compatible almost complex structure used to define it (see for instance \cite{MS17}).} We denote by $N_M$ the \textbf{minimal Chern number of $(M, \omega)$}, that is the positive generator of $\langle c_1, H_2(M; \Z) \rangle \subset \Z$. A \textbf{prequantization space} over a symplectic manifold $(M,\omega)$ is a contact manifold $(V, \xi := \ker \alpha)$ along with a principal $S^1$-bundle $\pi : (V,\alpha) \to (M,\omega)$, such that $\pi^* \omega = \droit \alpha$ and the Reeb vector field $R_{\alpha}$ induces the free $S^1$-action on $V$, where $S^1 = \R / \hbar \Z$, $\hbar > 0$ being the minimal period of a closed Reeb orbit. Note that a symplectic manifold $(M, \omega)$ is the base of a prequantization space if and only if there exists $r > 0$ such that $\omega / r$ is \textbf{integral}, that is $[\omega / r] \in H^2_b(M; \Z)$, in which case the image $\rho(eu(\pi))$ of the Euler class $eu(\pi)$ of $\pi$ is given by $-\frac{1}{\hbar} [\omega / r]$ (see for instance \cite[section 6.2 (d)]{Mor01} and then \cite{BW58}). We call a symplectic manifold $(M, \omega)$ \textbf{integral} if its symplectic form $\omega$ is integral, and we call the latter \textbf{primitive} if $\langle [\omega], H_2(M; \Z) \rangle = \Z$. Note that the symplectic form of a monotone symplectic manifold can always be rescaled so that it becomes integral, and that an integral symplectic form can always be rescaled so that it becomes primitive. A symplectic \textbf{toric} manifold $(M^{2d}, \omega, \mathbb{T})$ is a symplectic manifold endowed with an effective Hamiltonian action of a torus $\mathbb{T}$ of dimension $d$. The second cohomology group of a symplectic toric manifold being torsion-free, $\rho : H^2(M; \Z) \to H^2(M; \R)$ is injective and, for any $\hbar > 0$, if it exists, the prequantization space $(V, \xi := \ker \alpha)$ over $(M, \omega)$ with Euler class $- \frac{1}{\hbar} [\omega]$ is \textit{unique}, up to $\R / \hbar \Z$-bundle isomorphism.\footnote{The image of the Euler class in $H^2_b(M; \Z)$ determines a prequantization space over $(M, \omega)$, seen as a line bundle with connection $1$-form, only up to its tensor product with a line bundle admitting a \textit{flat} connection $1$-form, which corresponds to a torsion class in $H^2(M; \Z)$. We refer the reader, for instance, to \cite[section 8]{Woo97} for further details.} Furthermore, $\omega$ is primitive if and only if $\hbar = 1$.

\begin{exsec}
\normalfont The complex projective space $\C P^{n-1}$ endowed with the Fubini-Study $\omega_{\text{FS}}$ form is naturally an integral symplectic toric manifold for the standard action of $\mathbb{T}^n / S^1$, where $\mathbb{T}^n := \R^n / \Z^n$, and $S^1 \subset \mathbb{T}^n$ is the subtorus consisting of diagonal elements. Examples of prequantization spaces over $(\C P^{n-1}, \omega_{\text{FS}})$ include the standard $(2n-1)$-dimensional contact sphere $S^{2n-1}$ and real projective space $\R P^{2n-1}$. Note that $\omega_{\text{FS}}$ is monotone and primitive.
\end{exsec}

Our main result is the following.

\begin{theoremsubsec}{\label{theo1.1.1}}
Let $(M, \omega, \mathbb{T})$ be a closed monotone symplectic toric manifold with primitive symplectic form. Assume that it is different from $(\C P^{n-1}, \omega_{\text{FS}}, \mathbb{T}^n / S^1)$, and let $(V, \xi := \ker \alpha)$ be the prequantization space over $(M, \omega)$ with Euler class $- [\omega]$. Then any $\phi \in \normalfont{\text{Cont}_0}(V,\xi)$ has at least $N_M$ $\alpha$-translated points.
\end{theoremsubsec}

Some Morse and cuplength estimates in direction of Sandon's conjecture were previously established for the standard contact forms of prequantization spaces over $(\C P^{n-1}, \omega_{\text{FS}})$: for the standard contact sphere $S^{2n-1}$ and real projective space $\R P^{2n-1}$ in \cite{San13}, and later for lens spaces in \cite{GKPS17}. We will see that in this setting our arguments do not apply. The existence of translated points was also obtained in other contexts, for instance in \cite{AM13}, \cite{She17}, and \cite{MN18}. All the aforementioned works, except \cite{San13} and \cite{GKPS17}, are based on Floer-type constructions which rely among others on a non-degeneracy assumption for the contactomorphisms, analoguous to that of the symplectic framework (see \cite{San13}). In \cite{San13} and \cite{GKPS17}, and here as well, the technique used is that of generating functions, which allows to tackle the conjecture in a more general setting. Namely, our result holds for \textit{any} contactomorphism of the identity component $\text{Cont}_0(V,\xi)$ of the group of contactomorphisms. Note also that it happens very often that $N_M \geq 2$ (for instance if $M = \C P^1 \times \C P^1$ endowed with the sum of the Fubini-Study forms, $N_M = 2$). This is in sharp contrast to Floer-type constructions, with which one can in general prove the existence of only one translated point without the non-degeneracy assumption. On the other hand, our theorem holds in the case where the symplectic form $\omega$ of the monotone symplectic toric manifold $(M, \omega, \mathbb{T})$ is \textit{primitive}. Let $k \geq 1$ be an integer. Then the principal $\R / \frac{1}{k} \Z$-bundle with Euler class $-k[\omega]$ is of the form $\pi_k : V / \Z_k \to M$, where $\Z_k \subset \R / \Z$ denotes the subgroup of $\R / \Z$ of $k$-th roots of unity, and $\pi_k$ is defined by $\pi_k \circ pr = \pi$, where $pr : V \to V / \Z_k$ is the canonical projection. The contact form $\alpha$ on $V$ induces a well-defined contact form $\alpha_k$ on $V / \Z_k$, defined by $pr^* \alpha_k = \alpha$, satisfying $\droit \alpha_k = \pi_k^* \omega$, and pulling back contact Hamiltonians through the projection $pr$, it is then easy to see that, given a contact Hamiltonian $\phi_h \in \text{Cont}_0(V / \Z_k, \xi_k)$, an $\alpha$-translated point of $\phi_{pr^*h} \in \text{Cont}_0(V, \xi)$ projects to an $\alpha_k$-translated point of $\phi_h$. Note however that nothing prevents two geometrically different $\alpha$-translated points of $\phi_{pr^*h}$ to lie on the same $\Z_k$-orbit. With the notations of the discussion above, we thus have the following.

\begin{corsubsec}{\label{cor1.1.1}}
For any integer $k \geq 1$, the prequantization space $(V / \Z_k, \xi_k := \ker \alpha_k)$ over $(M, \omega, \mathbb{T})$ with Euler class $- k[\omega]$ is such that any $\phi \in \normalfont{\text{Cont}_0}(V / \Z_k, \xi_k)$ has at least one $\alpha_k$-translated point.
\end{corsubsec}

\begin{remarksubsec}
\normalfont The minimal Chern number $N_M$ of a monotone symplectic toric manifold $(M, \omega, \mathbb{T})$ is always strictly smaller than its cuplength $cl(M) = \dim_{\C} M + 1$, unless $(M, \omega) = (\C P^{n-1}, \omega_{\text{SF}})$, in which case both quantities equal $n$.\footnote{The cuplength of a manifold $M$ is the smallest positive integer $k$ such that the cup product of any $k$ cohomology classes $a_1, \hdots , a_k \in H^*(M; \Z)$ of positive degrees vanishes: $a_1 \cup \hdots  \cup a_k = 0$.} Moreover, it is straightforward from the definition of the cuplength that $cl(M) \leq cl(V)$. By Lusternik-Schnirelmann theory \cite{Cor03}, we have $cl(V) \leq \text{Crit}(V)$, where $\text{Crit}(V)$ is the minimal number of critical points of a function on $V$. Therefore, our lower bound in Theorem \ref{theo1.1.1} is stricly smaller than expected from Sandon's conjecture. As in the symplectic setting, Sandon's conjecture could be rephrased in a homological version, in which the minimal number of critical points of a function would be replaced with the sum $\dim H^*(V; \mathbb{Q})$ of the Betti numbers of $V$ when the contactomorphism is non-degenerate, and with $cl(V)$ in the general case. Still, our lower bound $N_M$ is smaller than $cl(V)$. It is perhaps not surprising that our result is weaker than expected, even in the homological version. This discrepancy can already be observed in the symplectic setting: it first appeared in Y.-G.\ Oh's paper \cite{Oh90} on the symplectic product $\mathbb{T}^{2k} \times \C P^{n-1}$ of the torus and the complex projective space with standard symplectic structure, where the lower bound was found to be $\max(2k+1, n)$, whereas the cuplength is equal to $2k+n$. Givental's theorem \cite{Giv95} applied to monotone symplectic toric manifolds gives another example of this kind: in this case it is equal to the minimal Chern number of the symplectic manifold as well. 
\end{remarksubsec}

Our approach to Theorem \ref{theo1.1.1} is based on the theory of \textit{generating functions} and \textit{equivariant cohomology}, as developed by A.\ Givental \cite{Giv95}. In the next section, we will give an overview the constructions and of the proof of Theorem \ref{theo1.1.1}. 

\subsection{Overview of the paper and proof of the theorem}{\label{sec1.2}}

We describe here the main steps of our constructions, and give a proof of Theorem \ref{theo1.1.1}, assuming several results which will be discussed in the sequel. Our purpose here is to provide a recipe of the paper, so that the reader can have a general insight of the presented arguments. The technical details are treated in the following sections.

Generating functions were extensively used in the eighties and nineties by numerous authors (see for instance \cite{Chap84}, \cite{LS85}, \cite{Giv90}, \cite{Vit92}, \cite{Giv95}, \cite{The95}, \cite{The98}). They provide a powerful tool when the manifold can be obtained somehow from a symplectic vector space. In \cite{Giv95}, Givental used this approach along with equivariant cohomology to establish a version of the Arnold conjecture for integral symplectic toric manifolds.

Below are the main lines of our construction. The reader shall notice that although Theorem \ref{theo1.1.1} is proved over \textit{monotone} symplectic toric manifolds with \textit{primitive} symplectic form, our constructions hold for \textit{any} closed \textit{integral} symplectic toric manifold. The monotonicity and the primitivity assumptions are only required for the proofs of Proposition \ref{prop1.2.3} and Theorem \ref{theo1.1.1}. 

\subsubsection*{The generating functions and the cohomology groups}

The main construction of this paper is adapted from \cite{Giv95} to the contact setting. Given a closed integral symplectic toric manifold $(M,\omega, \mathbb{T})$, we first construct a natural prequantization space over it, mainly following a procedure of Borman and Zapolsky \cite{BZ15}. Let $\mathbb{T}^n := \R^n / \Z^n$ denote the maximal torus acting on the standard symplectic Euclidean space $(\C^n, \omega_{\text{std}})$, where $\omega_{\text{std}} := \underset{j=1}{\overset{n} \sum} \droit x_j \wedge \droit y_j$, by rotation on each coordinate: 
$$
\mathbb{T}^n \times \C^n \to \C^n, \quad ((\lambda_1, \hdots, \lambda_n),(z_1, \hdots, z_n)) \mapsto (e^{2i\pi \lambda_1} z_1, \hdots, e^{2i\pi \lambda_n} z_n).
$$
A very convenient way of viewing $(M, \omega, \mathbb{T})$ comes from Delzant's theorem \cite{Del88}, which states that it can be obtained as a symplectic reduction of $(\C^n, \omega_{\text{std}})$, for some $n \in \N$, by the action of a subtorus $\K \subset \mathbb{T}^n$. This identification provides in particular isomorphisms
\begin{equation}{\label{eq1}}
\begin{tikzcd}
\liek^*  \ar[r, "\sim"] & H^2(M; \R) & & \liek  \ar[r, "\sim"] & H_2(M; \R)\\
\liek_{\Z}^* \ar[u,hook]  \ar[r, "\sim"] & H^2(M; \Z) \ar[u,hook] & & \liek_{\Z} \ar[u,hook]  \ar[r, "\sim"] & H_2(M; \Z) \ar[u,hook],
\end{tikzcd}
\end{equation}
where $\liek := \text{Lie}(\K)$ denotes the Lie algebra of $\K$, $\liek_{\Z} := \ker (\exp : \liek \to \K)$ is the kernel of the exponential map, and $\liek^*$, $\liek_{\Z}^*$ are their respective duals. The vertical homomorphisms are the natural ones, and are injections since symplectic toric manifolds have no torsion elements (we refer to \cite[chapter $7.3$]{Aud12} for more details). One can then identify $[\omega] \in H^2(M; \R)$ with an element $p \in \liek^*$. Along with the momentum map $P_{\K} : \C^n \to \liek^*$ associated with the $\K$-action, symplectic reduction yields an identification
$$
(M, \omega) \simeq (P_{\K}^{-1}(p) / \K, \overline{\omega}_{\text{std} | P_{\K}^{-1}(p)}),
$$
where $\overline{\omega}_{\text{std} | P_{\K}^{-1}(p)}$ is the symplectic form induced by the restriction of $\omega_{\text{std}}$ to $P_{\K}^{-1}(p)$. The integrality assumption implies that $p \in \liek_{\Z}^*$. In particular, its kernel is spanned by the intersection $\liek_{\Z} \cap \ker p$. Hence, applying the exponential map to $\ker p$ yields a codimension $1$ subtorus $\K_0 \subset \K$, having Lie algebra $\liek_0 := \ker p$. We will see that $P_{\K}^{-1}(p)$ is in fact a subset of a contact sphere $(S_p, \xi_{S_p} := \ker \alpha_{\text{std} | S_p}) \subset \C^n$, where $\alpha_{\text{std}} = \frac{1}{2} \underset{j=1}{\overset{n} \sum} (x_j \droit y_j - y_j \droit x_j)$, and that it is moreover the zero level set of the contact momentum map associated with the $\K_0$-action on $S_p$. By contact reduction, this yields a contact manifold
$$
(V, \xi := \ker \alpha) \simeq (P_{\K}^{-1}(p) / \K_0, \xi := \ker \overline{\alpha}_{\text{std}|S_p}),
$$
where $\overline{\alpha}_{\text{std}|S_p}$ is the contact form induced by the restriction of $\alpha_{\text{std}}$ to $S_p$. Under the above identifications, we obtain a natural prequantization bundle
$$
\pi : (V, \alpha) \overset{S^1}{\longrightarrow} (M, \omega),
$$
with fiber given by $S^1 := \K / \K_0$. This procedure is carried out in section \ref{sec3.1}.

The second step is a lifting procedure, which allows us to translate the search for $\alpha$-translated points on $V$ to that of fixed points on $\C^n$. We begin with a contact Hamiltonian $h : V \times [0,1] \to \R$. Lifting it to the contact sphere $(S_p, \xi_{S_p} = \ker \alpha_{\text{std}|S_p})$, and then to its symplectization $(\C^n \setminus \{0\}, \omega_{\text{std}|\C^n \setminus \{0\}})$, we obtain a Hamiltonian $H : \C^n \times [0,1] \to \R$, which is homogeneous of degree $2$ with respect to the standard $\R_{>0}$-action on $\C^n$, and $\K_0$-invariant. Let us denote by $\phi_h := \phi_h^1$ the time-$1$ map of the contact isotopy $\{\phi_h^t\}_{t \in [0,1]}$ of $V$, and by $\phi_H := \phi_H^1$ the time-$1$ map of the Hamiltonian isotopy $\{\phi_H^t\}_{t \in [0,1]}$ of $\C^n$ generated by $H$. Note that $\phi_H$ is $\K_0$-equivariant. The key observation is that $\alpha$-translated points of $\phi_h$ correspond to \textit{certain} fixed points of the following family of Hamiltonian symplectomorphisms:
$$
\exp(\lambda) \circ \phi_H,
$$
where $\lambda \in \liek$ varies in the Lie algebra of $\K$.\footnote{The reader shall notice the following crucial difference between our setting and that of Givental: in \cite{Giv95}, a Hamiltonian symplectomorphism of $(P_{\K}^{-1}(p) / \K, \overline{\omega}_{\text{std} | P_{\K}^{-1}(p)})$ is lifted up to a Hamiltonian symplectomorphism $\phi_H$ of $\C^n$ which is \textit{$\K$-equivariant}, and thus Givental looks for fixed points of the $\phi_H$ up to the $\K$-action on $\C^n$, that is fixed points of compositions $\exp(\lambda) \circ \Phi_H$, but there \textit{the latter are $\K$-equivariant}. In our setting, the Hamiltonian symplectomorphism $\phi_H$, and therefore the compositions $\exp(\lambda) \circ \phi_H$, are \textit{$\K_0$-equivariant}. However, instead of looking for fixed points of $\phi_H$ up to the $\K_0$-action on $\C^n$ (by making $\lambda$ vary in $\liek_0$), we must still look for fixed points \textit{up to the $\K$-action}. The reason is that the $\K$-action projects to $\K / \K_0$ on $V$, that is to the fiber of the prequantization bundle $\pi$, which is generated by the Reeb vector field of $\alpha$. In other words, we keep track of the fact that we are looking for $\alpha$-translated points of $\phi_h$, rather than fixed points. We are thus brought to considering the same family of Hamiltonian symplectomorphisms as in \cite{Giv95}, but with a different symmetry group.} This is carried out in section \ref{sec3.2}.

The third step is the construction of a generating family. In section \ref{sec3.3}, following \cite{Giv95}, we decompose the Hamiltonian isotopies $\{\phi_H^t\}_{t \in [0,1]}$ and $\{\exp(t \lambda)\}_{t \in [0,1]}$ into $2N_1$ and $2N_2$ Hamiltonian symplectomorphisms respectively  
$$
\phi_H = \phi_{2N_1} \circ \cdots \circ \phi_1 \quad \text{and} \quad \exp(\lambda) = \underset{2N_2-\text{times}}{\underbrace{\exp(\frac{\lambda}{2N_2}) \cdots \exp(\frac{\lambda}{2N_2})}},
$$
such that the graph of each symplectomorphism $\phi_j$ and $\exp(\frac{\lambda}{2N_2})$ projects diffeomorphically to the diagonal in $\C^n \times \C^n$. This allows us to define a family of \textbf{generating functions}
$$
\F_N : \C^{2nN} \times \Lambda_N \to \R, \quad \F_{\lambda}^{(N)} := \F_N(\cdot, \lambda), \quad N = N_1 + N_2,
$$
parametrized by compact subsets $\Lambda_N \subset \liek$ with boundary $\partial \Lambda_N$, such that the critical points of $\F_{\lambda}^{(N)}$ are in one-to-one correspondence with the fixed points of the Hamiltonian symplectomorphism
$$
-Id_{\C^n} \circ \exp(\lambda) \circ \phi_H. \footnote{The twist by $-Id_{\C^n}$ is technical, and will ensure the non-degeneracy of a certain quadratic form which will appear in the sequel.}
$$
The generating functions $\F_{\lambda}^{(N)}$ are homogeneous of degree $2$ with respect to the standard $\R_{>0}$-action on $\C^{2nN}$, and $\K_0$-invariant. This implies that critical points of $\F_{\lambda}^{(N)}$ appear as $\R_{>0}$-lines of $\K_0$-orbits, and moreover they all have critical value $0$. Throughout this text, we will use the term $(\R_{>0} \times \K_0)$-families to denote $\R_{>0}$-lines of $\K_0$-orbits, and the product action of $\R_{>0} \times \K_0$ on any invariant subspace of $\C^{2nN}$ will be the action induced by the linear multiplication of $\R_{>0}$ and the rotation of $\K_0$ on each factor $\C^n$.

The key observation now is that \textit{certain} $(\R_{>0} \times \K_0)$-familes of critical points of the functions $\F_{\lambda}^{(N)}$ correspond to $\alpha$-translated points of the composition $g \circ \phi_h$, where $g \in \text{Cont}_0(V,\xi)$ is the contactomorphism of $V$ induced by the restriction $-Id_{\C^n | P_{\K}^{-1}(p)}$. Since any estimate for the number of $\alpha$-translated points of all the compositions of $g$ with contactomorphisms of $\text{Cont}_0(V,\xi)$ will give rise to the same estimate for the contactomorphisms themselves, we are reduced to the search for critical points of the functions $\F_{\lambda}^{(N)}$. However, not all critical points will correspond to points on $V$. Therefore, the sublevel sets to consider are not the zero sublevel sets $\lbrace \F_{\lambda}^{(N)} \leq 0 \rbrace$, but rather must enclose the data defining the prequantization space, namely the cohomology class of $\omega$, that is $p$. Restricting the functions $\F_{\lambda}^{(N)}$ to the unit sphere $S_N \subset \C^{2nN}$ in order to take into account the $\R_{>0}$-invariance, the relevant subsets to investigate are of the form
\begin{equation}{\label{eq2}}
F_N^-(\nu) := \underset{\lambda \in \Lambda_N \cap p^{-1}(\nu)}{\bigcup} \lbrace \F_{\lambda}^{(N)} \leq 0 \rbrace \cap S^{4nN-1}, \quad \nu \in \R.
\end{equation}
These sets are $\K_0$-invariant, and one can study the following relative $\K_0$-equivariant cohomology groups, with complex coefficients:
$$
H_{\K_0}^*(F_N^-(\nu), \partial F_N^-(\nu)),
$$
where $\partial F_N^-(\nu)$ denotes the restriction of $F_N^-(\nu)$ to the boundary $\partial \Lambda_N$ of $\Lambda_N$. 

The final step in our construction is a limit process in $N \to \infty$. A large part of this paper (sections \ref{sec3.4} to \ref{sec3.6}) is devoted to finding a natural way of building a homomorphism
$$
H_{\K_0}^{*+2nN}(F_N^-(\nu), \partial F_N^-(\nu)) \to H_{\K_0}^{*+2nN'}(F_{N'}^-(\nu), \partial F_{N'}^-(\nu)),
$$
when $N \leq N'$. Note that a shift in the degree of the cohomology groups emerges naturally in this map. Under a genericity assumption on the real number $\nu$, we can apply a direct limit, and come to the definition of the following cohomology group:
$$
\Hcal_{\K_0}^*(F^-(\nu)) := \underset{N \to \infty}{\lim} H_{\K_0}^{*+2nN}(F_N^-(\nu), \partial F_N^-(\nu)).
$$
We call this limit the \textbf{cohomology of $H$ of level $\nu$}. It is the main character of this paper, and the remaining parts of the latter are devoted to its study.

\subsubsection*{Algebraic structures on the cohomology groups}

The limit $\Hcal_{\K_0}^*(F^-(\nu))$ is naturally associated with the Hamiltonian lift $H$ of $h$, and comes along with certain structures. First, note that the sets $F_N^-(\nu)$ from equation $(\ref{eq2})$ were defined by restricting the zero sublevel sets of the generating functions $\F_{\lambda}^{(N)}$ to the unit sphere $S^{4nN-1} \subset \C^{2nN}$, which takes account of the fact that critical points come in $\R_{>0}$-lines. Yet, the same limit process can be performed without this restriction. The sets
$$
\F_N^-(\nu) := \underset{\lambda \in \Lambda_N \cap p^{-1}(\nu)}{\bigcup} \lbrace \F_{\lambda}^{(N)} \leq 0 \rbrace, \quad \nu \in \R,
$$
contract onto $\Lambda_N \cap p^{-1}(\nu)$, and their restrictions $\partial \F_N^-(\nu)$ onto the boundary $\partial \Lambda_N$ of $\Lambda_N$ contract to $\partial \Lambda_N \cap p^{-1}(\nu)$. Therefore, in the limit $N \to \infty$, we obtain a cohomology group which is \textit{independent of $H$}. We will denote it by $\Hcal_{\K_0}^*(\nu)$.\footnote{The cohomology group $\Hcal_{\K_0}^*(\nu)$ is in fact also independent of $\nu$: indeed, the torus $\K_0$ acts trivially on the pair $(p^{-1}(\nu) \cap \Lambda_N, \partial \Lambda_N \cap p^{-1}(\nu))$ and therefore $H_{\K_0}^*(\Lambda_N \cap p^{-1}(\nu), \partial \Lambda_N \cap p^{-1}(\nu))$ is a free $H_{\K_0}^*(\pt)$-module of rank $1$ generated by the fundamental cocyle of the sphere $(\Lambda_N \cap p^{-1}(\nu)) / (\partial \Lambda_N \cap p^{-1}(\nu))$, whose dimension equals to $\dim \liek_0$, which is independent of $\nu$. We will also see another proof of this fact in section \ref{sec4}.}  The inclusion $(F_N^-(\nu), \partial F_N^-(\nu)) \subset (\F_N^-(\nu), \partial F_N^-(\nu))$ of pairs yields a natural homomorphism
$$
\Hcal_{\K_0}^*(\nu) \to \Hcal_{\K_0}^*(F^-(\nu)),
$$
called the \textbf{augmentation map}. Both groups are endowed with the structure of modules over the $\K_0$-equivariant cohomology of a point $H_{\K_0}^*(\pt)$, hence over the $\mathbb{T}^n$-equivariant cohomology of a point $H_{\mathbb{T}^n}^*(\pt)$ as well, via the natural homomorphism $H_{\mathbb{T}^n}^*(\pt) \to H_{\K_0}^*(\pt)$ induced by the inclusion $\K_0 \subset \mathbb{T}^n$. The augmentation map is a module homomorphism. Moreover, the Chern-Weil isomorphism (see \cite{BT13})
$$
H_{\mathbb{T}^n}^*(\pt) \simeq \C[u_1,\hdots,u_n]
$$
identifies the group $H_{\mathbb{T}^n}^*(\pt)$ with the ring of polynomials in $n$ variables $u_i$ of degree $2$, and $H_{\K_0}^*(\pt)$ with the quotient 
$$
H_{\K_0}^*(\pt) \simeq \C[u_1,\hdots,u_n] / I_0,
$$ where $I_0$ is the ideal generated by polynomials vanishing on the complexified Lie algebra $\liek_0 \otimes \C$. Equivalently, $H_{\mathbb{T}^n}^*(\pt)$ is the ring of regular functions on $\R^n \otimes \C$, where $\R^n$ is the Lie algebra of $\mathbb{T}^n$, and $H_{\K_0}^*(\pt)$ is that of regular functions on $\liek_0 \otimes \C$.

\subsubsection*{The Gysin sequence and the algebraic results}

A key novelty in this paper is the use of a so-called \textit{Gysin-type} long exact sequence for equivariant cohomology (section \ref{sec4}), relating our cohomology group $\Hcal_{\K_0}^*(F^-(\nu))$ to the one constructed by Givental in \cite{Giv95}, which is, contrary to our limit of $\K_0$-equivariant cohomology groups, obtained as a limit of $\K$-equivariant cohomology groups. Using this sequence, we will see that the cohomology group $\Hcal_{\K_0}^*(\nu)$ can in fact be viewed as the ring of regular functions on the intersection $(\liek_0 \otimes \C) \cap (\C^{\times})^n$, where $(\C^{\times})^n$ denotes the complex $n$-dimensional torus. If we denote this ring by $\mathcal{R}_0$, and by $\C[u,u^{-1}]$ the ring of Laurent polynomials in $n$-variables $u = (u_1,...,u_n)$, or equivalently the ring of regular functions on $(\C^{\times})^n$, we will prove that
$$
\Hcal_{\K_0}^*(\nu) \simeq \mathcal{R}_0 \simeq \C[u,u^{-1}] / I_0 \C[u,u^{-1}].
$$
In addition to this interpretation of $\Hcal_{\K_0}^*(\nu)$ and the aforementioned module structure, a natural isomorphism arises from the toric manifold. More precisely, for any $m \in H_2(M; \Z) \simeq \liek_{\Z}$, we have
$$
\Hcal_{\K_0}^*(F^-(\nu)) \simeq \Hcal_{\K_0}^{*+2c_1(m)}(F^-(\nu + p(m))),
$$
where $c_1$ is the first Chern class of $(M,\omega)$, identified with an element of $\liek_{\Z}^*$. The shifts by $2c_1(m)$ and $p(m)$ in the degree and the "level" of the cohomology group $\Hcal_{\K_0}^*(F^-(\nu))$ respectively, are the main ingredients in the proof of our theorem. Throughout this text, we will think of this isomorphism as a \textit{Novikov action} of $H_2(M; \Z)$, in analogy with filtered Floer homology (even though it is not a genuine action here). Similarly, $H_2(M; \Z)$ acts on the ring $\mathcal{R}_0$, which we recall is obtained through the same limit process as for the groups $\Hcal_{\K_0}^*(F^-(\nu))$, without restricting the sublevel sets of the generating functions to the unit sphere. Let $m = (m_1,...,m_n)$ denote the coordinates of $m$ through the inclusion $\liek_{\Z} \subset \R^n$. Then
$$
c_1(m) = \underset{j=1}{\overset{n} \sum}m_j.
$$
We can interpret the Novikov action on the kernel 
$$
\mathcal{J}_{\K_0}^*(F^-(\nu)) := \ker (\mathcal{R}_0 \to \Hcal_{\K_0}^*(F^-(\nu)))
$$
of the augmentation map. Of course, we first need to ensure that $\mathcal{R}_0$ is non-trivial, otherwise $\mathcal{J}_{\K_0}^*(F^-(\nu)) = \{0\}$. \textit{This is the case precisely when $\liek_0 \neq \{0\}$, that is when the symplectic toric manifold $(M, \omega, \mathbb{T})$ is different from $(\C P^{n-1}, \omega_{\normalfont{\text{FS}}}, \mathbb{T}^n / S^1)$, where $\omega_{\normalfont{\text{FS}}}$ is the Fubini-Study form, which we will assume from now on in this section}. We will see that, for any $m = (m_1,...,m_n) \in \liek_{\Z} \subset \R^n$, the action of $m$ on $\mathcal{J}_{\K_0}^*(F^-(\nu))$ is simply the multiplication by the image of the monomial $u_1^{m_1}...u_n^{m_n}$ through the quotient map $\C[u,u^{-1}] \to \mathcal{R}_0$:\footnote{Throughout this paper, we will identify a monomial of $\C[u,u^{-1}]$ with its image under the projection $\C[u,u^{-1}] \to \mathcal{R}_0$, for the sake of clarity in the notations.}
\begin{equation}{\label{eq3}}
u_1^{m_1}...u_n^{m_n} \mathcal{J}_{\K_0}^*(F^-(\nu)) \simeq \mathcal{J}_{\K_0}^{*+2c_1(m)}(F^-(\nu + p(m))).
\end{equation}

We now turn to two important properties of the augmentation map. Recall that given a contact manifold $(V, \xi)$, a contact form $\alpha$ and a contactomorphism $\phi$, the \textbf{translated spectrum} of $\phi$ with respect to $\alpha$ is the set
\begin{equation}{\label{eq4}}
\text{Spec}(\phi) := \lbrace s \in \R \ | \ \exists \ x \in V, x \text{ is an $\alpha$-translated point of $\phi$ with $\phi(x) = \phi_{\alpha}^s(x)$} \rbrace
\end{equation}
of \textit{Reeb shifts} (or \textit{Reeb time-shifts}) of $\alpha$-translated points of $\phi$.\footnote{In this paper, we will simply say translated spectrum of $\phi$, since the contact form $\alpha$ will be fixed.} Note that on a prequantization space, the translated spectrum of a contactomorphism is always periodic of period $\hbar$. Further in the paper, we will assume that $\hbar = 1$, which is equivalent to $p$ being a primitive integral vector in $\liek_{\Z}^*$. A key ingredient in our construction is that the generating functions $\F_{\lambda}^{(N)}$ are monotone in a certain direction in $\liek$. More precisely, they are \textit{decreasing in all directions on which $p$ is positive}. In particular, if $\nu_0 \leq \nu_1$ are two generic numbers, we have a natural homomorphism
$$
\Hcal_{\K_0}^*(F^-(\nu_1)) \to \Hcal_{\K_0}^*(F^-(\nu_0)),
$$
which commutes with the augmentation map. Recall that we have denoted by $g \in \text{Cont}_0(V, \xi)$ the contactomorphism of $V$ induced by the restriction $- Id_{\C^n | P_{\K}^{-1}(p)}$. The translated spectrum of $g \circ \phi_h$ is related to the groups $\Hcal_{\K_0}^*(F^-(\nu))$ by the two following analogues of \cite[Propositions $6.2$, $6.3$]{Giv95}:
\begin{propsubsec}{\label{prop1.2.1}}
Suppose that $[\nu_0,\nu_1] \cap \Spec(g \circ \phi_h) = \emptyset$. Then the homomorphism above is an isomorphism
$$\Hcal_{\K_0}^*(F^-(\nu_1)) \simeq \Hcal_{\K_0}^*(F^-(\nu_0)).$$
\end{propsubsec}

\begin{propsubsec}{\label{prop1.2.2}}
Suppose that the segment $[\nu_0,\nu_1]$ contains only one value $\nu \in \Spec(g \circ \phi)$, which corresponds to a finite number of $\alpha$-translated points. Let $v \in H_{\K_0}^*(\pt)$ be an element of positive degree, and $q \in \mathcal{R}_0$. Suppose that $q \in \mathcal{J}_{\K_0}^*(F^-(\nu_0))$. Then $v q \in \mathcal{J}_{\K_0}^*(F^-(\nu_1))$.
\end{propsubsec}
Proposition \ref{prop1.2.1} holds in the limit $N \to \infty$, however, we will see that, even if $N$ is big, it is \textit{not true} that given any two generic numbers $\nu_0 \leq \nu_1$ such that $[\nu_0, \nu_1] \cap \text{Spec}(g \circ \phi_h) = \emptyset$, the groups $H_{\K_0}^*(F_N^-(\nu_1), \partial F_N^-(\nu_1))$ and $H_{\K_0}^*(F_N^-(\nu_0), \partial F_N^-(\nu_0))$ are isomorphic. Similar statements as Proposition \ref{prop1.2.2} (in finite dimensional settings) are used for instance in \cite[Proposition $5.2$ (c)]{The98}, \cite[section $8$]{Giv90}, and \cite[Proposition $4.8$]{GKPS17}.
The second statement can be rephrased in terms of the following commutative diagram:
$$
\begin{tikzcd}
q \in \mathcal{R}_0 \ar[r] \ar[dr] & \Hcal_{\K_0}^*(F^-(\nu_1)) \ar[d] \ni q_1\\
& \Hcal_{\K_0}^*(F^-(\nu_0)) \ni q_0.
\end{tikzcd}
$$
Here, $q_1$ and $q_0$ are the images of an element $q \in \mathcal{R}_0$ by the augmentation maps $\mathcal{R}_0 \to \Hcal_{\K_0}^*(F^-(\nu_1))$ and $\mathcal{R}_0 \to \Hcal_{\K_0}^*(F^-(\nu_0))$ respectively. The above proposition states the following:
$$
q_0 = 0 \implies vq_1 = 0 \quad \text{for all} \quad v \in H_{\K_0}^*(\pt), \quad \text{deg}(v) > 0.
$$
In other words, there cannot be a non-zero element in $\Hcal_{\K_0}^*(F^-(\nu_1))$ which was trivial in $\Hcal_{\K_0}^*(F^-(\nu_0))$, and is still non-zero when multiplied by a positive degree element of $H_{\K_0}^*(\pt)$.

\textit{We now assume that $(M,\omega)$ is monotone and $\omega$ is primitive}. We have
$$
p = \frac{c_1}{N_M}.
$$
In this situation, the kernel $\mathcal{J}_{\K_0}^*(F^-(\nu))$ admits in some sense \textit{elements of minimal degree}:
\begin{propsubsec}{\label{prop1.2.3}}
There exists $q \in \mathcal{R}_0$, such that $q \notin \mathcal{J}_{\K_0}^*(F^-(\nu))$, but $u_i q \in \mathcal{J}_{\K_0}^*(F^-(\nu))$ for all $i=1,...,n$.
\end{propsubsec}
In contrast to the symplectic case \cite[Corollary $1.3$]{Giv95}, where elements of minimal degree always exist, the monotonicity assumption \textit{cannot be lifted} here, since in general, one might have 
$$
\mathcal{J}_{\K_0}^*(F^-(\nu)) = \mathcal{R}_0.
$$ 
This happens for instance when $(M, \omega, \mathbb{T}) = (\C P^{n-1}, \omega_{\text{FS}}, \mathbb{T}^n / S^1)$, in which case $\mathcal{R}_0 = \{0\}$. We will see another less trivial example of such an event in section \ref{sec5}.

\subsubsection*{Proof of the main theorem}
%\addcontentsline{toc}{subsection}{Proof of the main theorem}

The above discussion allows us to prove our main result, Theorem \ref{theo1.1.1}.

\begin{proof}
We adapt the proof of \cite{Giv95} to the contact setting. We work with the prequantization space $(V, \xi := \ker \alpha)$ constructed by the procedure described above, and a contactomorphism $\phi_h \in \text{Cont}_0(V,\xi)$. For any generic $\nu \in \R$, we have a cohomology group $\Hcal_{\K_0}^*(F^-(\nu))$, and the kernel $\mathcal{J}_{\K_0}^*(F^-(\nu))$ of the augmentation map, which satisfies Proposition \ref{prop1.2.3}. Suppose that $\nu \notin \Spec(g \circ \phi_h)$. Since $\Spec(g \circ \phi_h)$ is periodic of period $1$, it is enough to count the number of elements of $\Spec(g \circ \phi_h)$ between $\nu$ and $\nu + 1$. To every such element there corresponds at least one $\alpha$-translated point of $g \circ \phi_h$ on $V$. Therefore, to one $\alpha$-translated point of $g \circ \phi_h$ there correspond $l$ elements of the spectrum between $\nu$ and $\nu + l$, where $l \in \mathbb{N}$. Assume that $g \circ \phi_h$ has a finite number of $\alpha$-translated points. Let $m \in \liek_{\Z} \setminus \{0\}$ be such that $\iota(m) =  (m_1,...,m_n) \in \R_{\geq 0}^n$. We will see that $p(m) > 0$. Suppose that the number $\#$ of elements in $\Spec(g \circ \phi_h)$ between $\nu$ and $\nu + p(m)$ is strictly less than $c_1(m) = \underset{j=1}{\overset{n} \sum} m_j$. Let $q \in \mathcal{R}_0$ be such that $q \notin \mathcal{J}_{\K_0}^*(F^-(\nu))$, but $u_i q \in \mathcal{J}_{\K_0}^*(F^-(\nu))$ for all $i = 1,...,n$. Since $\# < c_1(m)$, Propositions \ref{prop1.2.1} and \ref{prop1.2.2} imply that $u_1^{m_1}...u_n^{m_n}q \in \mathcal{J}_{\K_0}^{*+2c_1(m)}(F^-(\nu+p(m)))$. This is precisely the Novikov action of $m$ (equation $(\ref{eq3})$), which is an isomorphism
$$
\mathcal{J}_{\K_0}^*(F^-(\nu)) \overset{u_1^{m_1}...u_n^{m_n}}{\overset{\sim} \longrightarrow} \mathcal{J}_{\K_0}^{*+2c_1(m)}(F^-(\nu+p(m))).
$$
In particular, $u_1^{m_1}...u_n^{m_n}q \notin \mathcal{J}_{\K_0}^{*+2c_1(m)}(F^-(\nu+p(m)))$, which is a contradiction. This means that the number $\#$ of elements of $\Spec(g \circ \phi_h)$ between $\nu$ and $\nu + p(m)$ is not less than $c_1(m)$, and thus the number of $\alpha$-translated points of $g \circ \phi_h$ is not less than $N_M$. 

To summarize, we have proved that for any $\phi_h \in \text{Cont}_0(V,\xi)$, if the number of $\alpha$-translated points of $g \circ \phi_h$ is finite, then it is at least $N_M$. Therefore, we have shown that for any $\phi_h \in \text{Cont}_0(V,\xi)$ the number of $\alpha$-translated points of $g \circ \phi_h$ is either infinite or at least $N_M$. Since this holds for any contactomorphism of $\text{Cont}_0(V,\xi)$, it holds also for $g^{-1} \circ \phi_h$, for any contact Hamiltonian $h$. Thus, the number of $\alpha$-translated points of $\phi_h$ is at least $N_M$.
\end{proof}

\noindent \textbf{{\large Discussion.}} \normalfont In \cite{Giv95}, Givental constructed the generating functions as a finite dimensional approximation of action functionals defined on the free loop space of $\C^n$ (see \cite[section $2$]{Giv95}). Following such an analogy, one could think of the limit $\Hcal_{\K_0}^*(F^-(\nu))$ as a kind of $\R$-filtered Floer cohomology group associated with the contactomorphism $\phi_h$. There exist several Floer-type constructions that pertain to translated points, see for instance \cite{AM13}, \cite{MU17}. Moreover, there are at least two ongoing projects, by Albers, Shelukhin and Zapolsky, and Leclercq and Sandon, concerned with comparable constructions and applications. Note that, even in the symplectic setting of \cite{Giv95}, building a precise relation between Floer homology and the limit of equivariant cohomology groups seems highly complicated, for Floer's construction is of much different nature than Givental's. Still, we observe a similar behavior, at least in the study of translated points. In addition, our limit has the very convenient property of being defined for any contactomorphism of $\text{Cont}_0(V,\xi)$, and our result holds in this level of generality. Another remark is that we currently don't know if the limit $\Hcal_{\K_0}^*(F^-(\nu))$ is non-trivial over a non-monotone base. If it is trivial, the above analogy raises the question of the existence of Floer homology groups for prequantization spaces over non-monotone symplectic manifolds.\\

\noindent \textbf{{\large Acknowledgements.}} I would like to warmly thank my advisor Frol Zapolsky for his support throughout this project. His professionalism and patience were crucial in every step of this work. I would also like to thank Leonid Polterovich and Michael Entov for listening to a preliminary version of this work and their interest, and Alexander Givental for his generous help regarding the paper \cite{Giv95}. I would like to thank the participants of the workshop Floer homology and contact topology, held at the University of Haifa in September $2017$, for their questions and attention during my talk about this project. Finally, I would like to thank the organizers and participants of the seminar Symplectix at the Institut Henri Poincaré in Paris, for their invitation and interest during my talk in January $2019$, more particularly Vincent Humilière, Alexandru Oancea, and Claude Viterbo. I was supported by grant number $1825/14$ from the Israel Science Foundation.

\section{Preliminaries}{\label{sec2}}

\subsection{Fronts and deformation of sublevel sets}{\label{sec2.1}}

The main symplectic ingredient of our construction is a family of generating functions associated with the lift of a contactomorphism (section \ref{sec3.3}). We will study suitable sublevel sets of these functions, and more particularly how they behave under certain deformations. We describe here the general setting and relevant results. This section is added only for the sake of completeness; we mainly follow \cite[section $3$]{Giv95}, except for some notations and statements that are adapted to this paper. Let $F : X \times \Lambda \to \R$ be a family of functions $F_{\lambda}$ on a compact manifold $X$, where $\Lambda$ is a compact parametrizing manifold with boundary $\partial \Lambda$. Suppose that $F \in C^{1,1}$, that is $F$ is differentiable with Lipschitz derivatives (so that gradient flow deformations apply), and $F$ is smooth at all points $(x,\lambda)$ such that $x$ is a critical point of $F_{\lambda}$ with critical value zero. We consider restrictions of $F$ to submanifolds $\Gamma \subset \Lambda$ with boundary $\partial \Gamma = \Gamma \cap \partial \Lambda$, and look at the sublevel sets
\begin{equation}{\label{eq5}}
F_{\Gamma}^- := \lbrace (x,\lambda) \in X \times \Lambda \ | \ \lambda \in \Gamma, F(x,\lambda) \leq 0 \rbrace, \quad  
F_{\partial \Gamma}^- := \lbrace (x,\lambda) \in X \times \Lambda \ | \ \lambda \in \partial \Gamma, F(x,\lambda) \leq 0 \rbrace.
 \end{equation}
We assume moreover that $0$ is a regular value of $F$. The \textbf{front of $F$} is defined as
$$
L := \lbrace \lambda \in \Lambda \ | \ 0 \text{ is a singular value of } F_{\lambda} \rbrace.
$$
It is the singular locus of the projection
$$F^{-1}(0) \to \Lambda,$$
and since $F$ is smooth at all critical points $(x, \lambda)$ such that $x$ is a critical point of $F_{\lambda}$ with critical value zero, it is of \textit{zero-measure}, due to Sard's lemma. To understand its structure, it is convenient to view it as a singular hypersurface in $\Lambda$, provided at every point with tangent hyperplanes (there can be more than one hyperplane at each point). The construction goes as follows: let $N$ be a smooth neighborhood in $X \times \Lambda$ of the set of critical points of $F$ with critical value $0$. The zero set $F_{|N}^{-1}(0) \subset X \times \Lambda$ gives rise naturally to a Legendrian submanifold

$$\mathcal{L} := \lbrace (x,\lambda,T_{(x,\lambda)}F_{|N}^{-1}(0)) \ |\  (x,\lambda) \in N \rbrace \subset PT^*(X \times \Lambda)$$
of the space $PT^*(X \times \Lambda)$ of contact elements of $X \times \Lambda$. Let $\mathcal{P}$ denote the space of vertical contact elements (also called the \textit{mixed space}, see \cite{AN01}): a hyperplane $H_{(x,\lambda)}$ is in $\mathcal{P}$ if and only if $T_x X \times \{ \lambda \} \subset H_{(x,\lambda)}$. Then the intersection $ \hat{\mathcal{L}} := \mathcal{L} \cap \mathcal{P}$ projects to a singular Legendrian submanifold $\hat{L} \subset PT^*(\Lambda)$ in the space of contact elements $PT^*(\Lambda)$ of $\Lambda$. The front $L$ is then defined as the projection of $\hat{L}$ to the base $\Lambda$, whereas a tangent hyperplane at a point $\lambda \in L$ is an element of $\hat{L} \cap PT^*_{\lambda} (\Lambda)$.\footnote{The term \textit{front} comes from the study of wave fronts in geometrical optics. Given a Legendrian submanifold $L$ of a Legendrian fibration $E \to B$, the front of $L$ is defined as the projection of $L$ to the base $B$ (see \cite[chapter $5$]{AN01}). If $L$ projects injectively onto its front, the latter is also called the \textit{generating hypersurface} of $L$. In our case $E = PT^*(\Lambda)$, and the (singular) Legendrian submanifold $\hat{L}$ is the projection of the (singular) intersection $\mathcal{L} \cap \mathcal{P} \subset PT^*(X \times \Lambda)$ to $PT^*(\Lambda)$. The zero set $F^{-1}(0) \subset X \times \Lambda$ is called a \textit{family of generating hypersurfaces} for $\hat{L}$. Note that if $\mathcal{L}$ was transversal to $\mathcal{P}$, $\hat{L}$ would be an immersed Legendrian submanifold in $PT^*(\Lambda)$. In fact, one can show that any immersed Legendrian submanifold of a space of contact elements can be obtained via such a procedure.}

\begin{propsubsec}[\cite{Giv95} Proposition $3.1$]{\label{prop2.1.1}}
A submanifold $\Gamma \subset \Lambda$ is transversal to $L$ if and only if the hypersurface $F^{-1}(0)$ is transversal to $X \times \Gamma$.
\begin{proof}
$X \times \Gamma$ is tangent to $F^{-1}(0)$ at a point $(x,\lambda) \in X \times \Lambda$ if and only if $T_{(x,\lambda)} (X \times \Gamma) \subset T_{(x,\lambda)} F^{-1}(0)$, if and only if $T_{(x,\lambda)} (X \times \Gamma) \subset \hat{\mathcal{L}}$, if and only if $T_{\lambda} \Gamma \subset \hat{L} \cap PT_{\lambda}^*(\Lambda)$.
\end{proof}
\end{propsubsec}
\begin{corsubsec}[\cite{Giv95} Corollary $3.2$]{\label{cor2.1.1}}
Let $\Gamma_t := \rho^{-1}(t)$ be non-singular levels of some smooth function $\rho : \Lambda \to \R$. Then almost all $\Gamma_t$ are transversal to $L$. \qed
\end{corsubsec}
The two following propositions can be proved using standard gradient flow deformations.
\begin{propsubsec}{\label{prop2.1.2}}
Let $F : X \times \Lambda \times [0,1] \to \R$ be a $C^{1,1}$ family of functions $F_s : X \times \Lambda \to \R$ such that for any $s \in [0,1]$, $0$ is a regular value of $F_s$, and $F_s$ is smooth at all critical points $(x, \lambda)$ such that $x$ is a critical point of $F_{s, \lambda}$ with critical value $0$. Let $\Gamma \subset \Lambda$ be a compact submanifold with boundary $\partial \Gamma = \Gamma \cap \partial \Lambda$, and suppose that for any $s \in [0,1]$, $\Gamma$ and $\partial \Gamma$ are transversal to $L_s$. Then there exists a Lipschitz isotopy $I : X \times \Lambda \times [0,1] \to X \times \Lambda$ such that
$$I_s(F_{0, \Gamma}^-, F_{0, \partial \Gamma}^-) = (F_{s, \Gamma}^-, F_{s, \partial \Gamma}^-),$$
where we have denoted by $F_{s, \Gamma}^-, F_{s, \partial \Gamma}^-$ the sublevel sets defined for $F_s$ as in equation (\ref{eq5}). If moreover each $F_s$ is invariant under the action of a compact Lie group $G$ on $X$, the isotopy can be made $G$-equivariant.
\end{propsubsec}
\begin{propsubsec}{\label{prop2.1.3}}
Let $\lbrace \Gamma_s \rbrace_{s \in [0,1]}$ be a smooth family of regular levels $\Gamma_s = \rho^{-1}(s) \subset \Lambda$ of some smooth fuction $\rho : \Lambda \to \R$. Suppose that for any $s \in [0,1]$, $\Gamma_s$ and $\partial \Gamma_s$ are transversal to $L$, and that $F$ is smooth at all critical points $(x, \lambda)$ such that $x$ is a critical point of $F_{\lambda}$ with critical value $0$. Then there exists a Lipschitz isotopy $I : X \times \Lambda \times [0,1] \to X \times \Lambda$ such that
$$I_s(F_{\Gamma_0}^-, F_{\partial \Gamma_0}^-) = (F_{\Gamma_s}^-, F_{\partial \Gamma_s}^-).$$
If moreover $F$ is invariant under the action of a compact Lie group $G$ on $X$, the isotopy can be made $G$-equivariant.
\end{propsubsec}

\begin{remarksubsec}
\normalfont In the sequel $\Lambda$ and $\Gamma$ will be a \textit{stratified} manifolds. In this setting, one must improve the notion of transversality: $\Gamma$ is transversal to $L$ if each of its strata is. The results above remain valid when replacing manifolds by stratified manifolds, for this adapted notion of transversality (we refer to \cite{GM88} for a detailed treatment of stratified Morse theory).
\end{remarksubsec}

\subsection{Equivariant cohomology and conical spaces}{\label{sec2.2}}

In the sequel we will study the equivariant homotopy type of conical sublevel sets of generating families. This section is meant to recall several facts from equivariant cohomology, fix some notations, and describe a simple identification in equivariant cohomology which holds for conical spaces. The latter will be convenient for defining homomorphisms of equivariant cohomology groups (see section \ref{sec3.6}). We refer to \cite{BGS13} for a complete treatment of equivariant cohomology theory. Let $X$ be a topological space provided with the action of a compact Lie group $G$. Equivariant cohomology $H^*_G(X;\C)$ with complex coefficients is defined as the singular cohomology $H^*(X_G;\C)$ of the quotient
$$X_G := (X \times EG) / G,$$
where $EG \to BG$ is the universal principal $G$-bundle over the classifying space $BG$ of $G$. The canonical projection
$$(X \times EG)/G \to EG/G$$
provides $H^*_G(X;\C)$ with the structure of a module over $H^*(BG;\C)$, which plays the role of the coefficient ring $H^*_G(\pt;\C)$ in equivariant cohomology theory. If $G$ is the $n$-dimensional torus $\mathbb{T}^n = \R^n / \Z^n$, the coefficient ring $H^*_{\mathbb{T}^n}(\pt;\C)$ is naturally isomorphic to a polynomial algebra in $n$ variables $u = (u_1,\hdots ,u_n)$ of degree $2$. More precisely, there is a natural algebra isomorphism, called the \textit{Chern-Weil isomorphism}
$$\chi : H_{\mathbb{T}^n}^*(\pt;\C) \simeq \C[\R^{n*}],$$
between the cohomology of the classifying space $B \mathbb{T}^n$ and the polynomial algebra on $\R^{n*}$ over $\C$. 

Throughout this paper we will mainly deal with singular cohomology with complex coefficients, and will therefore use the notation $H^*(Y) := H^*(Y;\C)$. If $(u_1,\hdots ,u_n)$ denotes the standard basis of $\R^{n*}$, the Chern-Weil isomorphism writes
$$
H^*_{\mathbb{T}^n}(\pt) \simeq \C[u_1,\hdots ,u_n].
$$

Let now $A \subset X$ be a $G$-invariant subspace of $X$, and $pr : X \times EG \to X_G$ denote the canonical projection. We introduce the following notations for $G$-equivariant cochain complexes:
\begin{itemize}
\item $C_G^*(X) := C^*(X_G)$ and $C_G^*(A) := C^*(A_G)$;
\item $C_G^*(X, A) := \ker (C_G^*(X) \to C_G^*(A))$, where the map is induced by the inclusion $A \subset X$;
\item $C_{G,c}^*(X) := \lbrace \sigma \in C_G^*(X) : \exists K \subset X \text{ compact such that } pr^* \sigma \text{ vanishes on } (X \setminus K) \times EG\rbrace$, the complex of equivariant cocycles with compact support;
\item $C_{G,c}^*(A) := \lbrace \sigma \in C_G^*(A) : \exists K \subset A \text{ compact such that } pr^*_{| A \times EG} \sigma \text{ vanishes on } (A \setminus K) \times EG\rbrace$;
\item $C_{G,c}^*(X,A) := \ker (C^*_{G,c}(X) \to C^*_{G,c}(A))$;
\item if $X \to M$ is a fiber bundle of $G$-spaces, $C_{G,cv}^*(X)$, $C_{G,cv}^*(A)$, $C_{G,cv}^*(X,A)$, the complexes of equivariant (relative) cocycles with vertical compact support, that is we replace the term compact above by compact in the fibers direction. We will denote by $H_{G,cv}^*$ the corresponding equivariant cohomology groups.  
\end{itemize}
Suppose that $X$ and $A$ are closed conical subspaces of the standard Euclidean space $\R^{2m} = \C^m$, that is $tX \subset X$ and $tA \subset A$ for all $t > 0$, and that $G$ is a subtorus of the maximal torus $\mathbb{T}^m := (S^1)^m$, acting on $\C^m$ by rotation on each coordinate. 
Then the natural retraction of $\R^{2m}$ onto the closure $\overline{B}$ of its unit ball $B$ induces a $G$-equivariant retraction of $X$ (resp. $A$) onto $X \cap \overline{B}$ (resp. $A \cap \overline{B}$), and of $X \setminus (X \cap B)$ (resp. $A \setminus (A \cap B)$) onto $X \cap S$ (resp. $A \cap S$), where $S = \partial \overline{B}$. Moreover, the inclusion $C_G^*(X, X \setminus X \cap B) \hookrightarrow C_{G,c}^*(X)$ induces an isomorphism in equivariant cohomology, since any compact subset $K \subset X$ is included in an intersection $X \cap \widetilde{B}$, where $\widetilde{B}$ is a ball centered at $0$, and $X \cap \widetilde{B}$ is $G$-equivariantly homotopic to $X \cap B$ ($X$ is conical and $G$ commutes with $\R_{>0}$). The same argument applies when replacing $X$ with $A$. Putting all these simple facts together yields a natural isomorphism of cohomology groups
$$
H^*(C_G^*(X \cap S, A \cap S)) \simeq H^*(C_G^*(X,A) / C_{G,c}^*(X,A)).
$$
\qed

\section{The constructions}{\label{sec3}}

\subsection{The prequantization space}{\label{sec3.1}}

In this section, we describe a natural construction of a prequantization space over an integral symplectic toric manifold $(M^{2d},\omega, \mathbb{T})$, mainly following \cite{BZ15}. The action of $\mathbb{T}$ is induced by a momentum map $M \to \liet^*$, where $\liet^*$ is the dual to the Lie algebra $\liet$ of $\mathbb{T}$. The image $\Delta$ of the momentum map is called the moment polytope. If $\Delta$ has $n$ facets, it is given by
$$
\Delta = \lbrace x \in \liet^* \ |\  \langle x,v_j \rangle + a_j \geq 0 \text{ for } j=1,\hdots ,n \rbrace,
$$
where the conormals $v_j$ are primitive vectors in the integer lattice $\liet_{\Z} := \ker(\exp : \liet \to \mathbb{T})$, and $a := (a_1,\hdots ,a_n) \in \R_{\geq 0}^{n*} \setminus \{0\}$. The polytope $\Delta$ is compact and smooth, that is each $k$-codimensional face of $\Delta$ is the intersection of exactly $k$ facets and the $k$ associated conormals $\lbrace v_{l_1}, \hdots , v_{l_k} \rbrace$ can be extended to an integer basis for the lattice $\liet_{\Z}$. 

\subsubsection{Delzant's construction of symplectic toric manifolds}{\label{sec3.1.1}}

Let us first recall Delzant's construction of symplectic toric manifolds \cite{Del88}. The standard Hamiltonian action of the torus $\mathbb{T}^n := \R^n / \Z^n$ on $(\C^n, \omega_{\text{std}} := \underset{j=1}{\overset{n} \sum} \droit x_j \wedge \droit y_j)$ by rotation in each coordinate is induced by the momentum map
\begin{equation}{\label{eq6}}
P : \C^n \to \R^{n*}, \quad \text{where} \quad \langle P(z), \lambda \rangle = \pi \underset{j=1}{\overset{n} \sum} \lambda_j |z_j|^2, \quad \text{and} \quad \lambda = (\lambda_1,\hdots ,\lambda_n) \in \R^n. 
\end{equation}
Consider the following surjective linear map:
$$\beta_{\Delta} : \R^n \to \liet, \quad \epsilon_j \mapsto v_j, \quad \text{for all} \quad j=1,\hdots ,n,$$
where $(\epsilon_1,\hdots ,\epsilon_n)$ is the standard basis of $\R^n$, and $v_j \in \liet_{\Z}$ are the conormals. Since $\Delta$ is compact and smooth, the map $\beta_{\Delta}$ satisfies $\beta_{\Delta} (\Z^n) = \liet_{\Z}$, and therefore it induces a homomorphism $[ \beta_{\Delta} ] : \mathbb{T}^n \to \mathbb{T}$. We define the connected subtorus 
$$\K \subset \mathbb{T}^n$$
as the kernel of $[\beta_{\Delta}]$. It has Lie algebra
$$\liek := \ker(\beta_{\Delta} : \R^n \to \liet),$$
and if $\iota : \mathfrak{k} \hookrightarrow \R^n$ denotes the inclusion, the momentum map for the action of $\K$ on $\C^n$ is given by
$$
P_{\K} := \iota^* \circ P : \C^n \to \liek^*.
$$
The torus $\K$ acts freely on the regular level set
$$
P_{\K}^{-1}(p), \quad \text{where} \quad p := \iota^*(a) \in \liek^* \setminus \{0\},\footnote{We will see below that compactness of $M$ is equivalent to $\ker \iota^* \cap \R_{\geq 0}^{n*} = \{0\}$. If $p = 0$, then $a \in \ker \iota^*$. This cannot happen, since $\ker \iota^* \cap \R_{\geq 0}^{n*} = \{0\}$, and $a \in \R_{\geq 0}^{n*} \setminus \{0\}$.}
$$
and if $X_{\lambda}(z) = 2 i \pi (\lambda_1 z_1,\hdots ,\lambda_n z_n) \in \C^n = T_z \C^n$ denotes the Hamiltonian vector field for the function $\langle P, \lambda \rangle : \C^n \to \R$, and $\alpha_{\text{std}} = \frac{1}{2} \underset{j=1}{\overset{n}{\sum}}(x_j \droit y_j - y_j \droit x_j)$ is the standard $1$-form on $\C^n$ so that $\droit \alpha_{\text{std}} = \omega_{\text{std}}$, we have
\begin{equation}{\label{eq7}}
\alpha_{\text{std}}(X_{\lambda}) = \langle P, \lambda \rangle \quad \text{and} \quad \iota_{X_{\lambda}} \droit \alpha_{\text{std}} = \iota_{X_{\lambda}} \omega_{\text{std}} = - \droit \langle P, \lambda \rangle.
\end{equation}
In particular, $(\mathcal{L}_{X_{\lambda}} \omega_{\text{std}})|_{P_{\K}^{-1}(p)} = 0$. Therefore, symplectic reduction gives rise to a symplectic manifold $(M_{\Delta}, \omega_{\Delta})$, where
$$
M_{\Delta} := P_{\K}^{-1}(p) / \K, \quad \text{and the symplectic form $\omega_{\Delta}$ is induced by} \quad \omega_{\text{std}}|_{P_{\K}^{-1}(p)}.
$$
Finally, Delzant's theorem \cite{Del88} shows that $(M_{\Delta},\omega_{\Delta}, \mathbb{T}^n / \K)$ and $(M,\omega, \mathbb{T})$ are equivariantly symplectomorphic as symplectic toric manifolds.

\subsubsection{The contact sphere}{\label{sec3.1.2}}

The generating families of our paper differ from those defined in \cite{Giv95} for they are related to a contactomorphism of a prequantization space over $M$, rather than to a symplectomorphism of $M$. Yet, we will see now that the regular level set $P_{\K}^{-1}(p)$ lies in a contact sphere in $\C^n \setminus \{0\}$ (in fact it lies in many). Since the latter is the symplectization of the former, this will allow us to lift the contactomorphism of the prequantization space to a symplectomorphism of $(\C^n, \omega_{\text{std}})$, (see section \ref{sec3.2.2}), and thus we will be able to define the generating families as in the aforementioned paper. The existence of such a sphere is ensured by compactness of the toric manifold $(M, \omega, \mathbb{T})$. There is no canonical way of choosing it, however the prequantization space won't depend on such a choice (in fact it does not depend on the sphere at all, see Remark \ref{rem3.1.1}). 

The image of the momentum map $P$ from equation $(\ref{eq6})$ is the first orthant $\R^{n*}_{\geq 0}$, and the polytope $\Delta$ can be identified with 
$$
(\iota^*)^{-1}(p) \cap \R^{n*}_{\geq 0}.
$$
Indeed, there is a commutative diagram
$$
\begin{tikzcd}
\C^n \arrow[r, "P"] & \R^{n*} \arrow[r, "\iota^*"] & \liek^*\\
P_{\K}^{-1}(p) \arrow[u, hook] \arrow[r, "\widetilde{\mu}"] &  \liet^* \arrow[u, "\beta_{\Delta}^*"]\\
& 0 \ar[u],
\end{tikzcd}
$$
where $\widetilde{\mu}$ is the composition $\mu \circ \pi$ of the momentum map of the $\mathbb{T}$-action on $M \simeq P_{\K}^{-1}(p) / \K$ with the natural projection $\pi : P_{\K}^{-1}(p) \to P_{\K}^{-1}(p) / \K$. The image of $\widetilde{\mu}$ is the moment polytope $\Delta$, so that
$$
\Delta \simeq \beta_{\Delta}^*(\Delta) = \beta_{\Delta}^* \circ \widetilde{\mu}(P_{\K}^{-1}(p)) = P(P_{\K}^{-1}(p)) = (\iota^*)^{-1}(p) \cap P(\C^n) = (\iota^*)^{-1}(p) \cap \R_{\geq 0}^{n*}.
$$
Moreover, compactness of $M_{\Delta}$ is equivalent to that of $\Delta$, which is ensured by the condition 
$$
\ker \iota^* \cap \R^{n*}_{\geq 0} = \lbrace 0 \rbrace.
$$
Note that $\ker \iota^*$ is equal to the annihilator $\iota(\liek)^{\perp}$ of $\iota(\liek)$ in $\R^{n*}$. We claim that $\iota(\liek) \cap \R_{>0}^n \neq \emptyset$. In such a case, the contact sphere can be defined as follows: for any $b \in \liek$ such that $\iota(b) = (b_1, \hdots , b_n) \in \R_{>0}^n$, we have
$$
p(b) = \langle \iota^*(a), b \rangle = \langle a, \iota(b) \rangle > 0,
$$
since $a \in \R_{\geq 0}^{n*} \setminus \{0\}$. Then the following contact sphere obviously contains $P_{\K}^{-1}(p)$:
\begin{equation}{\label{eq8}}
S_p := \lbrace z \in \C^n \ | \ \underset{j=1}{\overset{n} \sum} b_j \pi |z_j|^2 = p(b) \rbrace \quad \text{with contact form} \quad \alpha_p := \alpha_{\text{std} | S_p}.
\end{equation}
We are thus led to prove that
$$
\iota(\liek) \cap \R_{>0}^n \neq \emptyset.
$$
This is a consequence of the hyperplane separation theorem (see for instance \cite{BV04}): if $\iota(\liek) \cap \R_{>0}^n = \emptyset$, there exists a \textit{non-zero} vector $v$ and a real number $c \in \R$ such that 
$$
\langle v, \iota(\lambda) \rangle \leq c \quad \text{and} \quad \langle v, y \rangle \geq c,
$$
for all $\lambda \in \liek$, and $y \in \R_{>0}^n$. Since $\iota(\liek)$ is a vector space, we have necessarily $\langle v, \iota(\lambda) \rangle = 0$, for all $\lambda \in \liek$. In particular, $c \geq 0$, whence $\langle v, y \rangle \geq 0$, for all $y \in \R_{>0}^n$. For any $\epsilon$, we then have 
$$
\langle v, (1, \epsilon, \hdots , \epsilon) \rangle \geq 0.
$$
In the limit $\epsilon \to 0$, this yields $v_1 \geq 0$. Arguing similarly for the other coordinates of $v$, we obtain that $v \in \iota(\liek)^{\perp} \cap \R_{\geq 0}^{n*}$, which is a contradiction with $M_{\Delta}$ being compact. This proves that $\iota(\liek) \cap \R_{>0}^n \neq \emptyset$, as well as the existence of the above contact sphere.

\subsubsection{The prequantization space}{\label{sec3.1.3}}

We now construct our prequantization space. From now on we assume that the symplectic toric manifold $(M,\mathbb{T},\omega)$ is integral. Notice that since $\beta_{\Delta}(\Z^n) = \liet_{\Z}$, the inclusion $\iota$ satisfies
$$\iota(\liek_{\Z}) \subset \Z^n.$$
Thus, the integrality assumption is equivalent to
$$
p \in \liek_{\Z}^*.
$$
Let $\liek_0$ denote the kernel of $p : \liek \to \R$. Then $\liek_{0,\Z} := \ker (p : \liek_{\Z} \to \Z) \subset \liek_0$ is a sublattice of $\liek_{\Z}$. This means that $\K_0 := \exp (\liek_0) \subset \K$ is a subtorus of codimension $1$ of $\K$ which, in particular, acts freely on the regular level set $P_{\K}^{-1}(p)$. Let $j : \liek_0 \hookrightarrow \liek$ denote the inclusion of Lie algebras. Then $P_{\K}^{-1}(p)$ is the zero level of the contact momentum map
$$
\mu : S_p \to \liek_0^*, \quad \mu(z) := j^* \circ P_{\K}(z)
$$
associated with the action of $\K_0$ on the contact sphere $(S_p, \alpha_p)$. Therefore, contact reduction yields a contact manifold
$$
(V := P_{\K}^{-1}(p) / \K_0, \xi := \ker \alpha), \quad \rho^* \alpha = \alpha_{p | P_{\K}^{-1}(p)},
$$
where $\rho : P_{\K}^{-1}(p) \to V$ denotes the canonical projection (see \cite{Gei08}  
for more on contact reductions). For the circle $S^1 := \K / \K_0$, the projection
$$
\pi : (V, \alpha) \to (M_{\Delta},\omega_{\Delta})
$$
defines a principal $S^1$-bundle, and satisfies $\pi^* \omega_{\Delta} = \droit \alpha$, since $\omega_{\text{std}} = \droit \alpha_{\text{std}}$. Finally, one can choose an equivariant symplectomorphism $(M_{\Delta},\omega_{\Delta},\mathbb{T}^n / \K) \simeq (M,\omega,\mathbb{T})$, and obtain a prequantization space over the symplectic toric manifold $(M,\omega,\mathbb{T})$.

\begin{remarksubsec}{\label{rem3.1.1}}
\normalfont In \cite{BZ15}, the construction of $V$ does not involve a contact sphere. Here as well, it is enough to remark that the infinitesimal action of $\K_0$ is tangent to $\ker \alpha_{\text{std}}$ along $P_{\K}^{-1}(p)$, which follows from (\ref{eq7}). This implies that the contact form $\alpha$ is well-defined. We view $V$ as the quotient of the $S_p$ only for the lifting procedure of the following section.
\end{remarksubsec}

\subsection{Lifting contact isotopies}{\label{sec3.2}}

In this section we explain the procedure for lifting a contact isotopy of $V$ to a Hamiltonian isotopy of $\C^n$. Recall that for any contact manifold $(V,\xi)$ and any choice of contact form $\alpha$, any \textbf{contact Hamiltonian} $h : V \times [0,1] \to \R$ gives rise to a unique time-dependent vector field $X_h$ satisfying
$$
\alpha(X_h^t) = h_t \quad \text{and} \quad \droit \alpha(X_h^t,.) = - \droit h_t + \droit h_t(R_{\alpha}) \alpha, \quad h_t := h(.,t).
$$
The vector field $\{ X_h^t \}_{t \in [0,1]}$ preserves $\xi$ and, if $V$ is compact, it integrates into a contact isotopy defined for all $t \in [0,1]$, and denoted by $\{ \phi_h^t \}_{t \in [0,1]}$. This establishes a bijection, depending on the contact form $\alpha$, between smooth time-dependent functions $h : V \times [0,1] \to \R$ and contact isotopies of $V$.

\subsubsection{Lift to the contact sphere}{\label{sec3.2.1}}

Following \cite[Definition $1.6$]{BZ15}, we say that a closed submanifold $Y \subset V$ transverse to $\xi$ is \textit{strictly coisotropic with respect to $\alpha$} if it is coisotropic, that is the subbundle $TY \cap \xi$ of the symplectic vector bundle $(\xi,\droit \alpha_{| \xi})$ is coisotropic:
$$
\lbrace X \in \xi_y \ | \ \iota_X \droit \alpha = 0 \text{ on } T_yY \cap \xi_y \rbrace \subset T_yY \cap \xi_y, \quad \text{for all} \quad y \in Y,
$$
and additionally $R_{\alpha} \in T_yY$ for all $y \in Y$, that is the Reeb vector field is tangent to $Y$. 

Consider the setting
$$
(S_p, \alpha_p) \supset (P_{\K}^{-1}(p),\alpha_{p | P_{\K}^{-1}(p)}) \overset{\rho}{\to} (V,\alpha).
$$
Then $P_{\K}^{-1}(p)$ is strictly coisotropic with respect to $\alpha_p$, since it is the zero level of the contact momentum map $\mu : S_p \to \liek_0^*$ associated with the action of $\K_0$ on the contact sphere $(S_p, \alpha_p)$ (see for instance \cite[Lemma $7.7.4$]{Gei08}). Let $h : V \times [0,1] \to \R$ be a time-dependent contact Hamiltonian of $V$, and let $\phi_h := \phi_h^1$ denote the time-$1$ map of the contact isotopy $\lbrace \phi_h^t \rbrace_{t \in [0,1]}$ generated by $h$. We first lift $h_t := h(.,t)$ to a $\K_0$-invariant function 
$$
\tilde{h}_t : P_{\K}^{-1}(p) \to \R, \quad \tilde{h}_t := \rho^* h_t,
$$
and then extend $\tilde{h}$ to a $\K_0$-invariant contact Hamiltonian $\overline{h} : S_p \times [0,1] \to \R$, so that $\overline{h}_{t | P_{\K}^{-1}(p)} = \tilde{h}_t$. By \cite[Lemma $3.1$ and $3.2$]{BZ15}, the contact isotopy $\lbrace \phi_{\overline{h}}^t \rbrace_{t \in [0,1]}$ generated by $\overline{h}$ and the Reeb flow $\lbrace \phi_{\alpha_p}^t \rbrace_{t \in \R}$ of $\alpha_p$ preserve $P_{\K}^{-1}(p)$, and project to the contact isotopy $\lbrace \phi_h^t \rbrace_{t \in [0,1]}$ and the Reeb flow $\lbrace \phi_{\alpha}^t \rbrace_{t \in \R}$ of $\alpha$ respectively. More precisely, the following equalities hold:
$$
\begin{array}{lcl}
\rho \circ \phi_{{\alpha_p} | P_{\K}^{-1}(p)}^t = \phi_{\alpha}^t \circ \rho : P_{\K}^{-1}(p) \to V \quad \text{for all $t \in \R$};\\
\rho \circ \phi_{\overline{h} | P_{\K}^{-1}(p)}^t = \phi_h^t \circ \rho : P_{\K}^{-1}(p) \to V \quad \text{for all $t \in [0,1]$}.\\
\end{array}
$$
	
Let $q \in V$ be an $\alpha$-translated point of $\phi_h$, that is
$$
\phi_h(q) = \phi_{\alpha}^s(q) \quad \text{for some $s \in \R$, and} \quad (\phi_h^* \alpha)_q = \alpha_q.
$$
Then $q$ is a \textbf{discriminant point} of $\phi_{\alpha}^{-s} \circ \phi_h$, that is an $\alpha$-translated point which is also a fixed point of $\phi_{\alpha}^{-s} \circ \phi_h$. %In other words, on can say that $q$ is a discriminant point of $\phi_h$, up to the $\K / \K_0$-action on $V$. Now, f
For any $z \in P_{\K}^{-1}(p)$ such that $\rho(z) = q$, the equations above show moreover that
$$
\rho (\phi_{\alpha_p}^{-s} \circ \phi_{\overline{h}}(z)) = \rho(z).
$$
In other words, $\phi_{\alpha_p}^{-s} \circ \phi_{\overline{h}}(z)$ and $z$ lie in the same $\K_0$-orbit in $P_{\K}^{-1}(p)$. Since the Reeb orbits of $\alpha$ generate the circle $\K / \K_0$, there exists $\lambda \in \liek$ such that $p(\lambda) = -s$ and $\exp(\lambda) \circ \phi_{\overline{h}}(z) = z$. Moreover,
$$
(\phi_{\overline{h} | P_{\K}^{-1}(p)}^* \rho^* \alpha)_z = ((\rho \circ \phi_{\overline{h} | P_{\K}^{-1}(p)})^* \alpha)_z = ((\phi_h \circ \rho)^* \alpha)_z = (\rho^* \phi_h^* \alpha)_z = (\rho^* \alpha)_z.
$$
On the other hand, $\phi_{\overline{h}}$ is a contactomorphism, so that $\phi_{\overline{h}}^*\alpha_p = e^g \alpha_p$, for a function $g : S_p \to \R$. Therefore, we have
$$
(\phi_{\overline{h} | P_{\K}^{-1}(p)}^* \rho^* \alpha)_z = ((\phi_{\overline{h}}^* \alpha_p)_ {| P_{\K}^{-1}(p)})_z = e^{g_{|P_{\K}^{-1}(p)}(z)} (\alpha_{p | P_{\K}^{-1}(p)})_z = e^{g(z)} (\rho^* \alpha)_z.
$$
In particular, $g(z) = 0$, and therefore $(\phi_{\overline{h}}^* \alpha_p)_z = \alpha_{p,z}$. Finally, recall that $\K$ acts by $\alpha_p$-preserving transformations (see equation (\ref{eq7})). Putting everything together, we have shown that \textit{$z$ is a discriminant point of $\exp(\lambda) \circ \phi_{\overline{h}}$}.

\subsubsection{Lift to a symplectic vector space}{\label{sec3.2.2}}

A convenient way to pass from the contact to the symplectic setting consists of associating to a contact manifold its symplectization. We briefly recall this procedure, and apply it to lift contactomorphisms of the contact sphere $(S_p, \alpha_p)$ to symplectomorphisms of $\C^n$.

Let $(N,\xi = \ker \alpha)$ be a contact manifold. Its \textbf{symplectization} is the symplectic manifold
$$
SN := N \times \R \quad \text{with symplectic form} \quad \droit (e^r \alpha),
$$
where $r$ is a coordinate on $\R$. Let $\phi$ be a contactomorphism of $N$. Then $\phi$ lifts up to a symplectomorphism 
$$
\begin{array}{lcl}
\Phi : SN \to SN \\
\quad (x,r) \to (\phi(x),r - g(x)),
\end{array}
$$
where $g : N \to \R$ is the function satisfying $\phi^* \alpha = e^g \alpha$. In particular, a discriminant point $q$ of $\phi$ corresponds to an $\R$-line $\lbrace (q,r) \in SN \ |\ r \in \R \rbrace$ of fixed points of $\Phi$. If $\phi := \phi_h$ is the time-$1$ map of a contact isotopy $\lbrace \phi_h^t \rbrace_{t \in [0,1]}$ generated by a contact Hamiltonian $h : N \times [0,1] \to \R$, $\Phi$ is the time-$1$ map $\Phi = \Phi_H := \Phi_H^1$ of the Hamiltonian isotopy $\lbrace \Phi_H^t \rbrace_{t \in [0,1]}$ generated by the Hamiltonian
$$
H_t(x,r) := e^r h_t(x), \quad t \in [0,1].
$$

Suppose now that $N \subset \C^n$ is a star-shaped hypersurface, that is the image of a map 
$$
\begin{array}{lcl}
S^{2n-1} \to \C^n\\
\quad z \to f(z)z,
\end{array}
$$
where $f$ is a smooth positive function on $S^{2n-1}$. One can show that the standard Liouville form $\alpha_{\text{std}}$ on $\C^n$ restricts to a contact form on $N$, and the symplectization $(SN, \droit(e^r \alpha_{\text{std}|N}))$ of $(N, \alpha_{\text{std}|N})$ is symplectomorphic to $(\C^n \setminus \lbrace 0 \rbrace, \omega_{\text{std}|\C^n \setminus \lbrace 0 \rbrace})$, via the symplectomorphism
$$
\begin{array}{lcl}
\Psi : N \times \R \to \C^n \setminus \lbrace 0 \rbrace. \\
\quad \quad (x,r) \to e^{\frac{r}{2}} x
\end{array}
$$
We then have
$$
\begin{array}{lcl}
\Psi \Phi \Psi^{-1} : \C^n \setminus \lbrace 0 \rbrace \to \C^n \setminus \lbrace 0 \rbrace,\\
\quad \quad \quad \quad \quad \ \  z \to \frac{|z|}{|pr(z)|e^{ \frac{1}{2} g(pr(z))}} \phi(pr(z))
\end{array}
$$
where 
$$
\begin{array}{lcl}
pr : \C^n \setminus \lbrace 0 \rbrace \to N \\
\quad \quad \quad \ z \to f(\frac{z}{|z|})\frac{z}{|z|}
\end{array}
$$
is the radial projection, and $|z| := \sqrt{\underset{j=1}{\overset{n} \sum} |z_j|^2}$. For the time-$1$ map $\phi_h$ of a contact isotopy $\{ \phi_h^t \}_{t \in [0,1]}$, the Hamiltonian of $\C^n \setminus \lbrace 0 \rbrace$ generating the Hamiltonian isotopy $\lbrace \Psi \Phi_H^t \Psi^{-1} \rbrace_{t \in [0,1]}$ is of the form
$$
\widetilde{H}_t(z) := \frac{|z|^2}{|pr(z)|^2} h_t(pr(z)). 
$$
It is homogeneous of degree $2$ with respect to the $\R_{>0}$-action on $\C^n \setminus \lbrace 0 \rbrace$, that is
$$
\widetilde{H}_t(r z) = r^2 \widetilde{H}_t(z), \quad \text{for all} \quad r>0.
$$
The Hamiltonian isotopy $\lbrace \Phi_{\widetilde{H}}^t \rbrace_{t \in [0,1]} = \lbrace \Psi \Phi^t_H \Psi^{-1} \rbrace_{t \in [0,1]}$ is therefore $\R_{>0}$-equivariant, and we can extend $\widetilde{H}_t$ and $\Phi^t_{\widetilde{H}}$  continuously to $\C^n$ by 
\begin{align*}
\begin{array}{lcl}
\widetilde{H}_t(0)=0 \\
\Phi_{\widetilde{H}}^t(0)=0.
\end{array}
\end{align*}

Back to our setting, the contact sphere $(S_p,\alpha_p)$ is a star-shaped hypersurface of $\C^n$, with
$$
f(z) := \sqrt{\frac{p(b)}{\underset{j=1}{\overset{n}\sum} \pi b_j |z_j|^2}}.
$$
The action of $\K$ on $S_p$ lifts to the linear action of $\K$ on $\C^n$, and $pr$ is $\K$-equivariant. Consider a contact Hamiltonian $h : V \times [0,1] \to \R$, and its $\K_0$-invariant lift $\overline{h} : S_p \times [0,1] \to \R$. The Hamiltonian $\widetilde{H}$ extending $\overline{h}$ to $\C^n$ is $\K_0$-invariant, and therefore the Hamiltonian isotopy $\lbrace \Phi_{\widetilde{H}}^t \rbrace_{t \in [0,1]}$ is $\K_0$-equivariant. We call $\widetilde{H}$ \textbf{a Hamiltonian lift of $h$}.

Let $q \in V$ be an $\alpha$-translated point of $\phi_h$. Any $z \in P_{\K}^{-1}(p)$ such that $\rho(z) = q$ is a discriminant point of $\exp(\lambda) \circ \phi_{\overline{h}}$, for some $\lambda \in \liek$. On $\C^n$, the latter becomes an $\R_{>0}$-line of fixed points of $\exp(\lambda) \circ \Phi_{\widetilde{H}}$:
$$
\exp(\lambda) \circ \Phi_{\widetilde{H}}(rz)=rz, \quad \text{for all} \quad r > 0.
$$

\begin{remarksubsec}{\label{rem3.2.1}}
\normalfont Notice that $\lambda$ is not unique, since $\exp(\lambda + \lambda_0) = \exp(\lambda)$, for all $\lambda_0 \in \liek_{\Z}$.
\end{remarksubsec}

\noindent Conversely, let $z \in \C^n$ be a fixed point of $\exp(\lambda) \circ \Phi_{\widetilde{H}}$ such that $r z \in P_{\K}^{-1}(p)$, for some $r > 0$. Then $r z \in S_p$ is a discriminant point of $\exp(\lambda) \circ \phi_{\overline{h}}$, and therefore $\rho(r z)$ is a discriminant point of 
$$
\phi_{\alpha}^{p(\lambda)} \circ \phi_h.
$$
In other words, \textit{$\rho(pr(z))$ is an $\alpha$-translated point of $\phi_h$} (see figure \ref{fig1}). 

\begin{figure}[H]
\begin{center}
\def\svgwidth{0.55\textwidth}
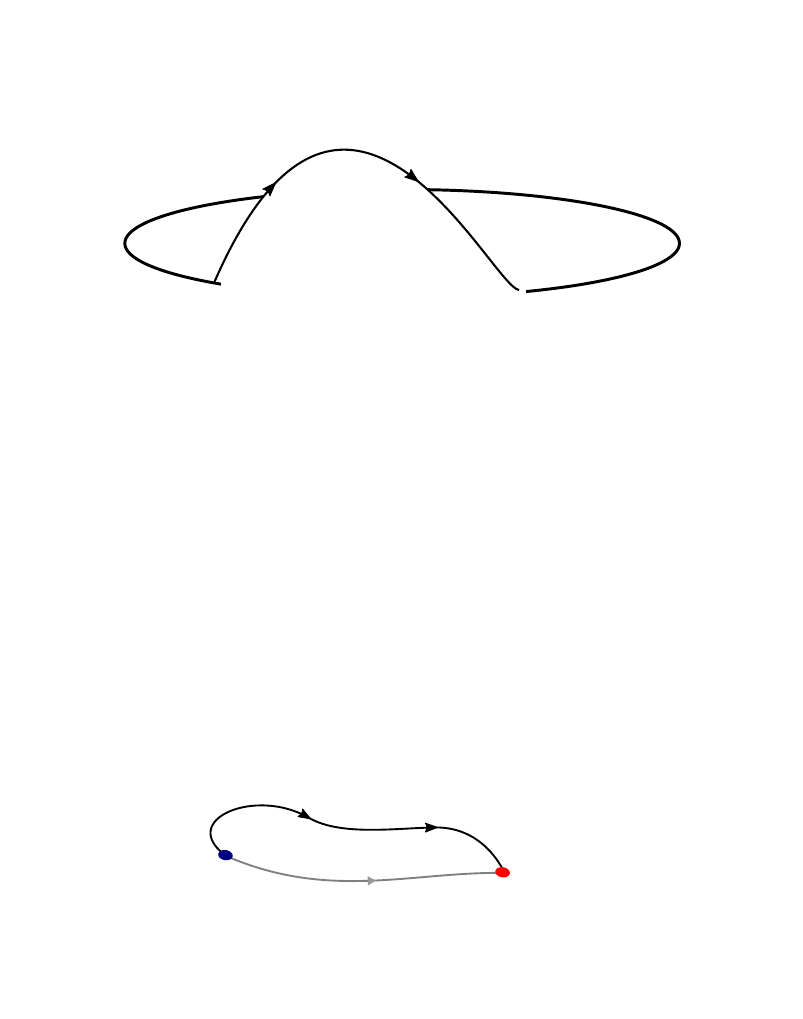
\caption{The translated point $\rho(pr(z))$ of $\phi_h$ corresponds to an $(\R_{>0} \times \K_0)$-family of fixed points of $\exp(\lambda) \circ \Phi_{\widetilde{H}}$ - in blue. The Hamiltonian trajectory $\Phi_{\widetilde{H}}^t(pr(z))$ lifting the contact trajectory $\phi_h^t((\rho(pr(z)))$ does not necessarily preserve the level set $P_{\K}^{-1}(p)$, but rather lies in the cone of $P_{\K}^{-1}(p)$ - in black. The Reeb flow of $\alpha$ lifts to the $\K$-action on $P_{\K}^{-1}(p)$ - in gray.}
\label{fig1}
\end{center}
\end{figure}

\subsection{The generating families}{\label{sec3.3}}

In this section, we introduce the generating families from which we will derive our cohomology groups. We first recall the general construction of a generating function from a Hamiltonian symplectomorphism of $\C^n$, following \cite{Giv95}. We then apply it to the lifts of the previous section. Finally, we add the torus action to this construction and come to the notion of generating family. In contrast to the aforementioned paper, our generating functions are $\K_0$-invariant, and critical points are related to $\alpha$-translated points on $V$, rather than fixed points on $M$.

\subsubsection{General construction}{\label{sec3.3.1}}

Let $H : \C^n \times [0,1] \to \R$ be a time-dependent Hamiltonian, and $\lbrace \phi_H^t \rbrace_{t \in [0,1]}$ be the Hamiltonian isotopy generated by $H$. Dividing the interval $[0,1]$ into an even number of parts, say $2N$, we decompose the time-$1$ map $\phi_H := \phi_H^1$ as follows:
$$
\phi_H = \phi_{2N} \circ \cdots  \circ \phi_1,
$$
where $\phi_j := \phi_H^{\frac{j}{2N}} \circ (\phi_H^{\frac{j-1}{2N}})^{-1}$. If $N$ is big enough so that, for any $z$, $-1$ is not an eigenvalue of $\droit_z \phi_j$, the graph 
$$
Gr_{\phi_j} := \lbrace (z,\phi_j(z)) \ | \ z \in \C^n \rbrace \subset \overline{\C^n} \times \C^n,
$$
projects diffeomorphically onto the diagonal
$$
\Delta := \lbrace (z,z) \ | \ z \in \C^n \rbrace \subset \overline{\C^n} \times \C^n,
$$
where $\overline{\C^n}$ denotes the symplectic vector space $(\C^n, - \omega_{\text{std}})$. The linear symplectomorphism
$$
\begin{array}{lcl}
\Psi : \overline{\C^n} \times \C^n \to (T^* \C^n, - \droit (p\droit q)) \\
\quad \quad (z,w) \mapsto (\frac{z+w}{2}, i(z-w))
\end{array}
$$
sends the diagonal $\Delta$ to the zero-section $0_{\C^n} \subset T^* \C^n$, and therefore $\Psi(Gr_{\phi_j})$ is the graph of a closed $1$-form. This form is exact (either because $H^1(\C^n;\R) = \lbrace 0 \rbrace$ or because $\phi_j$ is Hamiltonian): 
$$
\Psi(Gr_{\phi_j}) = Gr_{\droit \mathcal{H}_j}.
$$
The function $\mathcal{H}_j$ is a \textit{generating function} for the Lagrangian submanifold $\Psi(Gr_{\phi_j})$. In particular, the critical points of $\mathcal{H}_j$ are, through $\Psi$, in one-to-one correspondence with the points of the intersection $Gr_{\phi_j} \cap \Delta$, that is with the fixed points of $\phi_j$.

Consider now the Lagrangian product
$$
Gr_{\phi_1} \times \cdots  \times Gr_{\phi_{2N}} \subset (\overline{\C^n} \times \C^n)^{2N}.
$$
Applying the above to each component and the identification $(T^* \C^n)^{2N} = T^* \C^{2nN}$, we can write
$$
\underset{j=1}{\overset{2N} \prod} \Psi(Gr_{\phi_j}) := \Psi(Gr_{\phi_1}) \times \cdots  \times \Psi(Gr_{\phi_{2N}}) = Gr_{\droit \mathcal{H}},
$$
were 
$$
\begin{array}{lcl}
\mathcal{H} : \C^{2nN} \to \R\\
\ \  (x_1,\hdots ,x_{2N}) \mapsto \underset{j=1}{\overset{2N} \sum} \mathcal{H}_j(x_j).
\end{array}
$$
The critical points of $\mathcal{H}$ are in one-to-one correspondence with the fixed points of the product
$$
\underset{j=1}{\overset{2N} \prod} \phi_j := \phi_1 \times \cdots  \times \phi_{2N} : (\C^n)^{2N} \to (\C^n)^{2N}.
$$
Yet, they do not correspond to fixed points of $\phi_H$, which are rather in one-to-one correspondence with the solutions of the equation
$$
(z_2,\hdots ,z_{2N},z_1) = (\phi_1(z_1),\hdots ,\phi_{2N}(z_{2N})).
$$
The graph $Gr_q \subset (\overline{\C^n})^{2N} \times (\C^n)^{2N} \simeq (\overline{\C^n} \times \C^n)^{2N}$ of the "twisted" cyclic shift \footnote{This twist will be convenient in the sequel in order to ensure the non-degeneracy of a certain quadratic form.} 
$$
\begin{array}{lcl}
q : (\C^n)^{2N} \to (\C^n)^{2N} \\
(z_1,\hdots ,z_{2N}) \mapsto (z_2,\hdots ,z_{2N},-z_1)
\end{array}
$$
corresponds, through $\underset{j=1}{\overset{2N} \prod} \Psi$, to a Lagrangian subvector space of $T^* \C^{2nN}$, which has a generating quadratic form. Since we have decomposed $\phi_H$ into an even number of parts and have a factor $-1$ in the map $q$ above, the graph $Gr_q$ intersects both the multi-diagonal and multi-antidiagonal
$$
(\pm \Delta)^{2N} := \lbrace (z_1,\pm z_1,\hdots ,z_{2N}, \pm z_{2N}) \ | \ z_j \in \C^n \rbrace
$$
only at the origin. The latter are sent, through $\underset{j=1}{\overset{2N} \prod} \Psi$, to the zero-section $0_{\C^{2nN}}$ and the fiber $\lbrace 0 \rbrace \times \C^{2nN} \subset T^* \C^{2nN}$ respectively. Therefore, in $T^* \C^{2nN}$, we can write
$$
\underset{j=1}{\overset{2N} \prod} \Psi (Gr_q) = Gr_{\droit \mathcal{Q}},
$$
where $\mathcal{Q} : \C^{2nN} \to \R$ is a \textit{non-degenerate} quadratic form. The intersection points of $Gr_q$ with $\underset{j=1}{\overset{2N} \prod} Gr_{\phi_j}$ are in one-to-one correspondence with the fixed points of the Hamiltonian symplectomorphism
$$
-Id_{\C^n} \circ \phi_H.
$$
Consider the function
$$
\F^{(N)} : \C^{2nN} \to \R, \quad \F^{(N)} := \mathcal{Q} - \mathcal{H}.
$$
The critical points of $\F^{(N)}$ are in one-to-one correspondence with the fixed points of $- Id_{\C^n} \circ \phi_H$. The function $\F^{(N)}$ is called the \textbf{generating function} associated with the decomposition $\phi_H = \phi_{2N} \circ \cdots  \circ \phi_1$.\\

\subsubsection{Generating functions for the Hamiltonian lifts}{\label{sec3.3.2}}

Let $h : V \times [0,1] \to \R$ be a contact Hamiltonian, and $\widetilde{H} : \C^n \times [0,1] \to \R$ a Hamiltonian lift of $h$. We apply the construction of the previous section to the Hamiltonian symplectomorphism
$$
\exp(\lambda) \circ \Phi_{\widetilde{H}}, \quad \lambda \in \liek
$$
from section \ref{sec3.2.2}. Consider a decomposition
$$
\Phi_{\widetilde{H}} = \Phi_{2N_1} \circ \cdots  \circ \Phi_1
$$
of the Hamiltonian symplectomorphism $\Phi_{\widetilde{H}}$ into $2N_1$ small parts $\Phi_j := \Phi_{\widetilde{H}}^{\frac{j}{2N_1}} \circ (\Phi_{\widetilde{H}}^{\frac{j-1}{2N_1}})^{-1}$, and similarly, a decomposition
$$
\exp(\lambda) = \underbrace{\exp(\frac{\lambda}{2N_2}) \circ \cdots  \circ \exp(\frac{\lambda}{2N_2})}_{2N_2-\text{times}}
$$
of the Hamiltonian symplectomorphism $\exp(\lambda)$. We denote by
$$
\F_{\lambda}^{(N)} := \mathcal{Q} - \mathcal{H}_{\lambda}: \C^{2nN} \to \R, \quad N = N_1+N_2, 
$$
the generating function associated with the decomposition 
$$
\exp(\lambda) \circ \Phi_{\widetilde{H}} = \underbrace{\exp(\frac{\lambda}{2N_2}) \circ \cdots  \circ \exp(\frac{\lambda}{2N_2})}_{2N_2-\text{times}} \circ \Phi_{2N_1} \circ \cdots  \circ \Phi_1.
$$
The function $\mathcal{H}_{\lambda}$ is of the form
$$
\mathcal{H}_{\lambda}(x_1,\hdots ,x_{2N}) := \underset{j=1}{\overset{2N_1} \sum} \mathcal{H}_j(x_j) + \underset{j=2N_1+1}{\overset{2N} \sum} \mathcal{T}_{\lambda}(x_j), \quad x_j \in \C^n,
$$
where $\mathcal{H}_j$ and $\mathcal{T}_{\lambda}$ are the generating functions of $\Phi_j$ and $\exp(\frac{\lambda}{2N_2})$ respectively. Moreover, a direct computation shows that
$$
\mathcal{T}_{\lambda}(x_j) = \underset{k=1}{\overset{n} \sum} \tan (\frac{\pi \lambda_k}{2N_2})|q_j^k|^2, \quad \text{where} \quad x_j = (q_j^1,\hdots ,q_j^n) \in \C^n \quad \text{and} \quad \iota(\lambda) = (\lambda_1,\hdots ,\lambda_n) \in \R^n.
$$

Let us list several properties of the function $\F_{\lambda}^{(N)}$:
\begin{enumerate}
\item the Hamiltonian lift $\widetilde{H}$ is smooth on $\C^n \setminus \lbrace 0 \rbrace$, and is $C^2$ at $0$ only if it is quadratic. However, it is homogeneous of degree $2$, and therefore it is $C^1$ on $\C^n$ with Lipschitz derivative near $0$. Hence, for any $\lambda \in \liek$, $\F_{\lambda}^{(N)}$ is $C^{1,1}$ on $\C^{2nN}$, and smooth on $(\C^n \setminus \lbrace 0 \rbrace)^{2N}$;

\item for any $\lambda \in \liek$, $\F_{\lambda}^{(N)}$ is homogeneous of degree $2$ and $\K_0$-invariant. In particular, the critical points of $\F_{\lambda}^{(N)}$ appear as \textit{$\R_{>0}$-lines of $\K_0$-orbits (or $(\R_{>0} \times \K_0)$-families) in $\C^{2nN}$}, and have \textit{critical value $0$};

\item the function $\F_{\lambda}^{(N)}$ is well-defined as long as $\iota(\lambda) \in (-N_2, N_2)^n \subset \R^n$;

\item{\label{p1}} the family $\lambda \mapsto \F_{\lambda}^{(N)}$ decreases in positive directions: for any $s \in \liek$ such that $\iota(\lambda + s) \in (-N_2, N_2)^n$ and $\iota(s) \in \R_{>0}^n$, we have
$$
\F_{\lambda + s}^{(N)} \leq  \F_{\lambda}^{(N)};
$$

\item for any $\lambda \in \liek$, the critical points of $\F_{\lambda}^{(N)}$ are in one-to-one correspondence with the fixed points of the Hamiltonian symplectomorphism
$$
-Id_{\C^n} \circ \exp(\lambda) \circ \Phi_{\widetilde{H}}.
$$
\end{enumerate}
The symplectomorphism $- Id_{\C^n}$ is Hamiltonian and $\K$-equivariant. It preserves the sphere $S_p$ and the contact form $\alpha_p$, as well as the level set $P_{\K}^{-1}(p)$. Therefore, it projects to a contactomorphism $g \in \text{Cont}_0(V,\xi)$. In particular, any estimate for the number of $\alpha$-translated points of all the compositions of $g$ with contactomorphisms in $\text{Cont}_0(V,\xi)$ will give rise to the same estimate for the contactomorphisms in $\text{Cont}_0(V,\xi)$ themselves. In other words, the twist by $-Id_{\C^n}$, added so that the quadratic form $\mathcal{Q}$ becomes non-degenerate, won't affect the estimation on the number of $\alpha$-translated points.

\subsubsection{Adding the torus action}{\label{sec3.3.3}}

We are looking to count the number of $\alpha$-translated points of the time $1$-map $\phi_h$ on $V$. By the discussions of sections \ref{sec3.2.2} and \ref{sec3.3.2}, we can look for fixed points of the composition
$$
-Id_{\C^n} \circ \exp(\lambda) \circ \Phi_{\widetilde{H}},
$$
for all values of $\lambda$. To that aim, we shall consider $\liek$ as the space of Lagrange multipliers. For any subset $\Lambda_N \subset \liek$ such that $\iota(\Lambda_N) \subset (-N_2,N_2)^n$, we can consider the following function
$$
\F_N : \C^{2nN} \times \Lambda_N \to \R, \quad \F_N(x,\lambda) := \F_{\lambda}^{(N)}(x).
$$
Let $S_N$ denote the unit sphere $S^{4nN-1} \subset \C^{2nN}$, and $F_N$ be the restriction of $\F_N$ to $S_N \times \Lambda_N$. We call $F_N$ (resp. $\F_N$) the \textbf{generating family} (resp. \textbf{homogeneous generating family}) associated with the decomposition $\Phi_{\widetilde{H}} = \Phi_{2N_1} \circ \cdots  \circ \Phi_1$. In the sequel, \textit{we fix this decomposition, and for any $N > N_1$, we take $\Lambda_N$ to be a cube in $\liek$ with boundary $\partial \Lambda_N$, centered at the origin, of fixed size growing linearly with $N$, and such that} 
$$
\underset{N}{\cup} \Lambda_N = \liek.
$$

\begin{remarksubsec} 
\normalfont In the sequel we will study the equivariant homotopy type of sublevel sets of the generating family relatively to the boundary $\partial \Lambda_N$. When studying the regularity of $\F_N$ and the critical point sets of the functions $\F_{\lambda}^{(N)}$, we will keep in mind that they are actually defined for all $\lambda \in \liek$ such that $\iota(\lambda) \in (-N_2,N_2)^n$.
\end{remarksubsec}

We are looking for critical points of the functions $\F_{\lambda}^{(N)}$ that lie in $\R_{>0}$-lines that, through $\Psi$, intersect the level set $P_{\K}^{-1}(p)$. By homogeneity, any critical point of $\F_{\lambda}^{(N)}$ has critical value $0$, and moreover the zero-set $F_N^{-1}(0) \subset S_N \times \Lambda_N$ is $\K_0$-invariant. Consider the function
$$
\begin{array}{lcl}
\hat{p}_N : F_N^{-1}(0) \to \R.\\
\quad \quad \ (x,\lambda) \mapsto p(\lambda)
\end{array}
$$
It is $\K_0$-invariant. Recall that we have denoted by $g$ the contactomorphism of $\text{Cont}_0(V,\xi)$ induced by the restriction $-Id_{\C^n | P_{\K}^{-1}(p)}$. We have the following contact analogue of \cite[Proposition $4.3$]{Giv95}.
\begin{propsubsec}{\label{prop3.3.1}}
$0$ is a regular value of $F_N$, and to any critical $\K_0$-orbit of $\hat{p}_N$, there corresponds an $\alpha$-translated point of $g \circ \phi_h$.
\begin{proof}
Notice first that by homogeneity of $\F_N$, we have $\droit_x \F_{\lambda}^{(N)} = \droit_x F_{\lambda}^{(N)}$ for all $x \in (F_{\lambda}^{(N)})^{-1}(0)$. Let $(x,\lambda) \in F_N^{-1}(0)$ be a critical point of $\F_N$. Then $x \in S_N$ is a critical point of $\F_{\lambda}^{(N)}$. On $\C^n$, it corresponds to a fixed point $z$ of the decomposition
$$
- Id_{\C^n} \circ \underbrace{\exp(\frac{\lambda}{2N_2}) \circ \cdots  \circ \exp(\frac{\lambda}{2N_2})}_{2N_2 \text{ times}} \circ \Phi_{2N_1} \circ \cdots  \circ \Phi_1.
$$
Let us denote by $(z_1,\hdots ,z_{2N} = -z_1)$ the corresponding discrete trajectory in $\C^{2nN}$, that is $z_1=z$, and $z_j$ is obtained by applying the $(j-1)$-th symplectomorphism of the above decomposition to $z_{j-1}$. Choose coordinates on $\liek \simeq \R^k$, and write $\lambda = (\lambda_1,\hdots ,\lambda_k)$. By the Hamilton-Jacobi equation, the derivative of $\F_N$ in $\lambda$ is given by minus the Hamiltonian associated with the infinitesimal action of $\exp(\lambda)$ on the $2N_2$ last coordinates of $x$. Through $\Psi$, this means that
$$
\frac{\partial \F_N}{\partial \lambda_j} (x,\lambda) = - P_{\K}^j(z_{2N_1+1}) - \cdots  - P_{\K}^j(z_{2N}),
$$ 
where $P_{\K} = (P_{\K}^1,\hdots ,P_{\K}^k) : \C^n \to \R^{k*}$. On the other hand, we have $z_{2N_1+i} = \exp(\frac{\lambda}{2N_2})z_{2N_1+i-1}$, for all $i = 1,\hdots ,2N_2$. Since $P_{\K}^j$ is $\K$-invariant, we obtain
$$
\frac{\partial \F_N}{\partial \lambda_j} (x,\lambda) = - N P_{\K}^j(z_{2N_1+1}).
$$
If $(x,\lambda)$ is a critical point of $\F_N$, then $P_{\K}^j(z_{2N_1+1}) = 0$, for all $j=1,\hdots ,k$. In particular, this means that $P(z_{2N_1+1}) \in \ker \iota^*$, and by compactness of the toric manifold $(M, \omega, \mathbb{T})$, this is possible only if $z_{2N_1+1} = 0$. By homogeneity of the symplectomorphisms in the decomposition above, this implies that $z = 0$, and thus $x=0$, which is impossible on $S_N$. Thus $0$ is a regular value of $F_N$.

By the method of Lagrange multipliers, the critical points of $\hat{p}_N$ are those points $(x,\lambda) \in F_N^{-1}(0)$ such that $x$ is a critical point of $\F_{\lambda}^{(N)}$ and $\frac{\partial \F_N}{\partial \lambda_j} (x,\lambda)$ is proportional to $p_j$, where $p = (p_1,\hdots ,p_k) \in \R^{k*}$. By the discussion above, this means that $x$ corresponds to a discrete trajectory $(z_1,\hdots ,z_{2N} = -z_1)$ satisfying 
$$
P_{\K}^j(z_{2N_1+1}) \sim p_j, \text{ for all $j=1,\hdots ,k$}.
$$
But then, $P_{\K}^j(z_{2N}) = P_{\K}^j(z) \sim p_j$, that is $z$ lies in an $\R_{>0}$-line that intersects $P_{\K}^{-1}(p)$.
\end{proof}
\end{propsubsec}

\subsection{From a decomposition to another}{\label{sec3.4}}

Our cohomology group is defined as a limit in $N \to \infty$ of equivariant cohomology groups associated with the generating families $F_N$. Therefore we must describe how $F_N$ changes when $N$ grows. The key ingredient is that we can deform the generating functions in a controllable way, as long as the front of the associated generating family remains unchanged during the deformation. Recall that the front of $F_N$ is given by: 
$$
L_N := \lbrace \lambda \in \Lambda_N \ | \ 0 \text{ is a singular value of } F_{\lambda}^{(N)} \rbrace.
$$
By construction of $F_N$, it will remain unchanged as long as the time-$1$ map $\Phi_{\widetilde{H}}$ remains unchanged as well. We closely follow \cite{Giv95}, and begin with an observation. Let $\Phi_H = \Phi_{2N'} \circ \hdots  \circ \Phi_1$ be the decomposition of a Hamiltonian isotopy $\lbrace \Phi_H^t \rbrace_{t \in [0,1]}$, such that the first $2(N'-N)$ parts consist of a loop
$$
Id_{\C^n} = \Phi_{2(N'-N)} \circ \cdots  \circ \Phi_1.
$$
We relate the generating function $\F^{(N')}$ associated with the whole decomposition
$$
\Phi_H = \Phi_{2N'} \circ \cdots  \circ \Phi_1,
$$
to the generating functions $\F^{(N)}$ and $\G^{(N'-N)}$ associated with the parts
$$
\Phi_{2N'} \circ \cdots  \circ \Phi_{2(N'-N)+1} \quad \text{and} \quad \Phi_{2(N'-N)} \circ \cdots  \circ \Phi_1
$$
respectively. Consider the following deformation of the graph $Gr_q$ from section \ref{sec3.3.1}:
$$
\begin{array}{lcl}
Q_{\epsilon} := \bigg \lbrace (z_1,w_1,\hdots ,z_{2N'},w_{2N'}) \ |\  (z_j,w_j) \in \overline{\C^n} \times \C^n, & \quad w_j = z_{j+1} \text{ for } j \notin \{ 2(N'-N), 2N' \} \\ &  w_{2(N'-N)} = \epsilon(z_{2(N'-N)+1} - w_{2N'}) - z_1 \\ & w_{2N'} = \epsilon (w_{2(N'-N)} - z_1) - z_{2(N'-N)+1} \bigg \rbrace,
\end{array}
$$
where $\epsilon \in [0,1]$. Then $Q_{\epsilon}$ is a Lagrangian subspace of $(\overline{\C^n} \times \C^n)^{2N'}$, which is transversal both to the multi-diagonal $\Delta^{2N'}$ and the multi-antidiagonal $(-\Delta)^{2N'}$. Thus, it corresponds in $T^* \C^{2N'}$ to the graph of a non-degenerate quadratic form $\mathcal{Q}_{\epsilon}$. Consider the generating function
$$
\F^{(N')}_{\epsilon} := \mathcal{Q}_{\epsilon} - \mathcal{H}.
$$
Its critical points are in one-to-one correspondence with the points of the intersection
$$
\underset{j=1}{\overset{2N'} \prod} \Psi(Gr_{\phi_j}) \cap Q_{\epsilon}.
$$
For $\epsilon = 0$, $Q_{\epsilon}$ is the product $Q^1 \times Q^2$, where $Q^1$ and $Q^2$ correspond to the twisted cyclic shifts in  $(\C^n)^{2N}$ and $(\C^n)^{2(N'-N)}$ respectively. Since $\mathcal{H}$ is the direct sum of the generating functions associated with the small Hamiltonian symplectomorphisms $\Phi_j$ for $j = 1,\hdots ,2N'$, we get
$$
\F^{(N')}_0 = \F^{(N)} \oplus \G^{(N'-N)} : \C^{2nN} \times \C^{2n(N'-N)} \to \R.
$$
Moreover, $\F^{(N')}_0$ admits a critical point if and only if $\F^{(N)}$ does, if and only if $\F^{(N')}$ does. For $\epsilon = 1$, we have $Q_{\epsilon} = Gr_q$, so that
$$
\F^{(N')}_1 = \F^{(N')}.
$$
Finally, for $\epsilon \in (0,1)$, notice that
$$
\begin{array}{lcl}
(z_1,w_1,\hdots ,z_{2N'},w_{2N'}) \in \underset{j=1}{\overset{2N'} \prod} \Psi(Gr_{\phi_j}) \cap Q_{\epsilon}  \iff \\  (\frac{1}{\epsilon}z_1,\frac{1}{\epsilon}w_1,\hdots ,\frac{1}{\epsilon}z_{2(N'-N)},\frac{1}{\epsilon}w_{2(N'-N)},z_{2(N'-N)+1},w_{2(N'-N)+1},\hdots ,z_{2N'},w_{2N'}) \in \underset{j=1}{\overset{2N'} \prod} \Psi(Gr_{\phi_j}) \cap Gr_q,
\end{array}
$$
where $q$ is the twisted cyclic shift in $\C^{2nN'}$. In other words, the critical points of $\F_{\epsilon}^{(N')}$ are in one-to-one correspondence with the critical points of $\F^{(N')}$, for all $\epsilon \in [0,1]$.

Observe now that given two decompositions 
\begin{align*}
\Phi_H^1 & = \Phi_{2N} \circ \cdots  \circ \Phi_1,\\
& = \Phi'_{2N'} \circ \cdots  \circ \Phi'_1
\end{align*}
of the same Hamiltonian isotopy $\lbrace \Phi_H^t \rbrace_{t \in [0,1]}$ of $\C^n$ with, say, $N' > N$, one can always use a reparametrization $H_t^s$ of $H_t$, so that $H_t^0 = H_t$, and $H_t^1$ generates the Hamiltonian symplectomorphism
$$
\Phi_{H^1}^1 = \Phi_{2N} \circ \cdots  \circ \Phi_1 \circ \underbrace{Id_{\C^n} \circ \cdots  \circ Id_{\C^n}}_{2(N'-N)-\text{times}}.
$$
In particular, consider the homogeneous generating family $\F_N$ associated with the decomposition 
$$
\Phi_{\widetilde{H}} = \Phi_{2N_1} \circ \cdots  \circ \Phi_1,
$$
where $\widetilde{H}$ is a Hamiltonian lift of the contact Hamiltonian $h$. We will denote by
$$\G_{\lambda}^{(K)} : \C^{2nK} \to \R$$ 
the generating function associated with the decomposition
$$
\underbrace{\exp(\frac{\lambda}{2K}) \circ \cdots  \circ \exp(\frac{\lambda}{2K})}_{2K-\text{times}},
$$
as defined in section \ref{sec3.3.1}. Let us also denote by $F_{N+K | \Lambda_N}$ (resp. $\F_{N+K | \Lambda_N}$) the restriction of the generating family $F_{N + K} : S_{N+K} \times \Lambda_{N+K} \to \R$ (resp. the homogeneous generating family $\F_{N + K} : \C^{2n(N+K)} \times \Lambda_{N+K} \to \R$) to $S_{N+K} \times \Lambda_N$ (resp. $\C^{2n(N+K)} \times \Lambda_N$). Note that the front of $F_{N+K|\Lambda_N}$ is equal to the front of $F_N$, that is we have $L_{N+K} \cap \Lambda_N = L_N$. 
% In the next proposition, by fiberwise, we mean in the fibers direction of the projection $\C^{n(N+K)} \times \Lambda_N \to \Lambda_N$.
\begin{propsubsec}{\label{prop3.4.1}}
There exists a homogeneous of degree $2$ and $\K_0$-invariant fiberwise $C^{1,1}$ homotopy between the restricted homogeneous family
$$
\F_{N+K |\Lambda_N}
$$
and the fiberwise direct sum
$$
\F_N \oplus_{\Lambda_N} \G_0^{(K)} : \C^{2n(N+K)} \times \Lambda_N \to \R,
$$
in a way that the front of the corresponding generating families on $S_{N+K} \times \Lambda_N$ remains unchanged during the deformation.
\begin{proof}
For any $\lambda \in \Lambda_N$, the generating functions $\F_{\lambda}^{(N+K)}$ and $\F_{\lambda}^{(N)}$ are associated respectively with the decompositions
$$
\begin{array}{lcl}
\exp(\lambda) \circ \Phi_{\widetilde{H}} = \underbrace{\exp(\frac{\lambda}{2(N_2+K)}) \circ \cdots  \circ \exp(\frac{\lambda}{2(N_2+K)})}_{2(N_2+K)-\text{times}} \circ \Phi_{2N_1} \circ \cdots  \circ \Phi_1\\ \\ \quad \quad \quad \quad \quad \quad \quad \quad \quad \quad \quad \quad \quad \text{ and} \\ \\
\exp(\lambda) \circ \Phi_{\widetilde{H}} = \underbrace{\exp(\frac{\lambda}{2N_2}) \circ \cdots  \circ \exp(\frac{\lambda}{2N_2})}_{2N_2-\text{times}} \circ \Phi_{2N_1} \circ \cdots  \circ \Phi_1.
\end{array}
$$
Up to a reparametrization of the $2(N_2+K)$ last factors, the first decomposition becomes
\begin{align*}
\begin{array}{lcl}
\exp(\lambda) \circ \Phi_{\widetilde{H}} & = 
\underbrace{\exp(\frac{\lambda}{2N_2}) \circ \cdots  \circ \exp(\frac{\lambda}{2N_2})}_{2N_2-\text{times}} \circ \underbrace{Id_{\C^n} \circ \cdots  \circ Id_{\C^n}}_{2K-times} \circ \Phi_{2N_1} \circ \cdots  \circ \Phi_1\\ 
& = \underbrace{\exp(\frac{\lambda}{2N_2}) \circ \cdots  \circ \exp(\frac{\lambda}{2N_2})}_{2N_2-\text{times}} \circ \Phi_{2N_1} \circ \cdots  \circ \Phi_1 \circ \underbrace{Id_{\C^n} \circ \cdots  \circ Id_{\C^n}}_{2K-times}.
\end{array}
\end{align*}

By the discussion above, the generating function associated with this last decomposition is $\K_0$-invariantly homotopic to $\F_{\lambda}^{(N)} \oplus \G_0^{(K)}$, and this homotopy does not change the front of the family. Thus, we obtain a homogeneous of degree $2$ and $\K_0$-invariant $C^{1,1}$ homotopy between $\F_{\lambda}^{(N+K)}$ and $\F_{\lambda}^{(N)} \oplus \G_0^{(K)}$. Since the time-$1$ map $\exp(\lambda) \circ \Phi_{\widetilde{H}}$ remains unchanged during the reparametrization, the front remains unchanged during the whole process. 
\end{proof}
\end{propsubsec}

\begin{remarksubsec}{\label{rem3.4.1}}
\normalfont Notice that the result above is independent of the reparametrization, since any two such reparametrizations are always homotopic. Moreover, one can show by a similar argument that the front of the generating family $F_N$ remains unchanged if one modifies the decomposition of $\Phi_{\widetilde{H}}$ \textit{into $2N_1$ parts}.
\end{remarksubsec}

\subsection{Sublevel sets and transversality}{\label{sec3.5}}

In this section, we introduce the sublevel sets which will be used to define the cohomology groups. Similarly as with the generating families, we describe how they behave when $N$ grows, using the results from section \ref{sec2.1}. In particular, we define a basis of $\C^{2nN}$ in which the quadratic generating functions associated with the torus action are diagonal, so that we have a canonical (independent of the given element of the torus) identification between their non-positive sublevel sets and their non-positive eigenspaces. Consider the following sublevel sets
$$
\F_N^- := \lbrace \F_N \leq 0 \rbrace, \quad F_N^- := \lbrace F_N \leq 0 \rbrace.
$$
For any $\nu \in \R$, we denote by $\Gamma_N(\nu)$ the intersection $\Lambda_N \cap p^{-1}(\nu)$, and define the sets
$$
\begin{array}{lcl}
\F_N^-(\nu) := \F_N^- \cap (\C^{2nN} \times \Gamma_N(\nu)), &
\partial \F_N^-(\nu) := \F_N^-(\nu) \cap (\C^{2nN} \times \partial \Gamma_N(\nu)), \\ \\
F_N^-(\nu) := F_N^- \cap (S_N \times \Gamma_N(\nu)), & 
\partial F_N^-(\nu) := F_N^-(\nu) \cap (S_N \times \partial \Gamma_N(\nu)),
\end{array}
$$
where
$$
\partial \Gamma_N(\nu) := \Gamma_N(\nu) \cap \partial \Lambda_N.
$$
Recall that the generating function $\F_{\lambda}^{(N)}$ is smooth on $(\C^n \setminus \lbrace 0 \rbrace)^{2N}$. If $x$ is a critical point of $\F_{\lambda}^{(N)}$, it corresponds, through the linear symplectomorphism $\Psi$ of section \ref{sec3.3.1}, to a solution of the equation
$$
(z_2,\hdots ,z_{2N}, -z_1) = (\gamma_1(z_1), \hdots , \gamma_{2N}(z_{2N})),
$$
where $\gamma_j = \Phi_j$ if $j \leq 2N_1$, and $\gamma_j = \exp(\frac{\lambda}{2N_2})$ otherwise. If $x$ lies on the coordinate cross, there exists $j$ such that $\gamma_j(z_j) = -z_j$. Since $\gamma_j$ is close to $Id_{\C^n}$, this can happen only if $z_j = 0$. Moreover, $\gamma_j(0) = 0$ for all $j$, and therefore $x = 0$. Thus $\F_{\lambda}^{(N)}$ is smooth at any non-zero critical points. In particular, \textit{the front $L_N$ is of zero-measure}.

Recall that the spectrum of a contactomorphism is the set of Reeb shifts of all its $\alpha$-translated points (equation $(\ref{eq4})$). The relation between the front $L_N$ and the spectrum $\Spec(g \circ \phi_h)$ of $g \circ \phi_h$ shall be understood as follows: the front $L_N$ is made of all the elements $\lambda \in \Lambda_N$ such that the Hamiltonian diffeomorphism
$$
-Id_{\C^n} \circ \exp(\lambda) \circ \Phi_{\widetilde{H}} : \C^n \to \C^n 
$$
admits a fixed point. However, this fixed point might not correspond to a point on $V$ (if it does not lie in an $\R_{>0}$-line that intersects $P_{\K}^{-1}(p)$), and therefore in particular to an $\alpha$-translated point of $g \circ \phi_h$. In contrast, the spectrum of $g \circ \phi_h$ is made of real numbers $\nu \in \R$ for which an $\alpha$-translated point with Reeb shift $\nu$ appears on $V$. By Propositions \ref{prop2.1.1} and \ref{prop3.3.1}, this happens if and only if there exist $N$ and $\lambda \in L_N$ such that $p(\lambda) = \nu$ and $\Gamma_N(\nu)$ is tangent to $L_N$ at $\lambda$. In other words, elements of $L_N$ correspond to fixed points on $\C^n$, whereas non-transverse intersections between $\Gamma_N(\nu)$ and $L_N$ correspond to $\alpha$-translated points of $\phi_h$ with Reeb shift $\nu$.

Note however that $\nu \notin \Spec(g \circ \phi_h)$ does not mean that the boundary $\partial \Gamma_N(\nu)$ is transversal to $L_N$. By Corollary \ref{cor2.1.1}, the set of $\nu$ such that $\Gamma_N(\nu)$ and $\partial \Gamma_N(\nu)$ are transversal to the front $L_N$ is of full measure. We will say that $\nu \in \R$ \textbf{is generic} if $\Gamma_N(\nu)$ \textit{and} $\partial \Gamma_N(\nu)$ are transversal to $L_N$ for all $N$ (we use here that a countable intersection of sets of full measure is of full measure). Note that this notion depends on $\Phi_{\widetilde{H}}$ and the family of cubes $\{ \Lambda_N \}_N$.

We have the following "homogeneous" version of \cite[Proposition $5.1$]{Giv95}.  
\begin{propsubsec}{\label{prop3.5.1}}
Let $\F : \C^M \to \R$ be a homogeneous of degree $2$ function, and $\hat{\F} : \C^{M+1} \to \R$ be its suspension
$$
\hat{\F}(x,z) := \F(x) + |z|^2.
$$
We denote by $\F^{\pm}$ and $\hat{\F}^{\pm}$ the sets
$$
\F^{\pm} := \lbrace x \in \C^M \ | \ \F(x) \geq 0 \ (\text{resp. $\leq 0$}) \rbrace, \quad \quad \hat{\F}^{\pm} := \lbrace (x,z) \in \C^{M+1} \ | \ \hat{\F}(x) \geq 0 \ (\text{resp. $\leq 0$}) \rbrace.
$$
Then there exist natural $\R_{>0}$-equivariant homotopy equivalences
$$
\hat{\F}^- \simeq \F^-, \quad \hat{\F}^+ \simeq \F^+ \times \C.
$$
Moreover, if $\F$ is invariant relatively to an $S^1$-action on $\C^M$, then the above homotopy equivalences can be made equivariant with respect to the product of the diagonal $S^1$-action on $\C^M$ with the standard $S^1$-action on $\C$. If $\F$ depends continuously on additional parameters, then the homotopy equivalences depend continuously on them.

\begin{proof}
Let us first consider the function 
$$
\hat{\F} : \C^M \times \R_{\geq 0} \to \R, \quad (x,r) \mapsto \F(x) + r^2
$$
and prove an analogue of the above statement in this setting, that is there exist natural $\R_{>0}$-equivariant homotopy equivalences
$$
\hat{\F}^- \simeq \F^- \quad \text{and} \quad \hat{\F}^+ \simeq \F^+ \times \R_{\geq 0}.
$$
The meridional contraction from the North pole $P = (0,1)$ of the unit sphere $S^{2M+1}$ preserves $\hat{\F}^- \cap (S^{2M-1} \times \lbrace 0 \rbrace)$, and therefore $\hat{\F}^- \cap S^{2M+1}$ retracts onto $\hat{\F}^- \cap (S^{2M-1} \times \lbrace 0 \rbrace)$. Extending this contraction homogeneously provides a deformation retraction from $\hat{\F}^-$ to $\F^-$ as well:
$$
\hat{\F}^- \simeq \F^-.
$$
Moreover, each meridional arc from the North pole to $\F^+ \cap S^{2M-1}$ lies in $\hat{\F}^+ \cap S^{2M+1}$, and therefore the same contraction provides a homotopy equivalence of pairs:
$$
(\hat{\F}^+, \F^+) \simeq (\C^M \times \R_{\geq 0}, \F^+).
$$
Now, the pairs $(\C^M \times \R_{\geq 0},\F^+)$ and $(\F^+ \times \R_{\geq 0},\F^+)$ are naturally homotopy equivalent, and therefore we have 
$$
(\hat{\F}^+, \F^+) \simeq (\F^+ \times \R_{\geq 0}, \F^+).
$$
Figure \ref{fig2} illustrates this deformation.

For the general case $z \in \C$, we view $\C^M \times \C$ as the quotient $\C^M \times \R_{\geq 0} \times S^1 / \sim$, with the identification $\C^M \times \lbrace 0 \rbrace \times S^1 \sim \C^M$. Then $\F^{\pm}$ and $\hat{\F}^{\pm}$ are simply given by their restrictions to $\C^M \times \R_{\geq 0}$ multiplied by $S^1$, with the relevant identifications. We obtain  
$$
\hat{\F}^- \simeq \F^-, \quad \hat{\F}^+ \simeq \F^+ \times \C.
$$
Since all the homotopies from above are carried out in a canonical way, they respect group actions and parametric dependence. 
\end{proof}
\end{propsubsec}

\begin{figure}[H]
\begin{center}
\def\svgwidth{0.6\textwidth}
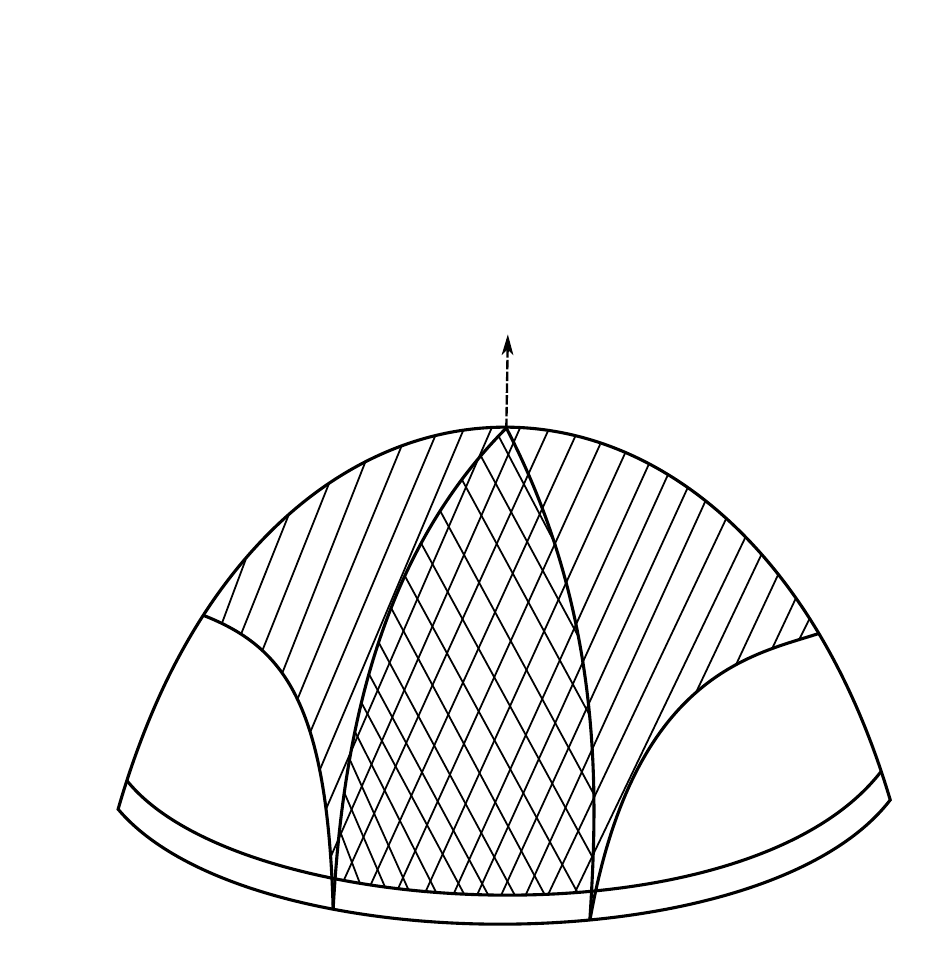
\caption[The meridional contraction from the North pole]{The meridional contraction from the North pole $P$ - in blue.}
\label{fig2}
\end{center}
\end{figure}

Recall that $\G_{\lambda}^{(N)}$ denotes the generating function of the decomposition
$$
\underbrace{\exp(\frac{\lambda}{2N}) \circ \cdots \circ \exp(\frac{\lambda}{2N})}_{2N-\text{times}},
$$
as defined in section \ref{sec3.3.1}. It is a quadratic form. Let us denote by $\G_{\lambda}^{(N) -}$ its non-negative eigenspace, after the following diagonalization.
\begin{lemmasubsec}{\label{lem3.5.1}}
There exist vectors $v_j^k$ in $\C^{2nN}$, and a linear isomorphism
$$
\C^{2nN} \simeq \underset{j,k}{{\bigoplus}} \C v_j^k, \quad j=1,\hdots ,n, \quad k=-N,\hdots ,N-1,
$$
such that:
\begin{enumerate}
\item the isomorphism is equivariant with respect to the diagonal action of the maximal torus $\mathbb{T}^n = \R^n / \Z^n$ on $\underset{2N-\text{times}}{\underbrace{\C^n \times \cdots \times \C^n}}$ (acting linearly on each factor $\C^n$) and the $\mathbb{T}^n$-action on $\underset{j,k}{{\bigoplus}} \C v_j^k$ defined as follows: for any $k$, $\mathbb{T}^n$ acts on the complex line generated by $v_j^k$ via its standard character $\chi_j : (e^{i\theta_1},\hdots ,e^{i\theta_n}) \mapsto e^{i\theta_j}$. In particular, the direct sum $\underset{k}{\oplus}\C v_j^k$ is $\mathbb{T}^n$-invariant;
\item for any $\lambda$, the quadratic form $\G_{\lambda}^{(N)}$ is diagonal in the basis $(v_j^k)_{j,k}$, and we have
$$
\dim(\G_{\lambda}^{(N)-}) = \underset{j=1}{\overset{n} \sum}2(N + \lfloor \lambda_j + \frac{1}{2} \rfloor), \quad \text{where} \quad \iota(\lambda) = (\lambda_1,\hdots ,\lambda_n) \in \Lambda_N,
$$
where $\lfloor . \rfloor$ denotes the integer part.
\end{enumerate}
\begin{proof}
Recall that for any $\lambda \in \liek \subset \R^n$, the quadratic form $\G_{\lambda}^{(N)}$ is of the form
$$
\G_{\lambda}^{(N)} = \mathcal{Q} - \mathcal{H}_{\lambda},
$$
where $\mathcal{H}_{\lambda} = \underset{j=1}{\overset{2N}{\bigoplus}} \mathcal{T}_{\lambda}$, and $\mathcal{T}_{\lambda}$ is the generating function of $\exp(\frac{\lambda}{2N})$. We have 
$$
\mathcal{H}_{\lambda}(q) = \underset{i=1}{\overset{n} \sum} \tan(\frac{\pi \lambda_i}{2N}) (|q_1^i|^2 + \cdots  + |q_{2N}^i|^2),
$$ 
where $q = (q_1^1,\hdots ,q_1^n,\hdots ,q_{2N}^1,\hdots ,q_{2N}^n) \in \C^{2nN}$, and moreover, a direct calculation shows that the non-degenerate quadratic form $\mathcal{Q}$ writes $\mathcal{Q}(q) = \langle C q, q \rangle$, where
$$
C = i(Id-A)(Id+A)^{-1}, \quad \quad A = \begin{bmatrix}
 
0&1& & & & &\\
 &.&.& & & &\\
 & &.&.& & &\\
 & & &.&.& &\\
 & & & &.&.&\\
 & & & & &.&1\\
-1 & & & & & & 0
\end{bmatrix} \in M_{2N \times 2N}(\C^{n}).
$$
For any $k = -N,\hdots ,N-1$, and $j=1,\hdots ,n$, we define the following vector:
$$
X_j^k = (e_j,e^{i\frac{(2k+1)\pi}{2N}} e_j, \hdots , e^{i(2N-1)\frac{(2k+1)\pi}{2N}} e_j) \in (\C^{n})^{2N},
$$
where $e_j = (0,\hdots ,0,1,0,\hdots ,0)$ is the $j$-th standard vector in $\C^n$. The reader may check that the following points hold:
\begin{itemize}
\item we have $i(Id-A)X_j^k = \tan(\frac{(2k+1)\pi}{4N})(Id+A)X_j^k$, and therefore the family $\lbrace v_j^k := (Id+A)(X_j^k) \rbrace_{j,k}$ is a $\C$-basis of $\C^{2nN}$ made of eigenvectors of $C$, with eigenvalues $\tan(\frac{(2k+1)\pi}{4N})$;
\item the maximal torus $\mathbb{T}^n$ acts on the complex line $\C v_j^k$ via the character $\chi_j$. In particular, $\C v_j^k$ is $\mathbb{T}^n$-invariant;
\end{itemize}
Put $V_j = \underset{k=-N}{\overset{N-1}{\oplus}} \C v_j^k$. Since $\mathcal{Q}$ is $\mathbb{T}^n$-invariant, each two lines $\C v_i^k$ and $\C v_j^l$ are orthogonal whenever $i \neq j$. It remains to diagonalize the restrictions $\mathcal{Q}_{|V_j}$. The action of $\mathbb{T}^n$ on the basis $(v_j^{-N},\hdots ,v_j^{N-1})$ is given by $\chi_j$. Since it is diagonal and linear, one can find a new basis (still denoted by $(v_j^{-N},\hdots ,v_j^{N-1})$) of $V_j$, on which $\mathbb{T}^n$ acts via the character $\chi_j$, and such that $C_{| V_j}$ is diagonal. It remains to concatenate these bases for $j=1,\hdots ,n$. 

For the second point, the spectrum of $\mathcal{Q}$ is given by
$$
\Spec(\mathcal{Q}) = \lbrace \tan(\frac{\pi (2k+1)}{4N}) \ | \  -N \leq k \leq N-1 \rbrace,
$$
and therefore the spectrum of $\G_{\lambda}^{(N)}$ is given by
$$
\Spec(\G_{\lambda}^{(N)}) = \lbrace \tan(\frac{\pi (2k+1)}{4N}) - \tan(\frac{\pi \lambda_j}{2N}) \ | \  -N \leq k \leq N-1, \  j=1,\hdots ,n \rbrace.
$$
We have
$$
\tan(\frac{\pi (2k+1)}{4N}) - \tan(\frac{\pi \lambda_j}{2N}) \leq 0 \iff k \leq \lambda_j - \frac{1}{2}.
$$
There are $N + \lfloor \lambda_j + \frac{1}{2} \rfloor$ such $k$'s, and therefore the real dimension of $\G_{\lambda,N}^-$ is:
$$
\dim(\G_{\lambda}^{(N)-}) = \underset{j=1}{\overset{n}\sum}2(N + \lfloor \lambda_j + \frac{1}{2} \rfloor),
$$
as claimed.
\end{proof}
\end{lemmasubsec}
This change of basis is canonical, that is it depends only on $N$ and the \textit{non-degenerate} quadratic form $\mathcal{Q}$. Consider the direct sum $\F_N \oplus \G_0^{(K)}$. The non-degenerate quadratic form $\G_0^{(K)}$ has $2nK$ negative eigenvalues, and is diagonal in the basis of the above proposition (note here the importance of the non-degeneracy of $\mathcal{Q}$). Applying Proposition \ref{prop3.5.1} multiple times, we get
\begin{corsubsec}{\label{cor3.5.1}}
There exists an $(\R_{>0} \times \K_0)$-equivariant homotopy equivalence
$$
\lbrace \F_N \oplus \G_0^{(K)} \leq 0 \rbrace \simeq \F_N^- \times \G_0^{(K)-} \simeq \F_N^- \times \C^{nK}.
$$
\end{corsubsec}
Applying Proposition \ref{prop3.4.1} along with Proposition \ref{prop2.1.2}, we get
\begin{propsubsec}{\label{prop3.5.2}}
If $\nu$ is generic, there exists an $(\R_{>0} \times \K_0)$-equivariant fiberwise homotopy equivalence
$$
(\F^-_{N+K | \Gamma_N(\nu)}, \F^-_{N+K | \partial \Gamma_N(\nu)}) \simeq (\F^-_N(\nu) \times \C^{nK}, \partial \F^-_N(\nu) \times \C^{nK}),
$$
where we have denoted by $\F_{N+K | \Gamma_N(\nu)}^-$ and $\F_{N+K | \partial \Gamma_N(\nu)}^-$ the restrictions of $\F_{N+K}^-$ to $\C^{2n(N+K)} \times \Gamma_N(\nu)$ and $\C^{2n(N+K)} \times \partial \Gamma_N(\nu)$ respectively.
\end{propsubsec}

\subsection{The cohomology groups}{\label{sec3.6}}

In this section we study the equivariant cohomology of the sublevel sets introduced in the previous section, and come to the definition of our cohomology group by taking a limit in $N \to \infty$. In a first step we use the identification of section \ref{sec2.2} for conical spaces in order to define a natural homomorphism from which we derive our limit. The latter comes along with a natural so-called augmentation map, as well as with several algebraic structures, namely the action of coefficient rings in equivariant cohomology, as well as a Novikov action of $H_2(M;\Z)$. We interpret the latter on the kernel of the augmentation map. 

As above, we consider the homogeneous generating family
$$
\F_N : \C^{2nN} \times \Lambda_N \to \R,
$$
associated with a decomposition of the Hamiltonian symplectomorphism $\Phi_{\widetilde{H}}$, where $\widetilde{H}$ is a Hamiltonian lift of a contact Hamiltonian $h$ of $V$, and we fix a generic $\nu \in \R$. We look at the $\K_0$-equivariant cohomology groups
$$
H_{\K_0}^*(F_N^-(\nu),\partial F_N^-(\nu)).
$$
Using the notations from section \ref{sec2.2}, we work with the following short exact sequence
$$
\begin{array}{lcl}
0 \longrightarrow C_{\K_0,c}^*(\F_N^-(\nu), \partial \F_N^-(\nu)) \longrightarrow C_{\K_0}^*(\F_N^-(\nu), \partial \F_N^-(\nu)) \longrightarrow \frac{C_{\K_0}^*(\F_N^-(\nu), \partial \F_N^-(\nu))}{C_{\K_0,c}^*(\F_N^-(\nu), \partial \F_N^-(\nu))} \longrightarrow 0,
\end{array}
$$
and identify the cohomology of the third term with that of the complex $C_{\K_0}^*(F_N^-(\nu), \partial F_N^-(\nu))$, that is with $H_{\K_0}^*(F_N^-(\nu),\partial F_N^-(\nu))$. Consider the $\C^{nK}$-bundle
$$
(\F_N^-(\nu) \times \C^{nK})_{\K_0} \to \F_N^-(\nu)_{\K_0}.
$$
In relative cohomology, the Thom isomorphism of this bundle
$$
H_{\K_0}^*(\F_N^-(\nu), \partial \F_N^-(\nu)) \simeq H_{\K_0,cv}^{*+2nK}(\F_N^-(\nu) \times \C^{nK}, \partial \F_N^-(\nu) \times \C^{nK}),
$$
and the natural homomorphism
$$
H_{\K_0,cv}^{*+2nK}(\F_N^-(\nu) \times \C^{nK}, \partial \F_N^-(\nu) \times \C^{nK}) \to H_{\K_0}^{*+2nK}(\F_N^-(\nu) \times \C^{nK}, \partial \F_N^-(\nu) \times \C^{nK})
$$
preserve compact supports, and therefore they induce a homomorphism
$$
H_{\K_0}^*(F_N^-(\nu), \partial F_N^-(\nu)) \to H_{\K_0}^{*+2nK}(F_N^-(\nu) \times \C^{nK}, \partial F_N^-(\nu) \times \C^{nK}).
$$
Along with Proposition \ref{prop3.5.2}, we obtain a homomorphism
$$
H_{\K_0}^*(F_N^-(\nu), \partial F_N^-(\nu)) \to H_{\K_0}^{*+2nK}(F_{N+K | \Gamma_N(\nu)}^-, F_{N+K | \partial \Gamma_N(\nu)}^-),
$$
where we have denoted by $F_{N+K | \Gamma_N(\nu)}^-$ and $F_{N+K | \partial \Gamma_N(\nu)}^-$ the restrictions of $F_{N+K}^-$ to $S_{N+K} \times \Gamma_N(\nu)$ and $S_{N+K} \times \partial \Gamma_N(\nu)$ respectively.
We now apply the equivariant version of the excision formula to the triple
$$
(\F_{N+K}^-(\nu), \F^-_{N +K | \Gamma_{N+K}(\nu) \setminus \mathring{\Gamma}_N(\nu)}, \F^-_{N +K | \Gamma_{N+K}(\nu) \setminus \Gamma_N(\nu)}),
$$
where $\mathring{\Gamma}_N(\nu)$ denotes the interior of $\Gamma_N(\nu)$ (the reader shall note here the importance of the fact that transversality is an open condition). It induces an isomorphism
$$
H_{\K_0}^*(\F_{N+K | \Gamma_N(\nu)}^-, \F_{N+K | \partial \Gamma_N(\nu)}^-) \simeq H_{\K_0}^*(\F_{N+K}^-(\nu), \F^-_{N +K | \Gamma_{N+K}(\nu) \setminus \mathring{\Gamma}_N(\nu)}),
$$
which preserves compact supports. Moreover, there is an inclusion of pairs
$$
(\F_{N+K}^-(\nu), \partial \F_{N+K}^-(\nu)) \subset (\F_{N+K}^-(\nu), \F^-_{N +K | \Gamma_{N+K}(\nu) \setminus \mathring{\Gamma}_N(\nu)}).
$$
Putting all  these maps together, we obtain a homomorphism
$$
H_{\K_0}^*(F_N^-(\nu), \partial F_N^-(\nu)) \to H_{\K_0}^{*+2nK}(F_{N+K}^-(\nu), \partial F_{N+K}^-(\nu)).
$$
We will take a limit in $N \to \infty$, and therefore it will be convenient to shift the grading by $2nN$. Thus we have built a homomorphism
$$
f_N^{N+K} : H_{\K_0}^{*+2nN}(F_N^-(\nu), \partial F_N^-(\nu)) \to H_{\K_0}^{*+2n(N+K)}(F_{N+K}^-(\nu), \partial F_{N+K}^-(\nu)).
$$
Notice that all the maps involved in the construction of $f_N^{N+K}$ are natural in cohomology: they involve topological inclusions, the excision formula, the Thom isomorphism, and the deformation of Proposition $\ref{prop3.5.2}$. In particular, we have the following cocycle condition
$$
f_{N+K}^{N+K+K'}\circ f_N^{N+K} = f_N^{N+K+K'}.
$$

\begin{defsubsec}
We define the \textbf{cohomology of $\widetilde{H}$ of level $\nu$} as the limit
$$
\mathcal{H}_{\K_0}^*(F^-(\nu)) := \underset{N \to \infty}{\lim} H_{\K_0}^{*+2nN}(F_N^-(\nu), \partial F_N^-(\nu)).
$$
\end{defsubsec}

\begin{remarksubsec}
\normalfont This definition is independent of the choice of a sequence $\{\Lambda_N\}_N$. Indeed, for any other sequence of cubes $\{\Lambda'_N\}_N$, and for any $N$, there exists $N'$ such that $\Lambda_N \subset \Lambda'_{N'}$. Applying the Thom isomorphism and the excision formula, one can then build a homomorphism from one limit to the other, and similarly in the other direction which, by naturality, are inverse of one another.
\end{remarksubsec}
The above limit comes along with certain structures. First, note that there is an inclusion of pairs $(F_N^-(\nu), \partial F_N^-(\nu)) \subset (\F_N^-(\nu), \partial \F_N^-(\nu))$, and the latter is equivariantly homotopic to the pair $(\Gamma_N(\nu),\partial \Gamma_N(\nu))$. In particular, there is a natural homomorphism
$$
H_{\K_0}^*(\Gamma_N(\nu), \partial \Gamma_N(\nu)) \to H_{\K_0}^*(F_N^-(\nu), \partial F_N^-(\nu)).
$$
We denote by $\mathcal{H}_{\K_0}^*(\nu)$ the limit 
$$
\mathcal{H}_{\K_0}^*(\nu) := \underset{N \to \infty}{\lim} H_{\K_0}^{*+2nN}(\Gamma_N(\nu), \partial \Gamma_N(\nu)),
$$
and call the induced homomorphism 
$$
\Hcal_{\K_0}^*(\nu) \to \Hcal_{\K_0}^*(F^-(\nu))
$$
the \textbf{augmentation map}. Note that the groups $\Hcal_{\K_0}^*(\nu)$ and $\Hcal_{\K_0}^*(F^-(\nu))$ inherit from their finite parts the structure of $H_{\K_0}^*(\pt)$-modules, and that the augmentation map is a module homomorphism. We denote by $\mathcal{J}_{\K_0}^*(F^-(\nu))$ its kernel:
$$
\mathcal{J}_{\K_0}^*(F^-(\nu)) := \ker(\Hcal_{\K_0}^*(\nu) \to \Hcal_{\K_0}^*(F^-(\nu))).
$$
Recall that under the isomorphism $H^2(M,\Z) \simeq \liek_{\Z}^*$ from equation $(\ref{eq1})$, the first Chern class
$c_1$ of $(M, \omega)$ writes 
$$
c_1(m) = \underset{j=1}{\overset{n} \sum} m_j, \quad \text{for all} \quad m \in \liek_{\Z}, \quad \iota(m) = (m_1,\hdots ,m_n).
$$ 
Notice moreover that since $\nu$ is generic and the front $L_N$ of $F_N$ is $\liek_{\Z}$-invariant (see Remark \ref{rem3.2.1}), all translations $\Gamma_N(\nu) + m$ and $\partial \Gamma_N(\nu) + m$ by elements $m \in \liek_{\Z}$ are transversal to $L_N$, for any $N$. In particular, we can apply Proposition \ref{prop3.4.1} replacing $\Lambda_N$ with $\Lambda_N + m$, where $m \in \liek_{\Z}$ (see figure \ref{fig3}). This yields a homogeneous of degree 2 and $\K_0$-invariant fiberwise $C^{1,1}$ homotopy between the restricted homogeneous family $\F_{N+K | \Lambda_N + m}$ and the fiberwise direct sum $\F_N \oplus_{\Lambda_N} \G_m^{(K)}$. Applying Lemma \ref{lem3.5.1}, Corollary \ref{cor3.5.1}, and Proposition \ref{prop3.5.2}, we then obtain an $(\R_{>0} \times \K_0)$-equivariant fiberwise homotopy equivalence
$$
(\F_{N+K | \Gamma_N(\nu) + m}^-,\F_{N+K | \partial \Gamma_N(\nu) + m}^-) \simeq (\F_N^-(\nu) \times \C^{nK + 2c_1(m)}, \partial \F_N^-(\nu) \times \C^{nK + 2c_1(m)}).
$$ 
By the Thom isomorphism and the excision formula, we get a homomorphism in equivariant cohomology
$$
H_{\K_0}^{*+2nN}(F_N^-(\nu), \partial F_N^-(\nu)) \to H_{\K_0}^{*+2n(N+K)+2c_1(m)}(F_{N+K}^-(\nu + p(m)), \partial F_{N+K}^-(\nu + p(m))).
$$
In the limit $N \to \infty$, the latter becomes
$$
\Hcal_{\K_0}^*(F^-(\nu)) \to \Hcal_{\K_0}^{*+2c_1(m)}(F^-(\nu +p(m))).
$$
It is an isomorphism (with inverse given by applying Proposition \ref{prop3.4.1} replacing $0$ with $-m$), which reflects the \textit{Novikov action} of $H_2(M;\Z)$. The latter induces an isomorphism between the kernels
\begin{equation}{\label{eq9}}
\mathcal{J}_{\K_0}^*(F^-(\nu)) \simeq \mathcal{J}_{\K_0}^{*+2c_1(m)}(F^-(\nu + p(m))).
\end{equation}
Notice that the torus $\K_0$ acts trivially on the pair $(\Gamma_N(\nu), \partial \Gamma_N(\nu))$, and therefore the cohomology group $H_{\K_0}^*(\Gamma_N(\nu), \partial \Gamma_N(\nu))$ is a free $H_{\K_0}^*(\pt)$-module of rank $1$ generated by the fundamental cocycle of the sphere $\Gamma_N(\nu) / \partial \Gamma_N(\nu)$. If $K$ is big enough so that $nK + c(m) > 0$, the homomorphism
$$
H_{\K_0}^{*+2nN}(\Gamma_N(\nu), \partial \Gamma_N(\nu)) \to H_{\K_0}^{* + 2n(N+K) + 2c_1(m)}(\Gamma_{N+K}(\nu + p(m)), \partial \Gamma_{N+K}(\nu + p(m)))
$$
can thus be written 
$$
H_{\K_0}^{*+2nN}(\pt) \to H_{\K_0}^{*+2n(N+K)+2c_1(m)}(\pt).
$$
Under this identification, the only map involved in the construction of the above homomorphism is the Thom isomorphism, which reduces to the multiplication by the Euler class of the bundle
$$
(\C^{nK + c_1(m)})_{\K_0} \to B \K_0.
$$
Recall the Chern-Weil isomorphism
$$
H_{\mathbb{T}^n}^*(\pt) \simeq \C[u_1,\hdots ,u_n].
$$
The Euler class of the above bundle is then given by the image of the product $u_1^{K + m_1} \cdots u_n^{K + m_n}$ under the surjective ring homomorphism
$$
H_{\mathbb{T}^n}^*(\pt) \to H_{\K_0}^*(\pt).
$$
This will serve us in section $\ref{sec4}$, where we will be able to compute the isomorphism $(\ref{eq9})$.

\begin{figure}[H]
\begin{center}
\def\svgwidth{0.5\textwidth}
%% Creator: Inkscape inkscape 0.92.4, www.inkscape.org
%% PDF/EPS/PS + LaTeX output extension by Johan Engelen, 2010
%% Accompanies image file 'fig2_3.pdf' (pdf, eps, ps)
%%
%% To include the image in your LaTeX document, write
%%   \input{<filename>.pdf_tex}
%%  instead of
%%   \includegraphics{<filename>.pdf}
%% To scale the image, write
%%   \def\svgwidth{<desired width>}
%%   \input{<filename>.pdf_tex}
%%  instead of
%%   \includegraphics[width=<desired width>]{<filename>.pdf}
%%
%% Images with a different path to the parent latex file can
%% be accessed with the `import' package (which may need to be
%% installed) using
%%   \usepackage{import}
%% in the preamble, and then including the image with
%%   \import{<path to file>}{<filename>.pdf_tex}
%% Alternatively, one can specify
%%   \graphicspath{{<path to file>/}}
%% 
%% For more information, please see info/svg-inkscape on CTAN:
%%   http://tug.ctan.org/tex-archive/info/svg-inkscape
%%
\begingroup%
  \makeatletter%
  \providecommand\color[2][]{%
    \errmessage{(Inkscape) Color is used for the text in Inkscape, but the package 'color.sty' is not loaded}%
    \renewcommand\color[2][]{}%
  }%
  \providecommand\transparent[1]{%
    \errmessage{(Inkscape) Transparency is used (non-zero) for the text in Inkscape, but the package 'transparent.sty' is not loaded}%
    \renewcommand\transparent[1]{}%
  }%
  \providecommand\rotatebox[2]{#2}%
  \newcommand*\fsize{\dimexpr\f@size pt\relax}%
  \newcommand*\lineheight[1]{\fontsize{\fsize}{#1\fsize}\selectfont}%
  \ifx\svgwidth\undefined%
    \setlength{\unitlength}{244.89834554bp}%
    \ifx\svgscale\undefined%
      \relax%
    \else%
      \setlength{\unitlength}{\unitlength * \real{\svgscale}}%
    \fi%
  \else%
    \setlength{\unitlength}{\svgwidth}%
  \fi%
  \global\let\svgwidth\undefined%
  \global\let\svgscale\undefined%
  \makeatother%
  \begin{picture}(1,0.98429926)%
    \lineheight{1}%
    \setlength\tabcolsep{0pt}%
    \put(0,0){\includegraphics[width=\unitlength,page=1]{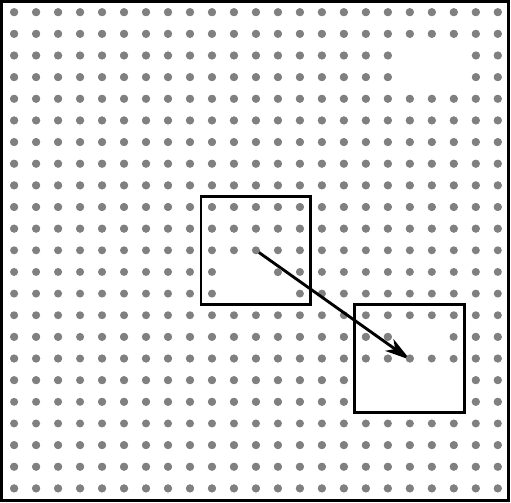}}%
    \put(0.78685043,0.84106448){\color[rgb]{0,0,0}\makebox(0,0)[lt]{\lineheight{1.25}\smash{\begin{tabular}[t]{l}$\Lambda_{N+K}$\end{tabular}}}}%
    \put(0,0){\includegraphics[width=\unitlength,page=2]{fig2_3.pdf}}%
    \put(0.44869055,0.4258584){\color[rgb]{0,0,0}\makebox(0,0)[lt]{\lineheight{1.25}\smash{\begin{tabular}[t]{l}$\Lambda_N$\end{tabular}}}}%
    \put(0.7179212,0.21368227){\color[rgb]{0,0,0}\makebox(0,0)[lt]{\lineheight{1.25}\smash{\begin{tabular}[t]{l}$\Lambda_N+m $\end{tabular}}}}%
    \put(0.80312501,0.31392393){\color[rgb]{0,0,0}\makebox(0,0)[lt]{\lineheight{1.25}\smash{\begin{tabular}[t]{l}$m$\end{tabular}}}}%
  \end{picture}%
\endgroup%

\caption[Moving the cube in the direction of $m \in \liek_{\Z}$]{Moving the cube in the direction of $m \in \liek_{\Z}$. Proposition \ref{prop3.4.1} can then be applied to the homogeneous generating family $\F_{N+K | \Lambda_N + m}$, which becomes homotopic to $\F_N \oplus_{\Lambda_N} \G_m^{(K)}$. It remains to observe that $\dim(\G_m^{(K)-}) = 2nK+2c_1(m)$, by Lemma \ref{lem3.5.1}.}
\label{fig3}
\end{center}
\end{figure}

\section{The Gysin sequence}{\label{sec4}}

In this section, we describe in more detail the domain $\Hcal_{\K_0}^*(\nu)$ and the kernel $\mathcal{J}_{\K_0}^*(F^-(\nu))$ of the augmentation map $\Hcal_{\K_0}^*(\nu) \to \Hcal_{\K_0}^*(F^-(\nu))$. We begin by showing, by means of a Gysin-type long exact sequence relating our construction to that of \cite{Giv95}, that $\Hcal_{\K_0}^*(\nu)$ can be identified with the ring of regular functions on the intersection $(\liek_0 \otimes \C) \cap (\C^{\times})^n$, where $(\C^{\times})^n$ is the complex torus. We then explicit $\mathcal{J}_{\K_0}^*(F^-(\nu))$ in the case where the homogeneous generating family $\F_N$ is associated with a decomposition of the Hamiltonian symplectomorphism $\Phi_0 = Id_{\C^n}$, $0$ being a Hamiltonian lift of the trivial contact Hamiltonian $0$ on $V$. We will see in section \ref{sec5} that this suffices for the proof of Proposition \ref{prop1.2.3}. We also include the proofs of Propositions \ref{prop1.2.1} and \ref{prop1.2.2} at the end of this section. Consider the aforementioned setting for the homogeneous generating family $\F_N$. Up to a reparametrization of the decomposition, the function $\F_{\lambda}^{(N)}$ is simply the generating function $\G_{\lambda}^{(N)}$ associated with the decomposition
$$\underbrace{\exp(\frac{\lambda}{2N}) \circ \cdots \circ \exp(\frac{\lambda}{2N})}_{2N-\text{times}},$$
as defined in section \ref{sec3.3.1}, so we shall use the notations $\G_N := \F_N$ and $G_N := F_N$. Note that for any $\lambda \in \Lambda_N$, $\G_{\lambda}^{(N)}$ is $\K$-invariant, so that one can consider either the $\K_0$-equivariant cohomology groups $H_{\K_0}^*(G_N^-(\nu), \partial G_N^-(\nu))$, or the $\K$-equivariant cohomology groups $H_{\K}^*(G_N^-(\nu),\partial G_N^-(\nu))$. In this last case, the functions $f_N^{N+K}$ from section \ref{sec3.6} can still be constructed in the exact same way as above, and one can define limits of $\K$-equivariant cohomology groups:
$$
\begin{array}{lcl}
\mathcal{H}_{\K}^*(G^-(\nu)) := \underset{N \to \infty}{\lim} H_{\K}^{*+2nN}(G_N^-(\nu),\partial G_N^-(\nu))\\ \\
\mathcal{H}_{\K}^*(\nu) := \underset{N \to \infty}{\lim} H_{\K}^{*+2nN}(\Gamma_N(\nu), \partial \Gamma_N(\nu)).
\end{array}
$$
The latter inherit the structure of $H_{\K}^*(\pt)$-modules, and therefore of $H_{\mathbb{T}^n}^*(\pt)$-modules as well, through the surjective homomorphism
$$
H_{\mathbb{T}^n}^*(\pt) \to H_{\K}^*(\pt).
$$
In \cite{Giv95}, Givental showed that the augmentation map
$$
\Hcal_{\K}^*(\nu) \to \Hcal_{\K}^*(G^-(\nu))
$$
has \textit{trivial cokernel}, and he described its kernel in terms of Newton diagrams associated with the level $p^{-1}(\nu)$. We will see now that there is a Gysin-type long exact sequence relating the cohomology groups $\Hcal_{\K}^*(G^-(\nu))$ of Givental \cite{Giv95} to our cohomology groups $\Hcal_{\K_0}^*(G^-(\nu))$. Gysin sequences were already used in symplectic topology to relate different kinds of Floer-type homologies, for instance in \cite{Per08}, \cite{BO13}, \cite{BK13}. We give here another example of such a use, and by further analyzing the maps involved, we will be able to describe the kernel of the augmentation map $\Hcal_{\K_0}^*(\nu) \to \Hcal_{\K_0}^*(G^-(\nu))$. Recall that given any oriented $S^1$-bundle $\pi : V \to M$, the cohomology groups of $V$ and $M$ are related by the \textit{Gysin long exact sequence}
$$
\cdots \longrightarrow H^*(M) \overset{\cup eu}{\longrightarrow} H^{*+2}(M) \overset{\pi^*}{\longrightarrow}  H^{*+2}(V) \overset{\pi_*}{\longrightarrow} H^{*+1}(M) \longrightarrow \cdots,
$$
where $eu \in H^2(M)$ is the Euler class of the bundle $\pi$, $\cup$ denotes the cup-product, $\pi^*$ is the pull-back and $\pi_*$ the push-forward. If $A \subset V$, the above sequence gives rise naturally to a long exact sequence in relative cohomology
$$
\cdots \longrightarrow H^*(M, \pi(A)) \overset{\cup eu}{\longrightarrow} H^{*+2}(M, \pi(A)) \overset{\pi^*}{\longrightarrow}  H^{*+2}(V, A) \overset{\pi_*}{\longrightarrow} H^{*+1}(M, \pi(A)) \longrightarrow \cdots,
$$
where we have used, for the sake of clarity, the same notation for the maps induced in cohomology and in relative cohomology. We will also use the notation $H_{\mathbb{T}^n}^*(\pt) \simeq \C[u]$, where $u = (u_1,\hdots ,u_n)$. Let $I$ (resp. $I_0$) denote the ideal of $\C[u]$ generated by polynomials vanishing on the complexified Lie algebra $\liek \otimes \C \subset \C^n$ (resp. $\liek_0 \otimes \C \subset \C^n$). There are natural isomorphisms
$$
H_{\K}^*(\pt) \simeq \C[u] / I \quad \text{and} \quad H_{\K_0}^*(\pt) \simeq \C[u] / I_0.
$$
In other words, $H_{\K}^*(\pt)$ (resp. $H_{\K_0}^*(\pt)$) is the ring of regular (or polynomial) functions on $\liek \otimes \C$ (resp. $\liek_0 \otimes \C$). Recall that the ring of regular functions on the complex torus $(\C^{\times})^n$ is given by $\C[u,u^{-1}]$. Let us denote by $\mathcal{R}$ and $\mathcal{R}_0$ the rings of regular functions on the intersections $(\liek \otimes \C) \cap (\C^{\times})^n$ and $(\liek_0 \otimes \C) \cap (\C^{\times})^n$ respectively. We have
\begin{center}
$\mathcal{R} = \C[u] / I \otimes \C[u,u^{-1}] \simeq \C[u,u^{-1}] / I \C[u,u^{-1}]$ \\ and \\ $\mathcal{R}_0 = \C[u] / I_0 \otimes \C[u,u^{-1}] \simeq \C[u,u^{-1}] / I_0 \C[u,u^{-1}]$.
\end{center}
Let $J^*(\nu)$ denote the $\C[u]$-submodule of $\C[u,u^{-1}]$ generated by monomials whose degrees lie in $\liek_{\Z}$ above the level $p^{-1}(\nu)$:
\begin{equation}{\label{eq10}}
J^*(\nu) := \langle u^{\iota(m)} \ |\  m \in \liek_{\Z}, \ p(m) \geq \nu \rangle,
\end{equation}
and let $\mathcal{J}_{\K}^*(\nu)$ denote its projection to $\mathcal{R}$. Givental proved the following
\begin{propsec}[\cite{Giv95} Propositions $5.3$, $5.4$, Corollary $5.5$]{\label{prop4.1}}
There are isomorphisms of $H_{\K}^*(\pt)$-modules
$$
\Hcal_{\K}^*(\nu) \simeq \mathcal{R} \quad \text{and} \quad \Hcal_{\K}^*(G^-(\nu)) \simeq \mathcal{R} / \mathcal{J}_{\K}^*(\nu).
$$
\end{propsec}

In particular, we have $\mathcal{J}_{\K}^*(G^-(\nu)) \simeq \mathcal{J}_{\K}^*(\nu)$. We denote by $\mathcal{J}_{\K_0}^*(\nu)$ the projection of $J^*(\nu)$ to $\mathcal{R}_0$. We claim that the kernel $\mathcal{J}_{\K_0}^*(G^-(\nu))$ of the augmentation map
$$
\Hcal_{\K_0}^*(\nu) \to \Hcal_{\K_0}^*(G^-(\nu))
$$
can be identified with $\mathcal{J}_{\K_0}^*(\nu)$. Note that if $EG \to BG$ denotes the universal bundle associated with a Lie group $G$, then $E \K_0$ may be thought of as $E \K$, since the latter is contractible, and $\K_0$ acts freely on it. Therefore, one can define the homotopy quotient $(G_N^-(\nu))_{\K_0}$ by
$$
(G_N^-(\nu))_{\K_0} := (G_N^-(\nu) \times E \K) / \K_0,
$$
and construct the following principal $S^1 := \K / \K_0$-bundle:
$$
\pi_N : (G_N^-(\nu))_{\K_0} \to (G_N^-(\nu))_{\K}.
$$
In relative cohomology, one can then relate the $\K$-equivariant and $\K_0$-equivariant cohomology groups above by the Gysin sequence
$$
\begin{array}{lcl}
\cdots \longrightarrow H_{\K}^{*+2nN}(G_N^-(\nu), \partial G_N^-(\nu)) \overset{\cup eu}{\longrightarrow} H_{\K}^{*+ 2nN +2}(G_N^-(\nu), \partial G_N^-(\nu)) \overset{\pi_N^*}{\longrightarrow} \\ \\
\quad \quad \quad  H_{\K_0}^{*+ 2nN +2}(G_N^-(\nu), \partial G_N^-(\nu)) \overset{\pi_{N*}}{\longrightarrow} H_{\K}^{*+ 2nN +1}(G_N^-(\nu), \partial G_N^-(\nu)) \longrightarrow \cdots.
\end{array}
$$
Let us interpret the Euler class $eu$ in this sequence. The classifying map
$$
l_N : (G_N^-(\nu))_{\K} \to B(\K / \K_0)
$$
of the bundle $\pi_N$ gives rise to a homomorphism
$$
l_N^* : H^*(B(\K / \K_0)) \to H_{\K}^*(G_N^-(\nu)).
$$
It restricts on $(\partial G_N^-(\nu))_{\K}$ to the classifying map of the principal $\K / \K_0$-bundle $(\partial G_N^-(\nu))_{\K_0} \to (\partial G_N^-(\nu))_{\K}$, which yields a homomorphism (still denoted by $l_N^*$)
$$
l_N^* : H^*(B(\K / \K_0)) \to H_{\K}^*(\partial G_N^-(\nu)).
$$
Moreover, by the Chern-Weil isomorphism, the group $H^*(B(\K/\K_0))$ can be identified with the polynomial algebra $\C[(\liek / \liek_0)^*] \simeq \C[p]$ (here $p$ is assigned the degree $2$, and is viewed as a $\C$-valued functional on $\liek \otimes \C$). The long exact sequence of relative $\K$-equivariant cohomology writes
$$
\begin{array}{lcl}
\cdots \longrightarrow H_{\K}^{*-1}(G_N^-(\nu)) \longrightarrow H_{\K}^{*-1}(\partial G_N^-(\nu)) \longrightarrow H_{\K}^*(G_N^-(\nu), \partial G_N^-(\nu)) \longrightarrow \cdots.
\end{array}
$$
Along with the homomorphisms above, we obtain a diagram
$$
\begin{tikzcd}
\cdots \ar[r] & H_{\K}^{*-1}(G_N^-(\nu)) \ar[r] & H_{\K}^{*-1}(\partial G_N^-(\nu)) \ar[r] & H_{\K}^*(G_N^-(\nu), \partial G_N^-(\nu)) \ar[r] & \cdots\\
\cdots \ar[r] & \C[p] \ar[r,"\underset{=}{Id}"] \ar[u,"l_N^*"] & \C[p] \ar[r,"\underset{=}{Id}"] \ar[u,"l_N^*"] & \C[p] \ar[r] & \cdots.
\end{tikzcd}
$$
Therefore, we obtain a map in relative equivariant cohomology 
$$
l_N^* : \C[p] \to H_{\K}^*(G_N^-(\nu), \partial G_N^-(\nu)).
$$
Now, recall that any characteristic class of a principal bundle is given by the pull-back by its classifying map of a universal characteristic class of the associated universal principal bundle. In our case, we have the universal principal $S^1$-bundle $E(\K / \K_0) \to B(\K / \K_0)$, and by definition, the universal characteristic classes are the cohomology classes of $H^*(B (\K / \K_0))$, which is isomorphic to $\C[p]$. Moreover, since it is an $S^1$-bundle, the universal Euler class agrees with the first Chern class of the associated complex line bundle, which is, by definition, the generator of $H^2(B (\K / \K_0))$, that is $p$. This implies that the Euler class in the above Gysin sequence is given by the pull-back $l_N^*(p)$ of $p$ by the map $l_N$. The other two maps involved are the pull-back $\pi_N^*$ and the push-forward $\pi_{N*}$ (in relative cohomology). The latter are canonical, and therefore \textit{they commute with the limit $N \to \infty$}. Moreover, the cup product by the Euler class is \textit{functorial}. Since directs limits of exact sequences are exact sequences, we obtain a Gysin sequence in the limit $N \to \infty$:
$$
\cdots \longrightarrow \Hcal_{\K}^*(G^-(\nu)) \longrightarrow \Hcal_{\K}^{*+2}(G^-(\nu)) \longrightarrow \Hcal_{\K_0}^{*+2}(G^-(\nu)) \longrightarrow \Hcal_{\K}^{*+1}(G^-(\nu)) \longrightarrow \cdots.
$$
The same applies for the equivariant cohomology groups of the pair $(\Gamma_N(\nu), \partial \Gamma_N(\nu))$, for which the Gysin sequence in the limit is given by 
\begin{equation}{\label{eq11}}
\cdots \longrightarrow \Hcal_{\K}^*(\nu) \longrightarrow \Hcal_{\K}^{*+2}(\nu) \longrightarrow \Hcal_{\K_0}^{*+2}(\nu) \longrightarrow \Hcal_{\K}^{*+1}(\nu) \longrightarrow \cdots.
\end{equation}
\begin{propsec}{\label{prop4.2}}
There is a natural isomorphism of $H_{\K_0}^*(\pt)$-modules
$$
\Hcal_{\K_0}^*(\nu) \simeq \mathcal{R}_0.
$$
\begin{proof}
By Proposition $\ref{prop4.1}$, $\Hcal_{\K}^*(\nu)$ is the ring $\mathcal{R}$ of regular functions on the intersection $(\liek \otimes \C) \cap (\C^{\times})^n$. The latter is an irreducible variety, and thus there are no zero divisors in $\Hcal_{\K}^*(\nu)$. In particular, the map induced in the limit by the cup product with the Euler class is injective. This reduces the Gysin sequence $(\ref{eq11})$ to the following short exact sequence
$$
0 \longrightarrow \Hcal_{\K}^*(\nu) \longrightarrow \Hcal_{\K}^{*+2}(\nu) \longrightarrow \Hcal_{\K_0}^{*+2}(\nu) \longrightarrow 0.
$$
Now, for any $N$, recall that the torus $\K$ acts trivially on the pair $(\Gamma_N(\nu), \partial \Gamma_N(\nu))$. This makes $H_{\K}^*(\Gamma_N(\nu), \partial \Gamma_N(\nu))$ into a free $H_{\K}^*(\pt)$-module of rank $1$, generated by the fundamental cocycle of the sphere $\Gamma_N(\nu) / \partial \Gamma_N(\nu)$ (see section \ref{sec3.6}). The same applies for the action of $\K_0$, and therefore the principal $S^1$-bundle to be considered here reduces to
$$
B \K_0 \to B \K.
$$
Through the Chern-Weil isomorphism, in cohomology, this map becomes $\C[\liek^*] \to \C[\liek_0^*]$. For the map $l^* : \C[p] \to \C[\liek^*]$ induced in cohomology by the classifying map in this case, the Euler class $l^*(p)$ is simply the generator of the kernel of $\liek^* \to \liek_0^*$, that is $p$. Thus in $\Hcal_{\K}^*(\nu)$, the map induced by the cup product with the Euler class is simply the multiplication by $p$. By the very definition of $I_0$ and $I$, one gets
$$
\Hcal_{\K_0}^*(\nu) \simeq \mathcal{R} / p \mathcal{R} \simeq \mathcal{R}_0.
$$ 
\end{proof}
\end{propsec}

\begin{remarksec}
\normalfont Notice that if $M = \C P^{n-1}$, then $\liek = \R$, and $\liek_0 = \{0\}$. Therefore $I_0 = \C[u,u^{-1}]$, and $\mathcal{R}_0 = \{0\}$. Moreover, the Euler class $p$ in this case is the generator $v$ of $\mathcal{R} \simeq \C[v,v^{-1}]$, and thus it is invertible. Therefore the Gysin sequence in the limit $N \to \infty$ gives $\Hcal_{\K_0}^*(\nu) = \{0\}$.
\end{remarksec}

By the above discussion, we have the following commutative diagram
$$
\begin{tikzcd}
   &          &          &        \vdots    \ar[d]                    & \\ 
   &  0 \ar[d] & 0 \ar[d] & \Hcal_{\K_0}^{*+1}(G^-(\nu)) \ar[d] & \\ 
  0 \ar[r] & \mathcal{J}_{\K}^*(\nu) \ar[r] \ar[d,"p"] & \Hcal_{\K}^*(\nu) \ar[r] \ar[d,"p"] & \Hcal_{\K}^*(G^-(\nu)) \ar[r] \ar[d] & 0 \\
0 \ar[r] & \mathcal{J}_{\K}^{*+2}(\nu) \ar[r] \ar[d,"f"] & \Hcal_{\K}^{*+2}(\nu) \ar[r] \ar[d,"g"] & \Hcal_{\K}^{*+2}(G^-(\nu)) \ar[r] \ar[d,"h"] & 0 \\
0 \ar[r] & \mathcal{J}_{\K_0}^{*+2}(G^-(\nu)) \ar[r] \ar[d, dotted, red] & \Hcal_{\K_0}^{*+2}(\nu) \ar[r] \ar[d] & \Hcal_{\K_0}^{*+2}(G^-(\nu)) \ar[d] \\
&  0  & 0  & \Hcal_{\K}^{*+1}(G^-(\nu)) \ar[d]. \\
& & & \vdots
\end{tikzcd}
$$
The vertical right and middle sequences are Gysin sequences, and the horizontal sequences are induced by the augmentation maps. Note that the left vertical sequence is a priori not necessarily exact at the middle term. The diagram is clear from the discussion above, except for the surjectivity of the map $f$. From the two bottom rows of the diagram and the snake lemma, we have the following short exact sequence:
$$
0 \longrightarrow \ker f \longrightarrow \ker g \longrightarrow \ker h \longrightarrow \text{coker} f \longrightarrow 0.
$$
The homomorphism $\ker g \to \ker h$ is simply the surjection $\cup l^*(p) (\Hcal_{\K}^*(\nu)) \to \cup l^*(p) (\Hcal_{\K}^*(G^-(\nu)))$, where $\cup l^*(p)$ denotes the map induced in the limit $N \to \infty$ by the cup product with the Euler class (on $\Hcal_{\K_0}^*(\nu)$, as we've seen above, it is the multiplication by $p$). We conclude that $f$ is onto. In particular, the kernel $\mathcal{J}_{\K_0}^*(G^-(\nu))$ is the image of $\mathcal{J}_{\K}^*(\nu) \subset \Hcal_{\K}^*(\nu)$ under the map $g$. In other words, it is the projection to $\mathcal{R}_0$ of the module $J^*(\nu)$:
$$
\mathcal{J}_{\K_0}^*(G^-(\nu)) \simeq \mathcal{J}_{\K_0}^*(\nu).
$$

Let us go back to the general case, where the homogeneous generating family $\F_N$ is associated with a decomposition of a Hamiltonian lift of a contact Hamiltonian $h : V \times [0,1] \to \R$. In this setting, we cannot compute the kernel $\mathcal{J}_{\K_0}^*(F^-(\nu))$ of the augmentation map
$$
\Hcal_{\K_0}^*(\nu) \to \Hcal_{\K_0}^*(F^-(\nu)).
$$
However, we can now compute the isomorphism $(\ref{eq9})$:
$$
\mathcal{J}_{\K_0}^*(F^-(\nu)) \simeq \mathcal{J}_{\K_0}^{*+2c_1(m)}(F^-(\nu + p(m))).
$$
Recall from section $\ref{sec3.6}$ that the homomorphism
$$
H_{\K_0}^{*+2nN}(\Gamma_N(\nu), \partial \Gamma_N(\nu)) \to H_{\K_0}^{*+2n(N+K) + 2c_1(m)}(\Gamma_{N+K}(\nu + p(m)), \partial \Gamma_{N+K}(\nu + p(m)))
$$ 
is given by the multiplication by the image of the product $u_1^{K+m_1} \cdots u_n^{K+m_n}$ through the natural projection
$$
\C[u] \to \C[u] / I_0.
$$
In the limit $N \to \infty$, this implies that the isomorphism $(\ref{eq9})$ is given by the multiplication by $u_1^{m_1} \cdots u_n^{m_n}$:
$$
\mathcal{J}_{\K_0}^{*+2c_1(m)}(F^-(\nu + p(m))) \simeq u_1^{m_1} \cdots u_n^{m_n} \mathcal{J}_{\K_0}^*(F^-(\nu)).
$$

To close this section, we now prove Propositions \ref{prop1.2.1} and \ref{prop1.2.2}. They are included here only for the sake of completeness, for they are similar to the proofs of \cite[Propositions $6.2$, $6.3$]{Giv95}. Notice that since the generating functions $\F_{\lambda}^{(N)}$ decrease in positive directions (see section \ref{sec3.3.2} property $4$), for any generic $\nu_0 \leq \nu_1$, if $N$ is big enough so that $\Gamma_N(\nu_0)$ and $\Gamma_N(\nu_1)$ are not empy, there is an injection of pairs $(F_N^-(\nu_0), \partial F_N^-(\nu_0)) \hookrightarrow (F_N^-(\nu_1), \partial F_N^-(\nu_1))$, which induces a natural homomorphism
$$
\Hcal_{\K_0}^*(F^-(\nu_1)) \to \Hcal_{\K_0}^*(F^-(\nu_0)).
$$
\noindent \textbf{Proposition \ref{prop1.2.1}.} \textit{Suppose that $[\nu_0,\nu_1] \cap \Spec(g \circ \phi_h) = \emptyset$. Then the homomorphism above is an isomorphism}
$$
\Hcal_{\K_0}^*(F^-(\nu_1)) \simeq \Hcal_{\K_0}^*(F^-(\nu_0)).
$$
\begin{proof}
Assume that $[\nu_0, \nu_1] \cap \Spec(g \circ \phi_h) = \emptyset$. For any $\nu \in [\nu_0,\nu_1]$ and any $N$, $\Gamma_N(\nu)$ is transversal to the front $L_N$. Moreover, the set of generic $\nu$ is of full measure. Let $\nu \in [\nu_0,\nu_1]$ be in this set. Since transversality is an open condition, by Proposition \ref{prop2.1.3}, there exists $\epsilon > 0$ and a $\K_0$-equivariant homotopy equivalence
$$
(F_N^-(\nu + \epsilon), \partial F_N^-(\nu + \epsilon)) \simeq (F_N^-(\nu - \epsilon), \partial F_N^-(\nu - \epsilon)).
$$
In the limit $N \to \infty$ of $\K_0$-equivariant cohomology groups, this yields an isomorphism
$$
\Hcal_{\K_0}^*(F^-(\nu + \epsilon)) \simeq \Hcal_{\K_0}^*(F^-(\nu - \epsilon)).
$$
Choosing a finite subcovering of $[\nu_0,\nu_1]$ by such segments $[\nu - \epsilon, \nu + \epsilon]$ yields an isomorphism 
$$\Hcal_{\K_0}^*(F^-(\nu_1)) \simeq \Hcal_{\K_0}^*(F^-(\nu_0)),$$ 
as claimed.
\end{proof}

\begin{remarksec} 
\normalfont Note that the inclusion $(F_N^-(\nu_0), \partial F_N^-(\nu_0)) \hookrightarrow (F_N^-(\nu_1), \partial F_N^-(\nu_1))$ is not necessarily a homotopy equivalence, even $N$ is big: even if $\nu_0$ and $\nu_1$ are chosen generically (which is not the case), nothing forces the front $L_N$ to stay transversal to $\partial \Gamma_N(\nu)$ for all $\nu \in (\nu_0, \nu_1)$. Therefore, the isomorphism can only be obtained in the limit $N \to \infty$, which is why we perform a limit process.
\end{remarksec}

\noindent \textbf{Proposition \ref{prop1.2.2}.} \textit{
Suppose that the segment $[\nu_0,\nu_1]$ contains only one value $\nu \in \Spec(g \circ \phi_h)$, which corresponds to a finite number of translated points. Let $v \in H_{\K_0}^*(\pt)$ be an element of positive degree, and $q \in \mathcal{R}_0$. Suppose that $q \in \mathcal{J}_{\K_0}^*(F^-(\nu_0))$. Then $v q \in \mathcal{J}_{\K_0}^*(F^-(\nu_1))$.}
\begin{proof}
We rephrase the statement as follows: suppose that $q_1 \in \Hcal_{\K_0}^*(F^-(\nu_1))$ and $q_0 \in \Hcal_{\K_0}^*(F^-(\nu_0))$ are the images of $q$ under the augmentation maps
$$
\begin{tikzcd}
q \in \mathcal{R}_0 \ar[r] \ar[dr] & \Hcal_{\K_0}^*(F^-(\nu_1)) \ni q_1 \ar[d] \\
& \Hcal_{\K_0}^*(F^-(\nu_0)) \ni q_0.
\end{tikzcd}
$$
Then $q_0 = 0$ implies $v q_1 = 0$. 

Without loss of generality, we can assume that $\nu_0 = \nu - \epsilon$, $\nu_1 = \nu + \epsilon$ (using the argument of Proposition \ref{prop1.2.1}), and that there is only one $\K_0$-orbit of fixed points associated with $\nu$. Using the deformations from Proposition \ref{prop2.1.3}, for $N$ big enough, the pair $(F_N^-(\nu_0), \partial F_N^-(\nu_0))$ is embedded into $(F_N^-(\nu_1), \partial F_N^-(\nu_1))$ as the complement of a neighborhood to the pair $(F_N^-(\nu), \partial F_N^-(\nu))$ which contains a $\K_0$-orbit of critical points of $\hat{p}_N$ (since avoiding critical $\K_0$-orbits of $\hat{p}_N$ is equivalent to staying transversal to the front $L_N$). There exists a non-zero representative $\hat{q} \in H_{\K_0}^*(F_N^-(\nu_1),F_N^-(\nu_1))$ of $q_1$ (otherwise $q_1 = 0$ and the statement is trivial) which vanishes when restricted to $(F_N^-(\nu_0), \partial F_N^-(\nu_0))$. In particular, $\hat{q}$ is the image of some element $\alpha \in H_{\K_0}^*(F_N^-(\nu_1), F_N^-(\nu_0))$ under the long exact sequence
$$
\begin{array}{lcl}
\cdots \longrightarrow H_{\K_0}^*(F_N^-(\nu_1), F_N^-(\nu_0)) \overset{f}{\longrightarrow} H_{\K_0}^*(F_N^-(\nu_1), \partial F_N^-(\nu_1)) \longrightarrow H_{\K_0}^*(F_N^-(\nu_0), \partial F_N^-(\nu_0)) \longrightarrow \\ \\
\quad \quad \quad  H_{\K_0}^{*+1}(F_N^-(\nu_1), F_N^-(\nu_0)) \longrightarrow \cdots.
\end{array}
$$
\textit{The torus $\K_0$ acts freely in a neighborhood of the $\K_0$-orbit of critical points of $\hat{p}_N$} (since $M$ is compact), hence the equivariant cohomology $H_{\K_0}^*(F_N^-(\nu_1),F_N^-(\nu_0))$ is simply the singular cohomology $H^*(F_N^-(\nu_1) / \K_0, F_N^-(\nu_0) / \K_0)$. Moreover, the action of the coefficient ring $H_{\K_0}^*(\pt)$ on the latter is trivial, since the principal bundle associated with the free $\K_0$-action on the neighborhood of the critical $\K_0$-orbit of $\hat{p}_N$ can be trivialized. Thus,
$$
v \alpha = 0 \quad \text{and} \quad v \hat{q} = f(v \alpha) = 0.
$$
\end{proof} 

\section{Elements of minimal degree in the monotone case}{\label{sec5}}

The proof of Theorem \ref{theo1.1.1} relies on a strong algebraic property of the kernel $\mathcal{J}_{\K_0}^*(F^-(\nu))$. In the $\K$-equivariant case, Givental showed that the kernel $\mathcal{J}_{\K}^*(F^-(\nu))$ of the augmentation map \looseness=-1
$$
\mathcal{R} \simeq \Hcal_{\K}^*(\nu) \to \Hcal_{\K}^*(F^-(\nu))
$$
admits, in some sense, elements of minimal degree ({\cite[Corollary $1.3$]{Giv95}}). It appears that this is not always true for the $\K_0$-equivariant kernel $\mathcal{J}_{\K_0}^*(F^-(\nu))$ when the manifold is a prequantization space over a toric manifold which is not necessarily monotone. For instance, consider, for any $a, b \in \Z \setminus \{0\}$, the symplectic toric manifold $(\C P^1 \times \C P^1, \omega_{a,b}, \mathbb{T}^2 / S^1 \times \mathbb{T}^2 / S^1)$, where $\omega_{a,b} := a \omega_{\text{FS}} \oplus b \omega_{\text{FS}}$, $\omega_{\text{FS}}$ being the Fubini-Study form. Recall that it is obtained by symplectic reduction of $\C^2 \times \C^2$ by the diagonal $S^1$-action on each factor. The associated momentum map is given by 
$$
P_{\mathbb{T}^2} : \C^2 \times \C^2 \to \R^{2*}, \quad (z_1,z_2,w_1,w_2) \mapsto \pi(|z_1|^2 + |z_2|^2, |w_1|^2 + |w_2|^2),
$$
and therefore, $\liek \simeq \R^2$ embeds into $\R^4$ via
$$
\iota : \R^2 \hookrightarrow \R^4, \quad m = (m_1,m_2) \mapsto (m_1,m_1,m_2,m_2).
$$
Under the identification $H^2(M;\Z) \simeq \liek_{\Z}^* \simeq \Z^{2*}$, the cohomology class $[\omega_{a,b}]$ of $\omega_{a,b}$ is simply given by the regular value $p = (a,b) \in \Z^{2*}$ of the momentum map $P_{\mathbb{T}^2}$, and for any $m = (m_1, m_2) \in H_2(M; \Z) \simeq \Z^2$, we have
$$
[\omega_{a,b}](m) = a m_1 + b m_2 \quad \text{and} \quad c_1(m) = 2m_1 + 2m_2.
$$
In particular, $(\C P^1 \times \C P^1,  \omega_{a,b})$ is monotone if and only if $a = b$. Let us compute the kernel $\mathcal{J}_{\K_0}^*(G^-(\nu)) \simeq \mathcal{J}_{\K_0}^*(\nu)$ for the generating family $G_N$ associated with the decomposition
$$
\underbrace{\exp(\frac{\lambda}{2N}) \circ \cdots \circ \exp(\frac{\lambda}{2N})}_{2N-\text{times}}.
$$
The $\C[u_1,u_2,u_3,u_4]$-module 
$$
J^*(\nu) = \langle u^{\iota(m)} \ | \  m \in \liek_{\Z}, \ p(m) \geq \nu \rangle
$$ 
from equation $(\ref{eq10})$ can thus be written as
$$
J^*(\nu) = \langle u_1^{m_1}u_2^{m_1}u_3^{m_2}u_4^{m_2} \  | \  m \in \liek_{\Z}, \  am_1 + bm_2 \geq \nu \rangle.
$$
The ideal $I_0$ provides identifications $u_1 = u_2 = v_1, u_3 = u_4 = v_2$, and imposes that $av_1 +bv_2 = 0$. Hence, putting $v_2 = v$, the quotient $\mathcal{J}_{\K_0}^*(\nu) \subset \Hcal_{\K_0}^*(\nu)$ is given by
$$
\mathcal{J}_{\K_0}^*(\nu) = \langle -\frac{b}{a} v^{2(m_1+m_2)}  \  | \  am_1+bm_2 \geq \nu \rangle.
$$
We see here that the submodule $\mathcal{J}_{\K_0}^*(\nu)$ is different from $\Hcal_{\K_0}^*(\nu) \simeq \C[v,v^{-1}]$ if and only if $a=b$, that is if we are in the monotone case. Indeed, suppose that $a = b = k$ with, say $k > 0$ (the same argument applies if $k < 0$). Then $k(m_1 + m_2) \geq \nu$ implies that $2(m_1 + m_2) \geq \frac{2}{k}\nu$. Therefore, the monomials of the $\C[v]$-module $\mathcal{J}_{\K_0}^*(\nu)$ are all of the form $v^{l + l'}$, with $l \geq \frac{2}{k}\nu$ and $l' \geq 0$, which in particular implies that $\mathcal{J}_{\K_0}^*(\nu)$ is different from $\Hcal_{\K_0}^*(\nu) \simeq \C[v,v^{-1}]$. On the other hand, if $a \neq b$, then the expression $m_1 + m_2$ can take all values in $\Z$, even if $a m_1 + b m_2 \geq \nu$. Therefore, the monomials of the $\C[v]$-module $\mathcal{J}_{\K_0}^*(\nu)$ are all of the form $v^{2l + l'}$, where $l \in \Z$, and $l' \geq 0$, that is they are of the form $v^l$, $l \in \Z$. Hence, $\mathcal{J}_{\K_0}^*(\nu) = \Hcal_{\K_0}^*(\nu)$.

\begin{remarksec}
\normalfont Note that if we project $\mathcal{J}^*(\nu)$ to $\Hcal_{\K}^*(\nu) \simeq \C[u_1,u_2,u_3,u_4] / I$, as it is done in \cite{Giv95}, we have the identifications $u_1 = u_2 = v_1, u_3 = u_4 = v_2$, hence
$$
\mathcal{J}_{\K}^*(\nu) = \langle v_1^{2m_1}v_2^{2m_2} \  | \  am_1 + bm_2 \geq \nu \rangle.
$$
This module is different from the whole ring $\Hcal_{\K}^*(\nu) \simeq \C[v_1,v_2,v_1^{-1},v_2^{-1}]$, and always admits elements of minimal degree, in the sense of Proposition \ref{prop1.2.3}: there exists $q \in \mathcal{R}$ such that $q \notin \mathcal{J}_{\K}^*(\nu)$, but $u_i q \in \mathcal{J}_{\K}^*(\nu)$, for all $i=1,...,n$. (see figure \ref{fig4}).
\end{remarksec}

\begin{figure}[H]
\begin{center}
\def\svgwidth{0.6\textwidth}
%% Creator: Inkscape inkscape 0.92.4, www.inkscape.org
%% PDF/EPS/PS + LaTeX output extension by Johan Engelen, 2010
%% Accompanies image file 'fig4_1.pdf' (pdf, eps, ps)
%%
%% To include the image in your LaTeX document, write
%%   \input{<filename>.pdf_tex}
%%  instead of
%%   \includegraphics{<filename>.pdf}
%% To scale the image, write
%%   \def\svgwidth{<desired width>}
%%   \input{<filename>.pdf_tex}
%%  instead of
%%   \includegraphics[width=<desired width>]{<filename>.pdf}
%%
%% Images with a different path to the parent latex file can
%% be accessed with the `import' package (which may need to be
%% installed) using
%%   \usepackage{import}
%% in the preamble, and then including the image with
%%   \import{<path to file>}{<filename>.pdf_tex}
%% Alternatively, one can specify
%%   \graphicspath{{<path to file>/}}
%% 
%% For more information, please see info/svg-inkscape on CTAN:
%%   http://tug.ctan.org/tex-archive/info/svg-inkscape
%%
\begingroup%
  \makeatletter%
  \providecommand\color[2][]{%
    \errmessage{(Inkscape) Color is used for the text in Inkscape, but the package 'color.sty' is not loaded}%
    \renewcommand\color[2][]{}%
  }%
  \providecommand\transparent[1]{%
    \errmessage{(Inkscape) Transparency is used (non-zero) for the text in Inkscape, but the package 'transparent.sty' is not loaded}%
    \renewcommand\transparent[1]{}%
  }%
  \providecommand\rotatebox[2]{#2}%
  \newcommand*\fsize{\dimexpr\f@size pt\relax}%
  \newcommand*\lineheight[1]{\fontsize{\fsize}{#1\fsize}\selectfont}%
  \ifx\svgwidth\undefined%
    \setlength{\unitlength}{295.78045702bp}%
    \ifx\svgscale\undefined%
      \relax%
    \else%
      \setlength{\unitlength}{\unitlength * \real{\svgscale}}%
    \fi%
  \else%
    \setlength{\unitlength}{\svgwidth}%
  \fi%
  \global\let\svgwidth\undefined%
  \global\let\svgscale\undefined%
  \makeatother%
  \begin{picture}(1,0.72457639)%
    \lineheight{1}%
    \setlength\tabcolsep{0pt}%
    \put(0,0){\includegraphics[width=\unitlength,page=1]{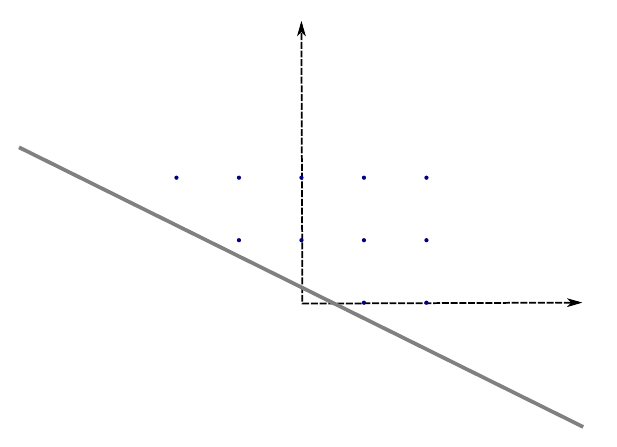}}%
    \put(0.96433276,0.01913902){\color[rgb]{0.50196078,0.50196078,0.50196078}\makebox(0,0)[lt]{\lineheight{1.25}\smash{\begin{tabular}[t]{l}$p^{-1}(1)$\end{tabular}}}}%
    \put(0.73532853,0.52206923){\color[rgb]{0,0,0}\makebox(0,0)[lt]{\lineheight{1.25}\smash{\begin{tabular}[t]{l}$\mathfrak{k}$\end{tabular}}}}%
    \put(0.46610182,0.19052017){\color[rgb]{0,0,0}\makebox(0,0)[lt]{\lineheight{1.25}\smash{\begin{tabular}[t]{l}$0$\end{tabular}}}}%
    \put(0,0){\includegraphics[width=\unitlength,page=2]{fig4_1.pdf}}%
    \put(0.26516944,0.40544084){\color[rgb]{0,0,0}\makebox(0,0)[lt]{\lineheight{1.25}\smash{\begin{tabular}[t]{l}$u_2$\end{tabular}}}}%
    \put(0.34631084,0.33444242){\color[rgb]{0,0,0}\makebox(0,0)[lt]{\lineheight{1.25}\smash{\begin{tabular}[t]{l}$u_1$\end{tabular}}}}%
  \end{picture}%
\endgroup%

\caption[The kernel of a $\K$-equivariant augmentation map]{The kernel of a $\K$-equivariant augmentation map. For the toric manifold $(\C P^1 \times \C P^1, \omega_{1,2}, \mathbb{T}^2 / S^1 \times \mathbb{T}^2 / S^1)$ ($\K = \mathbb{T}^2$), we show the exponents of the monomials generating $\mathcal{J}_{\K}^*(1)$ - in blue -, and of elements $q \in \mathcal{R}$ such that $q \notin \mathcal{J}_{\K}^*(1)$, but $u_i q \in \mathcal{J}_{\K}^*(1)$ for all $i=1,2$ - in red.\looseness=-1}
\label{fig4}
\end{center}
\end{figure}

In this section, we prove that a statement similar to \cite[Corollary $1.3$]{Giv95} holds for $\mathcal{J}_{\K_0}^*(F^-(\nu))$, provided that the symplectic toric manifold $(M,\omega, \mathbb{T})$ is \textit{monotone}. To that aim, we start by bounding $\mathcal{J}_{\K_0}^*(F^-(\nu))$ from below and above with the projections to $\mathcal{R}_0$ of two modules of a form similar to that from equation $(\ref{eq10})$. This simplifies the analysis of $\mathcal{J}_{\K_0}^*(F^-(\nu))$ to that of much explicit modules. The proof of Proposition \ref{prop1.2.3} is then a careful analysis of the degrees of the polynomials of these submodules. 

Assume that $(M,\omega)$ is \textit{monotone}. Recall that through the isomorphism $H^2(M;\R) \simeq \liek^*$ of equation $(\ref{eq1})$, the regular value $p$ represents the cohomology class of $\omega$. Suppose that $p$ is \textit{integral} and \textit{primitive}, and let $N_M$ denote the minimal Chern number of $(M, \omega)$. We have 
$$
p = \frac{c_1}{N_M}.
$$
Let $(V,\xi := \ker \alpha)$ be the prequantization space over $(M,\omega)$ constructed in section \ref{sec3.1.3}, $h$ be a contact Hamiltonian of $V$, and $\overline{h}$ a lift of $h$ to the sphere $S_p$ from equation $(\ref{eq8})$.
Let $c_-, c_+ \in \R$ be two constants such that for any $(z, t) \in S_p \times [0,1]$, we have
$$
 \quad \text{and} \quad c_- < \frac{\underset{z \in S_p}{\min} |z|^2}{2 \underset{z \in S_p}{\max} |z|^2} \overline{h}_t(z) \quad \text{and} \quad c_+ > \frac{2 \underset{z \in S_p}{\max} |z|^2}{\underset{z \in S_p}{\min} |z|^2} \overline{h}_t(z).
$$
We will denote by $H_-$ and $H_+$ the quadratic Hamiltonians on $\C^n$ generating respectively the Hamiltonian symplectomorphisms 
$$
\Phi_{H_-}(z) = \exp(c_-\frac{b}{p(b)})z \quad \text{and} \quad \Phi_{H_+}(z) = \exp(c_+\frac{b}{p(b)})z.
$$
Provided that $N_1$ is big enough, we wish to compare the generating families $F_{-, N}$, $F_N$, $F_{+, N}$ associated respectively with the decompositions 
$$
\begin{array}{ccr}
\Phi_{H_-} = \underbrace{\exp(c_-\frac{b}{2 N_1 p(b)}) \circ \cdots \circ \exp(c_-\frac{b}{2 N_1 p(b)})}_{2N_1-\text{times}}, \\\\
\Phi_{\widetilde{H}} = \Phi_{2N_1} \circ \cdots \circ \Phi_1, \\\\ \text{and} \\\\
\Phi_{H_+} = \underbrace{\exp(c_+\frac{b}{2 N_1 p(b)}) \circ \cdots \circ \exp(c_+\frac{b}{2 N_1 p(b)})}_{2N_1-\text{times}},
\end{array}
$$
where $\widetilde{H}$ is the Hamiltonian lift of $h$ obtained by extending $\overline{h}$ to $\C^n$. Recall also that we have denoted by $S_N := S^{4nN-1}$ the unit sphere in $\C^{2nN}$.
\begin{propsec}
There exists $N_1$ such that for any $N > N_1$, and any $(x, \lambda) \in S_N \times \Lambda_N$, we have
$$
F_{-, \lambda}^{(N)} \leq F_{\lambda}^{(N)} \leq F_{+, \lambda}^{(N)},
$$
where $F_{-, \lambda}^{(N)} := F_{-, N}(., \lambda)$, $F_{\lambda}^{(N)} := F_N(.,\lambda)$, and $F_{+, \lambda}^{(N)} := F_{N, +}(., \lambda)$ denote respectively the generating functions associated with the decompositions of $\Phi_{H_-}$, $\Phi_{\widetilde{H}}$, and $\Phi_{H_+}$ into $2N_1$ parts.

\begin{proof}
First, recall from Remark \ref{rem3.4.1} that the front of $F_N$, and therefore the cohomology of $\widetilde{H}$ at a generic level $\nu$ is independent of the chosen decomposition of the Hamiltonian symplectomorphism $\Phi_{\widetilde{H}}$ into $2N_1$ parts. Therefore, we shall assume that the decomposition of $\Phi_{\widetilde{H}}$ is of the form
$$
\Phi_{\widetilde{H}} = \underbrace{\Phi_{\widetilde{H}}^{\frac{1}{2N_1}} \circ \cdots \circ \Phi_{\widetilde{H}}^{\frac{1}{2N_1}}}_{2N_1-\text{times}}.
$$
This way, from section \ref{sec3.2.2}, the Hamiltonians associated respectively with $\Phi_{\widetilde{H}}^{\frac{1}{2N_1}}$ and $\exp(c_{\pm} \frac{b}{2N_1p(b)})$ are given, for any $(z, t) \in \C^n \setminus \{0\} \times [0,1]$, by 
$$
\frac{1}{2N_1}\widetilde{H}_{\frac{t}{2N_1}}(z) = \frac{1}{2N_1} \frac{|z|^2}{|pr(z)|^2} \overline{h}_{\frac{t}{2N_1}}(pr(z)) \quad \text{and} \quad \frac{1}{2N_1} H_{\pm, \frac{t}{2N_1}}(z) =  \frac{1}{2N_1} \frac{|z|^2}{|pr(z)|^2} c_{\pm}, 
$$
where $pr : \C^n \setminus \{0\} \to S_p$ is the radial projection to the sphere $S_p$ from equation $(\ref{eq8})$. Moreover, the homogeneous generating functions $\Hcal$ and $\mathcal{T}_{c_{\pm}}$ associated respectively with the Hamiltonian symplectomorphisms $\Phi_{\widetilde{H}}^{\frac{1}{2N_1}}$ and $\Phi_{H_{\pm}}^{\frac{1}{2N_1}} = \exp(c_{\pm} \frac{b}{2N_1p(b)})$ are independent of $j = 1, \hdots, 2N_1$, and therefore we can write
$$
\begin{array}{ccr}
\F_{\lambda}^{(N)}(x) = \mathcal{Q}(x) - \underset{j=1}{\overset{2N_1} \sum} \mathcal{H} (x_j) - \underset{j=2N_1+1}{\overset{2N} \sum} \mathcal{T}_{\lambda}(x_j), \\\\
\F_{\pm, \lambda}^{(N)}(x) = \mathcal{Q}(x) - \underset{j=1}{\overset{2N_1} \sum} \mathcal{T}_{c_{\pm}} (x_j) - \underset{j=2N_1+1}{\overset{2N} \sum} \mathcal{T}_{\lambda}(x_j),
\end{array}
$$
where $x = (x_1, \hdots, x_{2N})$ with $x_j \in \C^n$, and $\mathcal{T}_{\lambda}$ is the generating function associated with $\exp(\frac{\lambda}{2N_2})$. Therefore, it suffices to show that there exists $N_1$ such that, if $\overline{B(0,1)}$ denotes the closed ball centered at $0$ and of radius $1$ in $\C^n$ (recall that we wish to compare the generating families, rather than the homogeneous generating families, and therefore we restrict the functions to the sphere $S_N$), then for any $j=1, \hdots, 2N_1$, and any $x \in \overline{B(0,1)}$, we have
\begin{equation}{\label{eq12}}
\mathcal{T}_-(x) \geq \mathcal{H}(x) \geq \mathcal{T}_+(x).
\end{equation}
Recall that we have identified $(\overline{\C^n} \times \C^n, - \omega \oplus \omega)$ with $(T^*\C^n, - \droit (p\droit q))$ by means of the linear symplectomorphism $\Psi(z,w) = (\frac{z + w}{2}, i(z -w))$. Through $\Psi$, the Hamiltonians $\widetilde{H}_t$ and $H_{\pm, t}$ become respectively
$$
(0 \oplus \widetilde{H}_t) \circ \Psi^{-1} \quad \text{and} \quad (0 \oplus H_{\pm, t}) \circ \Psi^{-1}, \quad t \in [0,1],
$$
and their asscociated Hamiltonian isotopies $\{\Phi_{\widetilde{H}}^t\}_{t \in [0,1]}$ and $\{\Phi_{H_{\pm}}^t\}_{t \in [0,1]}$ write respectively
$$
\begin{array}{ccr}
\Phi_{(0 \oplus \widetilde{H}) \circ \Psi^{-1}}^t = \Psi \circ \Phi_{\widetilde{H}}^t \circ \Psi^{-1} \\\\
\text{and} \\\\
\Phi_{(0 \oplus H_{\pm}) \circ \Psi^{-1}}^t = \Psi \circ \Phi_{H_{\pm}}^t \circ \Psi^{-1}.
\end{array}
$$
In particular, since $\Psi^{-1}$ sends the zero-section $0_{\C^n}$ to the diagonal in $\overline{\C^n} \times \C^n$, notice that for any $(x, t) \in \C^n \times [0,1]$, we have
$$
\begin{array}{ccr}
(0 \oplus \widetilde{H}_t) \circ \Psi^{-1}(\Phi_{(0 \oplus \widetilde{H}) \circ \Psi^{-1}}^t(x, 0)) = \widetilde{H}_t(\Phi_{\widetilde{H}}^t(x)) \\\\ \text{and} \\\\ 
(0 \oplus H_{\pm, t}) \circ \Psi^{-1}(\Phi_{(0 \oplus H_{\pm}) \circ \Psi^{-1}}^t(x, 0)) = H_{\pm, t}(\Phi_{H_{\pm}}^t(x)).
\end{array}
$$
Now, if $\Hcal_t$ and $\mathcal{T}_{t, \pm}$ denote the homogeneous generating functions associated respectively with $\Phi_{\widetilde{H}}^{\frac{t}{2N_1}}$ and $\Phi_{H_{\pm}}^{\frac{t}{2N_1}}$, the Hamilton-Jacobi equation and the discussion above yield, for any $(x, t) \in \C^n \times [0,1]$,
$$
\frac{\partial}{\partial t} \mathcal{H}_t(x) = - \frac{1}{2N_1}\widetilde{H}_{\frac{t}{2N_1}}(\Phi_{\widetilde{H}}^{\frac{t}{2N_1}}(x)) \quad \text{and} \quad \frac{\partial}{\partial t} \mathcal{T}_{t, \pm}(x) = - \frac{1}{2N_1}H_{\pm, \frac{t}{2N_1}}(\Phi_{H_{\pm}}^{\frac{t}{2N_1}}(x)),
$$
since $\widetilde{H}_t(0) = \mathcal{T}_{t, \pm}(0) = 0$ for all $t \in [0,1]$, and the inequality $(\ref{eq12})$ holds for $x=0$ (note that $\widetilde{H} = \widetilde{H}_1$ and $\mathcal{T}_{\pm} = \mathcal{T}_{1, \pm}$). Assume now that $N_1$ is big enough so that for any $(x, t) \in (\overline{B(0,1)} \setminus \{0\}) \times [0,1]$, we have 
$$
\frac{1}{2} |\Phi_{H_-}^{\frac{t}{2N_1}}(x)|^2 \leq |\Phi_{\widetilde{H}}^{\frac{t}{2N_1}}(x)|^2 \leq 2 |\Phi_{H_+}^{\frac{t}{2N_1}}(x)|^2.
$$
In this case, we have on the one hand:
\begin{align*}
H_{-, \frac{t}{2N_1}}(\Phi_{H_-}^{\frac{t}{2N_1}}(x)) = \frac{|\Phi_{H_-}^{\frac{t}{2N_1}}(x)|^2}{|pr(\Phi_{H_-}^{\frac{t}{2N_1}}(x))|^2} c_- \leq & \frac{|\Phi_{H_-}^{\frac{t}{2N_1}}(x)|^2}{\underset{x \in S_p}\min |x|^2} c_- \\ 
\leq & \frac{|\Phi_{H_-}^{\frac{t}{2N_1}}(x)|^2}{2\underset{x \in S_p}{\max}|x|^2} \overline{h}_{\frac{t}{2N_1}}(\Phi_{\widetilde{H}}^{\frac{t}{2N_1}}(x)) \\
\leq & \frac{|\Phi_{\widetilde{H}}^{\frac{t}{2N_1}}(x)|^2}{\underset{x \in S_p}{\max}|x|^2} \overline{h}_{\frac{t}{2N_1}}(\Phi_{\widetilde{H}}^{\frac{t}{2N_1}}(x)) \\
\leq & \frac{|\Phi_{\widetilde{H}}^{\frac{t}{2N_1}}(x)|^2}{|pr(\Phi_{\widetilde{H}}^{\frac{t}{2N_1}}(x))|^2} \overline{h}_{\frac{t}{2N_1}}(\Phi_{\widetilde{H}}^{\frac{t}{2N_1}}(x))  = \widetilde{H}_{\frac{t}{2N_1}}(\Phi_{\widetilde{H}}^{\frac{t}{2N_1}}(x)),
\end{align*}
and one the other hand:
\begin{align*}
H_{+, \frac{t}{2N_1}}(\Phi_{H_+}^{\frac{t}{2N_1}}(x)) = \frac{|\Phi_{H_+}^{\frac{t}{2N_1}}(x)|^2}{|pr(\Phi_{H_+}^{\frac{t}{2N_1}}(x))|^2} c_+ \geq & \frac{|\Phi_{H_+}^{\frac{t}{2N_1}}(x)|^2}{\underset{x \in S_p}\max |x|^2} c_+ \\ 
\geq & \frac{2 |\Phi_{H_+}^{\frac{t}{2N_1}}(x)|^2}{\underset{x \in S_p}{\min}|x|^2} \overline{h}_{\frac{t}{2N_1}}(\Phi_{\widetilde{H}}^{\frac{t}{2N_1}}(x)) \\
\geq & \frac{|\Phi_{\widetilde{H}}^{\frac{t}{2N_1}}(x)|^2}{\underset{x \in S_p}{\min}|x|^2} \overline{h}_{\frac{t}{2N_1}}(\Phi_{\widetilde{H}}^{\frac{t}{2N_1}}(x)) \\
\geq & \frac{|\Phi_{\widetilde{H}}^{\frac{t}{2N_1}}(x)|^2}{|pr(\Phi_{\widetilde{H}}^{\frac{t}{2N_1}}(x))|^2} \overline{h}_{\frac{t}{2N_1}}(\Phi_{\widetilde{H}}^{\frac{t}{2N_1}}(x))  = \widetilde{H}_{\frac{t}{2N_1}}(\Phi_{\widetilde{H}}^{\frac{t}{2N_1}}(x)).
\end{align*}
As a consequence, the Hamilton-Jacobi equation shows that, for any $(x, t) \in \overline{B(0,1)} \setminus \{0\} \times [0,1]$,
$$
\frac{\partial}{\partial t} \mathcal{T}_{t, -}(x) \geq \frac{\partial}{\partial t} \mathcal{H}_t(x) \geq \frac{\partial}{\partial t} \mathcal{T}_{t, +}(x).
$$
It remains to notice that the homogeneous generating function associated with the identity is constant equal to $0$. Therefore, equation $(\ref{eq12})$ holds, and the proposition follows.
\end{proof}
\end{propsec}
Kipping the same notations as in the above proposition, we conclude that, for any $\nu \in \R$ we have inclusions of pairs
\begin{comment}
_- := \nu + c_- < \nu < \nu_+ := \nu + c_+$, this yields inclusions of pairs 
\end{comment}
$$
(F_{+,N}^-(\nu), \partial F_{+,N}^-(\nu)) \subset (F_N^-(\nu), \partial F_N^-(\nu)) \subset (F_{-,N}^-(\nu), \partial F_{-,N}^-(\nu)),
$$
and thus homomorphisms
$$
H_{\K_0}^*(F_{-,N}^-(\nu), \partial F_{-,N}^-(\nu)) \to H_{\K_0}^*(F_N^-(\nu), \partial F_N^-(\nu)) \to H_{\K_0}^*(F_{+,N}^-(\nu), \partial F_{+,N}^-(\nu)).
$$
Moreover, notice that, up to a reparametrization, for any $\lambda \in \Lambda_N$, the homogeneous generating functions $\F_{\pm,\lambda}^{(N)}$ are nothing else that the homogeneous generating functions $\G_{\lambda + c_{\pm}\frac{b}{p(b)}}$ associated with the decompositions
$$
\underbrace{\exp(\frac{\lambda + c_{\pm}\frac{b}{p(b)}}{2 N}) \circ \cdots \circ \exp(\frac{\lambda + c_{\pm}\frac{b}{p(b)}}{2 N})}_{2N-\text{times}}.
$$
Thus, the above homomorphisms become
$$
H_{\K_0}^*(G_N^-(\nu_-), \partial G_N^-(\nu_-)) \to H_{\K_0}^*(F_N^-(\nu), \partial F_N^-(\nu)) \to H_{\K_0}^*(G_N^-(\nu_+), \partial G_N^-(\nu_+)).
$$
where $\nu_{\pm} := \nu + c_{\pm}$. If $\nu_{\pm}$ and $\nu$ are generic, we obtain homomorphisms in the limit $N \to \infty$:
$$
\Hcal_{\K_0}^*(G^-(\nu_-)) \to \Hcal_{\K_0}^*(F^-(\nu)) \to \Hcal_{\K_0}^*(G^-(\nu_+)),
$$
leading to inclusions of kernels
$$
\mathcal{J}_{\K_0}^*(G^-(\nu_-)) \subset \mathcal{J}_{\K_0}^*(F^-(\nu)) \subset \mathcal{J}_{\K_0}^*(G^-(\nu_+)).
$$
Note that for any $\nu \in \R$, the $\C[u]$-module $J^*(\nu)$ from equation $(\ref{eq10})$ lies between two modules of the form
$$
J_r := \langle u^{\iota(m)} \ |\ m \in \liek_{\Z},\ p(m) \geq r \rangle.
$$
More precisely, one can find $r_- < r_+$ such that we have inclusions
$$
J_{r_+} \subset J^*(\nu_-) \subset J^*(\nu_+) \subset J_{r_-}.
$$
From the isomorphisms
$$
\mathcal{J}_{\K_0}^*(G^-(\nu_-)) \simeq \mathcal{J}_{\K_0}^*(\nu_-) \quad \text{and} \quad \mathcal{J}_{\K_0}^*(G^-(\nu_+)) \simeq \mathcal{J}_{\K_0}^*(\nu_+),
$$
we deduce embeddings
$$
\mathcal{J}_{r_+}^0 \subset \mathcal{J}_{\K_0}^*(F^-(\nu)) \subset \mathcal{J}_{r_-}^0,
$$
where $\mathcal{J}_{r_{\pm}}^0$ denote the images of $J_{r_{\pm}}$ by the quotient map 
$$
\C[u,u^{-1}] \to \C[u,u^{-1}] / I_0 \C[u,u^{-1}] \simeq \mathcal{R}_0.
$$

\noindent \textbf{Proposition \ref{prop1.2.3}.} \textit{There exists $q \in \mathcal{R}_0$, such that $q \notin \mathcal{J}_{\K_0}^*(F^-(\nu))$, but $u_i q \in \mathcal{J}_{\K_0}^*(F^-(\nu))$ for all $i = 1,...,n$.}
\begin{proof}
\textbf{Step $1$}: By the end of our demonstration, we will use a dimensionality result from \cite{Giv95}, which holds for the ideal $\C[u] \cap J_r + I$. In our case, we will deal with the ideal $\C[u] \cap J_r + I_0$, which we now relate to the module $\mathcal{J}_r^0$. Consider the projection 
$$
\text{pr} : \C[u] \subset \C[u,u^{-1}] \to \C[u,u^{-1}] / I_0\C[u,u^{-1}] \simeq \mathcal{R}_0.
$$ 
For $r \in \R$, the preimage $\text{pr}^{-1}(\mathcal{J}_r^0)$ is the intersection of $\C[u]$ with the preimage of $\mathcal{J}_r^0$ by the quotient map $\C[u,u^{-1}] \to \mathcal{R}_0$, which equals $J_r + I_0 \C[u,u^{-1}]$. Therefore, we have
$$
\text{pr}^{-1}(\mathcal{J}_r^0) = \C[u] \cap (J_r + \C[u,u^{-1}]I_0) \supset \C[u] \cap J_r + I_0.
$$

\noindent \textbf{Step $2$.} Here we show that all the polynomials in $J_r$ have minimal degree $rN_M$. For $m \in \liek_{\Z}$, we saw that $p(m) = \frac{1}{N_M} \underset{i=1}{\overset{n} \sum} m_i$, where $\iota(m) = (m_1, \hdots, m_n)$. Then $J_r$ consists of polynomials whose monomials are of the form $u^{\iota(m)+m'}$, where $m \in \liek_{\Z}$ is such that $p(m) \geq r$, and $m' = (m'_1,\hdots ,m'_n) \in \mathbb{Z}^n_{\geq 0}$. In particular, we have
$$
\underset{i=1}{\overset{n} \sum} m_i + \underset{i=1}{\overset{n} \sum} m'_i
\geq \underset{i=1}{\overset{n} \sum} m_i = p(m)N_M \geq rN_M.
$$
Therefore, letting $\C[u,u^{-1}]^{\geq d}$ denote the submodule generated by monomials of total degree at least $d$, we obtain
$$
J_r \subset \C[u,u^{-1}]^{\geq rN_M}.
$$

\noindent \textbf{Step $3$.} Here we show that, for the notion of degree induced on $\mathcal{R}_0$ by that on $\C[u,u^{-1}]$, the elements of the module $\mathcal{J}_r^0$ have minimal degree $rN_M$. The quotient map $\C[u,u^{-1}] \to \C[u,u^{-1}] / I_0 \C[u,u^{-1}]$ is the restriction map from the ring of regular functions on the complex torus $(\C^{\times})^n$ to the ring of regular functions on the intersection $(\liek_0 \otimes \C) \cap (\C^{\times})^n$. If $f$ is a homogeneous regular function of degree $d$ on $(\C^{\times})^n$, then for any $z \in (\C^{\times})^n$, and any $\mu \in \C^{\times}$, we have $f(\mu z) = \mu^d f(z)$. This characterizes entirely the degree of $f$. Moreover, $\C^{\times}$ acts on the ring of regular functions on $(\liek_0 \otimes \C) \cap (\C^{\times})^n$ in the same way, and the restriction is equivariant with respect to this action. This means that $f$ restricts to a regular function on $(\liek_0 \otimes \C) \cap (\C^{\times})^n$ which is of same degree, or equals $0$. Thus, if $\mathcal{R}_0^{\geq d}$ denotes the ring of regular functions of degree at least $d$ on $(\liek_0 \otimes \C) \cap (\C^{\times})^n$, we have
$$
\mathcal{J}_r^0 \subset \mathcal{R}_0^{\geq rN_M}.
$$

\noindent \textbf{Step $4$.} It is clear from the definition of $J_r$ that for any $m \in \liek_{\Z}$, we have $u^{\iota(m)} J_r = J_{r+r_0}$, where $r_0 = p(m)$. In particular
$$
u^{\iota(m)} \mathcal{J}_r^0 = \mathcal{J}_{r+r_0}^0.
$$
Therefore we can "move $\mathcal{J}_r^0$ above a certain minimal degree". This will serve us in Step $5$. Pick any $m \in \liek_{\Z}$ such that $(r_- + r_0)N_M \geq 1$. Then
$$
u^{\iota(m)} \mathcal{J}_{r_-}^0 = \mathcal{J}_{r_- + r_0}^0 \subset \mathcal{R}_0^{\geq (r_-+r_0)N_M} \subset \mathcal{R}_0^{\geq 1}.
$$
In particular $1 \notin u^{\iota(m)} \mathcal{J}_{r_-}^0$, which means that $1 \in \C[u]$ is not mapped to $u^{\iota(m)} \mathcal{J}_{r_-}^0$ by the projection $\text{pr} : \C[u] \to \mathcal{R}_0$, and thus is also not mapped to $u^{\iota(m)} \mathcal{J}_{\K_0}^*(F^-(\nu)) \subset u^{\iota(m)} \mathcal{J}_{r_-}^0$.\\

\noindent \textbf{Step $5$.} Let 
$$
A := \lbrace u^a \in \C[u] \  | \ \text{pr}(u^a) \notin u^{\iota(m)} \mathcal{J}_{\K_0}^*(F^-(\nu)) \rbrace.
$$
In the previous step we saw that $1 = u^0 \in A$, so $A \neq \emptyset$. We claim that the maximal degree $\underset{i=1}{\overset{n} \sum} a_i$ of any element of $A$ is bounded from above. Since $u^{\iota(m)} \mathcal{J}_{r_+}^0 \subset u^{\iota(m)} \mathcal{J}_{\K_0}^*(F^-(\nu))$, we see that $\text{pr}(u^a) \notin u^{\iota(m)} \mathcal{J}_{\K_0}^*(F^-(\nu))$ implies $\text{pr}(u^a) \notin u^{\iota(m)} \mathcal{J}_{r_+}^0 = \mathcal{J}_r^0$, where $r = r_+ + r_0$. Thus
$$
u^a \notin \text{pr}^{-1}(\mathcal{J}_r^0) \supset \C[u] \cap J_r + I_0 \quad \text{hence} \quad u^a \notin \C[u] \cap J_r + I_0.
$$
Therefore, $A$ lies in the complement in $\C[u]$ of the ideal $\C[u] \cap J_r + I_0$. By (the proof of) {\cite[Proposition $1.2$]{Giv95}}, the zero set $Z(\C[u] \cap J_r + I)$ of the ideal $\C[u] \cap J_r + I$ has at most one point, the origin. Since $I_0 \supset I$, we have
$$
Z(\C[u] \cap J_r + I_0) \subset Z(\C[u] \cap J_r + I) \subset \{0\}.
$$
By the Nullstellensatz, this implies that for every $i$, there exists $m_i \geq 0$ such that $u_i^{m_i} \in \C[u] \cap J_r + I_0$, and it is easy to see that every monomial of total degree $ \geq \underset{i=1}{\overset{n} \sum} m_i$ must then also belong to the ideal. The conclusion is that $\C[u] \cap J_r + I_0$ contains all monomials of sufficiently high degree, and as a result the maximal degree of a monomial $u^a \in A$ is bounded from above.\\

\noindent \textbf{Conclusion.} Let $u^a \in A$ have maximal degree. Then $u_i u^a \notin A$ for all $i=1,\hdots ,n$. This means that $u^a \notin u^{\iota(m)} \mathcal{J}_{\K_0}^*(F^-(\nu))$, while $u_i u^a \in u^{\iota(m)} \mathcal{J}_{\K_0}^*(F^-(\nu))$. Therefore $q = u^{a-\iota(m)} \in \mathcal{R}_0 \setminus \mathcal{J}_{\K_0}^*(F^-(\nu))$, but $u_i q \in \mathcal{J}_{\K_0}^*(F^-(\nu))$ for all $i=1,\hdots ,n$, as claimed.
\end{proof}

\bibliographystyle{alpha}
\bibliography{paper}{}

\begin{thebibliography}{GKPS17}

\bibitem[AM13]{AM13}
P.~Albers and W.J. Merry.
\newblock Translated points and {R}abinowitz {F}loer homology.
\newblock {\em Journal of Fixed Point Theory and Applications}, 13(1):201--214,
  Mar 2013.

\bibitem[AN01]{AN01}
V.I. Arnold and S.P. Novikov.
\newblock {\em Dynamical {S}ystems {I}{V}: {S}ymplectic {G}eometry and its
  {A}pplications}, volume~4.
\newblock 01 2001.

\bibitem[Arn65]{Arn65}
V.I. Arnold.
\newblock Sur une propri{\'e}t{\'e} topologique des applications globalement
  canoniques de la m{\'e}canique classique.
\newblock {\em CR Acad. Sci. Paris}, 261(19):3719--3722, 1965.

\bibitem[Aud12]{Aud12}
Mich\`ele Audin.
\newblock {\em Torus {A}ctions on {S}ymplectic {M}anifolds}, volume~93.
\newblock Birkh{\"a}user, 2012.

\bibitem[BGS13]{BGS13}
J.~Br{\"u}ning, V.W. Guillemin, and S.~Sternberg.
\newblock {\em Supersymmetry and {E}quivariant de {R}ham {T}heory}.
\newblock Springer Berlin Heidelberg, 2013.

\bibitem[BK13]{BK13}
P.~Biran and M.~Khanevsky.
\newblock A {F}loer-{G}ysin exact sequence for {L}agrangian submanifolds.
\newblock {\em Commentarii mathematici Helvetici}, 88(4):899 -- 952, 2013.

\bibitem[BO13]{BO13}
F.~Bourgeois and A.~Oancea.
\newblock The {G}ysin exact sequence for ${S}^1$-equivariant symplectic
  homology.
\newblock {\em Journal of Topology and Analysis}, 05, 2013.

\bibitem[BT13]{BT13}
R.~Bott and L.W. Tu.
\newblock {\em Differential Forms in Algebraic Topology}.
\newblock Graduate Texts in Mathematics. Springer New York, 2013.

\bibitem[BV04]{BV04}
S.~Boyd and L.~Vandenberghe.
\newblock {\em Convex Optimization}.
\newblock Cambridge University Press, 2004.

\bibitem[BW58]{BW58}
W.M. Boothby and H.C. Wang.
\newblock On {C}ontact {M}anifolds.
\newblock {\em Annals of Mathematics}, 68(3):721--734, 1958.

\bibitem[BZ15]{BZ15}
M.S. Borman and F.~Zapolsky.
\newblock Quasimorphisms on contactomorphism groups and contact rigidity.
\newblock {\em Geometry \& Topology}, 19(1):365--411, 2015.

\bibitem[Cha84]{Chap84}
Marc Chaperon.
\newblock Une id{\'e}e du type {\guillemotleft}g{\'e}od{\'e}siques
  bris{\'e}es{\guillemotright} pour les syst{\`e}mes hamiltoniens.
\newblock {\em Comptes rendus des s{\'e}ances de l'Acad{\'e}mie des sciences.
  S{\'e}rie 1, Math{\'e}matique}, 298(13):293--296, 1984.

\bibitem[Cor03]{Cor03}
Octave Cornea.
\newblock {\em Lusternik-Schnirelmann Category}.
\newblock Mathematical surveys and monographs. American Mathematical Soc.,
  2003.

\bibitem[Del88]{Del88}
Thomas Delzant.
\newblock Hamiltoniens p{\'e}riodiques et images convexes de l’application
  moment.
\newblock {\em Bull. Soc. Math. France}, 116(3):315--339, 1988.

\bibitem[Flo89]{Flo89a}
Andreas Floer.
\newblock Symplectic fixed points and holomorphic spheres.
\newblock {\em Communications in Mathematical Physics}, 120, 1989.

\bibitem[FO99]{FO99}
K.~Fukaya and K.~Ono.
\newblock Arnold conjecture and {G}romov–{W}itten invariants.
\newblock {\em Topology}, 38, 1999.

\bibitem[Gei08]{Gei08}
Hansj{\"o}rg Geiges.
\newblock {\em An {I}ntroduction to {C}ontact {T}opology}, volume 109.
\newblock Cambridge University Press, 2008.

\bibitem[Giv90]{Giv90}
Alexander Givental.
\newblock Nonlinear generalization of the {M}aslov index.
\newblock {\em Theory of singularities and its applications}, 1:71--103, 1990.

\bibitem[Giv95]{Giv95}
Alexander Givental.
\newblock A symplectic fixed point theorem for toric manifolds.
\newblock In {\em The Floer Memorial Volume}, pages 445--481. Birkh{\"a}user
  Basel, Basel, 1995.

\bibitem[GKPS17]{GKPS17}
G.~Granja, Y.~Karshon, M.~Pabiniak, and S.~Sandon.
\newblock Givental's non-linear {M}aslov index on {L}ens spaces.
\newblock {\em arXiv preprint arXiv:1704.05827}, 2017.

\bibitem[GM88]{GM88}
M.~Goresky and R.~MacPherson.
\newblock Stratified {M}orse {T}heory.
\newblock In {\em Stratified Morse Theory}, pages 3--22. Springer, 1988.

\bibitem[HS95]{HS95}
H.~Hofer and D.A Salamon.
\newblock Floer homology and {N}ovikov rings.
\newblock In {\em The Floer memorial volume}, pages 483--524. Springer, 1995.

\bibitem[LS85]{LS85}
F.~Laudenbach and J.~Sikorav.
\newblock Persistance d'intersection avec la section nulle au cours d'une
  isotopie hamiltonienne dans un fibré cotangent.
\newblock {\em Inventiones mathematicae}, 82, 1985.

\bibitem[LT98]{LT98}
G.~Liu and G.~Tian.
\newblock Floer homology and {A}rnold conjecture.
\newblock {\em Journal of Differential Geometry}, 49(1):1--74, 1998.

\bibitem[MN18]{MN18}
M.~Meiwes and K.~Naef.
\newblock Translated points on hypertight contact manifolds.
\newblock {\em Journal of Topology and Analysis}, 10(02):289--322, 2018.

\bibitem[Mor01]{Mor01}
Shigeyuki Morita.
\newblock {\em Geometry of {D}ifferential {F}orms}.
\newblock Iwanami series in sodern mathematics. American Mathematical Society,
  2001.

\bibitem[MS17]{MS17}
D.~McDuff and D.~Salamon.
\newblock {\em Introduction to {S}ymplectic {T}opology}.
\newblock Oxford Mathematical Monographs. Oxford University Press, 2017.

\bibitem[MU17]{MU17}
W.J. Merry and I.~Uljarevic.
\newblock Maximum principles in symplectic homology.
\newblock {\em Israel Journal of Mathematics}, 2017.

\bibitem[Oh90]{Oh90}
Yong-Geun Oh.
\newblock A symplectic fixed point theorem on ${T}^{2n} \times \mathbb{C}
  {P}^k$.
\newblock {\em Mathematische Zeitschrift}, 203(1):535--552, Jan 1990.

\bibitem[Ono95]{Ono95}
Kaoru Ono.
\newblock On the {A}rnold conjecture for weakly monotone symplectic manifolds.
\newblock {\em Inventiones mathematicae}, 119, 1995.

\bibitem[Per08]{Per08}
Timothy Perutz.
\newblock A symplectic {G}ysin sequence.
\newblock {\em arXiv preprint arXiv:0807.1863}, 2008.

\bibitem[San11]{San11}
Sheila Sandon.
\newblock Contact homology, {C}apacity and {N}on-squeezing in
  $\mathbb{R}^{2n}\times {S}^{1}$ via generating functions.
\newblock {\em Annales de l'Institut Fourier}, 61(1):145--185, 2011.

\bibitem[San13]{San13}
Sheila Sandon.
\newblock A {M}orse estimate for translated points of contactomorphisms of
  spheres and projective spaces.
\newblock {\em Geometriae Dedicata}, 165(1):95--110, 2013.

\bibitem[Sch98]{Sch98}
Matthias Schwarz.
\newblock A quantum cup-length estimate for symplectic fixed points.
\newblock {\em Inventiones mathematicae}, 133(2):353--397, Jul 1998.

\bibitem[She17]{She17}
Egor Shelukhin.
\newblock The {H}ofer norm of a contactomorphism.
\newblock {\em Journal of Symplectic Geometry}, 15(4):1173--1208, 2017.

\bibitem[Th{\'e}95]{The95}
David Th{\'e}ret.
\newblock {\em Utilisation des fonctions g{\'e}n{\'e}ratrices en
  g{\'e}om{\'e}trie symplectique globale}.
\newblock PhD thesis, 1995.

\bibitem[Th{\'e}98]{The98}
David Th{\'e}ret.
\newblock Rotation numbers of {H}amiltonian isotopies in complex projective
  spaces.
\newblock {\em Duke mathematical journal}, 94(1):13--27, 1998.

\bibitem[Vit92]{Vit92}
Claude Viterbo.
\newblock Symplectic topology as the geometry of generating functions.
\newblock {\em Mathematische Annalen}, 292(1):685--710, 1992.

\bibitem[Woo97]{Woo97}
Nicholas~M.J. Woodhouse.
\newblock {\em Geometric Quantization}.
\newblock Oxford mathematical monographs. Clarendon Press, 1997.

\end{thebibliography}

\end{document}